\documentclass[preprint,12pt]{elsarticle}

\usepackage{amssymb}
\usepackage{amsmath}
 \usepackage{amsthm}
  \usepackage{color}

\usepackage{subfigure}

\newtheorem{thm}{Theorem}[section]

\newtheorem{lem}[thm]{Lemma}

\newtheorem{rem}[thm]{Remark}

\newtheorem{asm}[thm]{Assumption}

\def\ba{{\mathbf{a}}}

\def\bE{{\bf E}}

\def\bW{{\bf W}}

\newcommand{\C}{\rm I\kern-.5emC}
\newcommand{\R}{\rm I\kern-.19emR}

\def\3bar{{|\hspace{-.02in}|\hspace{-.02in}|}}

\journal{Elsevier}

\begin{document}
\newcommand{\BX}{{\bf X}}
\newcommand{\cv}{{\cal V}}
\newcommand{\cW}{{\cal W}}
\newcommand{\co}{{\cal O}}

\renewcommand{\theequation}{\thesection.\arabic{equation}}
\def\@eqnnum{{\reset@font\rm (\theequation)}}

%\newtheorem{theorem}{Theorem}[section]
%\newtheorem{corollary}{Corollary}[section]
%\newtheorem{lemma}{Lemma}[section]
%\newtheorem{proposition}{Proposition}[section]
%\newtheorem{conjecture}{Conjecture}[section]
%\newtheorem{remark}{Remark}[section]
%\newtheorem{definition}{Definition}[section]
%\newtheorem{problem}{Problem}[section]

%\def\abstract{
%\advance \rightskip by 10mm
%\advance \leftskip by 10mm
%\vspace{-0.8em}
%\noindent
%\small{\bf Abstract.}
%}
%\def\endabstract{\par\normalsize\rm}

%\def\Xint#1{\mathchoice
%{\XXint\displaystyle\textstyle{#1}}%
%{\XXint\textstyle\scriptstyle{#1}}%
%{\XXint\scriptstyle\scriptscriptstyle{#1}}%
%{\XXint\scriptscriptstyle\scriptscriptstyle{#1}}%
%\!\int}
%\def\XXint#1#2#3{{\setbox0=\hbox{$#1{#2#3}{\int}$}
%\vcenter{\hbox{$#2#3$}}\kern-.5\wd0}}
%\def\ddashint{\Xint=}
%\def\dashint{\Xint-}

%Greek Letters
\def\a{\alpha}
\def\b{\beta}
\def\d{\delta}\def\D{\Delta}
\def\e{\epsilon}
\def\g{\gamma}\def\G{\Gamma}
\def\k{\kappa}
\def\lam{\lambda}\def\Lam{\Lambda}
\renewcommand\o{\omega}\renewcommand\O{\Omega}
\def\s{\sigma}\def\S{\Sigma}
\renewcommand\t{\theta}\def\vt{\vartheta}
\newcommand{\vphi}{\varphi}
\def\z{\zeta}

\newcommand{\tsigma}{\tilde{\s}}
\newcommand{\tbsigma}{\tilde{\bsigma}}
\def\te{\tilde{\e}}
\def\tu{\tilde{u}}

\newcommand{\bchi}{\mbox{\boldmath$\chi$}}
\newcommand{\bdelta}{\mbox{\boldmath$\delta$}}
\newcommand{\bepsilon}{\mbox{\boldmath$\epsilon$}}
\newcommand{\bfeta}{\mbox{\boldmath$\eta$}}
\newcommand{\bgamma}{\mbox{\boldmath$\gamma$}}
\newcommand{\bomega}{\mbox{\boldmath$\omega$}}
\newcommand{\bvphi}{\mbox{\boldmath$\varphi$}}
\newcommand{\bphi}{\mbox{\boldmath$\phi$}}
\newcommand{\bPhi}{\mbox{\boldmath$\Phi$}}
\newcommand{\bpsi}{\mbox{\boldmath$\psi$}}
\newcommand{\bPsi}{\mbox{\boldmath$\Psi$}}
\newcommand{\bsigma}{\mbox{\boldmath$\sigma$}}
\newcommand{\btau}{\mbox{\boldmath$\tau$}}
\newcommand{\bxi}{\mbox{\boldmath$\xi$}}
\newcommand{\brho}{\mbox{\boldmath$\rho$}}
\newcommand{\bbeta}{\mbox{\boldmath$\beta$}}
\newcommand{\bzeta}{\mbox{\boldmath$\zeta$}}

\def\bk{\boldsymbol{\kappa}}
\def\bmu{\boldsymbol\mu}
\def\bxi{\boldsymbol{\xi}}
\def\bz{\boldsymbol{\zeta}}
%%%%%%%%

%English Letters
\def\ba{{\bf a}}
\def\bb{{\bf b}}
\def\bc{{\bf c}}
\def\be{{\bf e}}
\def\bff{{\bf f}}
\def\bg{{\bf g}}
\def\bn{{\bf n}}
\def\bp{{\bf p}}
\def\bq{{\bf q}}
\def\bs{{\bf s}}
\def\bt{{\bf t}}
\def\bu{{\bf u}}
\def\bv{{\bf v}}
\def\bw{{\bf w}}
\def\bx{{\bf x}}
\def\by{{\bf y}}
\def\bzz{{\bf z}}

\def\bD{{\bf D}}
\def\bE{{\bf E}}
\def\bF{{\bf F}}
\def\bH{{\bf H}}
\def\bJ{{\bf J}}
\def\bV{{\bf V}}
\def\bU{{\bf U}}
\def\bW{{\bf W}}
\def\bX{{\bf X}}
\def\bY{{\bf Y}}

\def\cA{{\cal A}}
\def\cC{{\cal C}}
\def\cD{{\cal D}}
\def\cE{{\cal E}}
\def\cF{{\cal F}}
\def\cG{{\cal G}}
\def\cI{{\cal I}}
\def\cJ{{\cal J}}
\def\cK{{\cal K}}
\def\cL{{\cal L}}
\def\cO{{\cal O}}
\def\cP{{\cal P}}
\def\cQ{{\cal Q}}
\def\cR{{\cal R}}
\def\cS{{\cal \Sigma}}
\def\cT{{\cal T}}
\def\cU{{\cal U}}
\def\cV{{\cal V}}

\def\scT{{_\cT}}
\def\sD{{_D}}
\def\sE{{_E}}
\def\sF{{_F}}
\def\sFz{{_{F_z}}}
\def\sK{{_K}}
\def\sI{{_I}}
\def\sb{{_b}}
\def\sN{{_N}}

\def\curl{{{\bf curl} \ }}
\def\rot{{\mbox{rot}\ }}
\def\BPI{{\bf \Pi}}

\def\cth{\cT_h}
\def\ctH{\cT_H}

\def\tJ{\tilde{\J}}

\def\hK{\widehat{K}}
\def\hx{\widehat{x}}
\def\hy{\widehat{y}}
\def\bhv{\widehat{\bv}}
%%%%%%%%%

%math symbols
\def\l{\ell}
\def\bl{\boldsymbol{\ell}}
\def\col{\colon}
\def\f12{\frac12}
\def\dfrac{\displaystyle\frac}
\def\dint{\displaystyle\int}
\def\nab{\nabla}
\def\p{\partial}
\def\sm{\setminus}
\def\dsum{\displaystyle\sum}
\newcommand{\pp}[2]{\frac{\partial {#1}}{\partial {#2}}}
\def\bzero{{\bf 0}}

\def\divv{\nab\cdot}
\def\divx{\nab_x\cdot}
\def\divtx{\nab_{t,x}\cdot}
\def\nabx{\nab_x}

\newcommand{\grad}{\nabla}
\newcommand{\curlt}{{\nabla \times}}
\newcommand{\gperp}{\nabla^{\perp}}
\newcommand{\gradt}{\nabla\cdot}

\def\forallqq{\quad\forall\,}
\def\aph{A^{1/2}}
\def\amh{A^{-1/2}}

\def\osc{{\rm osc \, }}

\def\Im{{\rm Im}}
\newcommand{\tr}{{\rm tr}}
\def\divvr{{\rm div}}
\def\curllr{{\rm curl}}
\def\curll{{\rm curl}}
\def\curl{{\bf curl}}
\newcommand{\bgrad}{{\bf grad}}
\newcommand\diam{\mathrm{diam\,}}
\renewcommand\Im{\mathrm{Im\,}}
\def\Span{\mbox{Span}}
\def\supp{\mbox{supp\,}}
\newcommand{\trace}{{\rm trace}}

\newcommand{\tri}{|\!|\!|}
\newcommand{\ljump}{\lbrack\!\lbrack}
\newcommand{\rjump}{\rbrack\!\rbrack}
%%%%%%%%%%%%%%%%%%%%%%%%%%%%%%%%%%%%%%%%%%%%%%
\newcommand{\bdm}{\begin{displaymath}}
\newcommand{\edm}{\end{displaymath}}
\newcommand{\beq}{\begin{equation}}
\newcommand{\eeq}{\end{equation}}
\newcommand{\beqa}{\begin{eqnarray}}
\newcommand{\eeqa}{\end{eqnarray}}
\newcommand{\beqas}{\begin{eqnarray*}}
\newcommand{\eeqas}{\end{eqnarray*}}
%misc
\newcommand{\ul}{\underline}
\newcommand{\wh}{\widehat}
\newcommand{\la}{\langle}
\newcommand{\ra}{\rangle}

%Spaces
\newcommand{\Lt}{L^2(\Omega)}
\newcommand{\Lts}{L^2(\Omega)^2}
\newcommand{\Ltc}{L^2(\Omega)^3}
\newcommand{\Ho}{H^1(\Omega)}
\newcommand{\Hoh}{H^1(\wh{\Omega})}
\newcommand{\Hoi}{H^1(\Omega_i)}
\newcommand{\Hos}{H^1(\Omega)^2}
\newcommand{\Hoc}{H^1(\Omega)^3}
\newcommand{\Hoch}{H^1(\wh{\Omega})^3}
\newcommand{\Hoci}{H^1(\Omega_i)^3}
\newcommand{\Hoz}{H^1_0(\Omega)}
\newcommand{\Ht}{H^2(\Omega)}
\newcommand{\Hti}{H^2(\Omega_i)}
\newcommand{\Hts}{H^2(\Omega)^2}
\newcommand{\Htc}{H^2(\Omega)^3}
\newcommand{\Htz}{H^0(\Omega)}
\newcommand{\Hh}{H^{1/2}(\Gamma)}
\newcommand{\Hhi}{H^{1/2}(\Gamma_i)}
\newcommand{\Hmh}{H^{-1/2}(\Gamma)}
\newcommand{\Hdiv}{H(\divvr;\,\Omega)}
\newcommand{\Hdivh}{H(\divv;\,\wh \Omega)}
\newcommand{\hcurl}{H(\curl\,A;\,\Omega)}
\newcommand{\Hcurl}{H(\curll\,A;\,\Omega)}
\newcommand{\Hcrl}{H(\curll\,;\,\Omega)}
\newcommand{\hcrl}{H(\curl\,;\,\Omega)}
\newcommand{\Hcrlh}{H(\curll\,;\,\wh\Omega)}
\newcommand{\hcrlh}{H(\curl\,;\,\wh\Omega)}
\newcommand{\Wdiv}{\BW_0(\mbox{\divv}\,;\,\Omega)}
\newcommand{\Wcurl}{\BW_0(\mbox{\curl}\,A;\,\Omega)}
\newcommand{\WcrossV}{\BW \times V}

\begin{frontmatter}



\title{Adaptive Least-Squares Finite Element Methods 
for Linear Transport Equations Based on an H(div) Flux Reformulation}


\author{Qunjie Liu} %\fnref{fn2}
\ead{qjliu2-c@my.cityu.edu.hk}
\address{Department of Mathematics, 
City University of Hong Kong, Hong Kong SAR, China}

\author{Shun Zhang \corref{cor1}}
\ead{shun.zhang@cityu.edu.hk}
\address{Department of Mathematics, 
City University of Hong Kong, Hong Kong SAR, China}

% \fntext[fn2]{}
%\fntext[fn1]{This work was supported in part by Hong Kong Research Grants Council under the GRF Grant Project No. 11305319, CityU.}
\cortext[cor1]{Corresponding author}
%\address{}

\begin{keyword}
least-squares finite element method 
%\sep linear hyperbolic equation 
\sep linear transport equation
\sep error estimate
%\sep a posteriori error estimate
\sep discontinuous solution
\sep overshooting
\sep adaptive LSFEM


\end{keyword}

\begin{abstract}
In this paper, we study the least-squares finite element methods (LSFEM) for the linear hyperbolic transport equations. The linear transport equation naturally allows discontinuous solutions and discontinuous inflow conditions, while the normal component of the flux across the mesh faces needs to be continuous. Traditional LSFEMs using continuous finite element approximations will introduce unnecessary extra error for discontinuous solutions and boundary conditions. In order to separate the continuity requirements, a new flux variable is introduced. With this reformulation, the continuities of the flux and the solution can be handled separately in natural $H(\mbox{div};\Omega)\times L^2(\Omega)$ conforming finite element spaces. Several variants of the methods are developed to handle the inflow boundary condition strongly or weakly. 

With the reformulation, the new LSFEMs can handle discontinuous solutions and boundary conditions much better than the traditional LSFEMs with continuous polynomial approximations. With least-squares functionals as a posteriori error estimators, the adaptive methods can naturally identify error sources including singularity and non-matching discontinuity. For discontinuity aligned mesh, no extra error is introduced. If an $RT_0 \times P_0$ pair is used to approximate the flux and the solution, the new adaptive LSFEMs can  approximate discontinuous solutions with almost no overshooting even when the mesh is not aligned with discontinuity.

%Compared with other non-conforming least squares methods without the flux reformulation, whose lowest order finite element has to be linear, the new LSFEMs can use piecewise constant functions to approximation a discontinuous solution when the mesh is not aligned. This advantage makes the flux-reformulated LSFEMs can handle the non-aligned discontinuity much better than the discontinuous Galerkin method with much smaller overshooting on an adaptive mesh.

Existence and uniqueness of the solutions and a priori and a posteriori error estimates are established for the proposed methods. Extensive numerical tests are performed to show the effectiveness of the methods developed in the paper.
\end{abstract}
\end{frontmatter}

\section{Introduction}\label{intro}
\setcounter{equation}{0}
In this paper, we consider the following linear transport equation 
in the conservative form. It is a scalar linear partial differential equation of  
hyperbolic type, which is also called the linear advection equation:  
\begin{align} \label{transporteqn}
\gradt(\bbeta u)+\gamma u &= f \quad \mbox{in} \,\ \O, \\ \nonumber
                        u &= g \quad \mbox{on} \,\ \Gamma_{-},
\end{align}
with $\bbeta$ an advection field and $\Gamma_-$ the inflow boundary. 
Detailed descriptions of the equation can be found in Section 2.

%When developing a finite element method, it is essential to have the right discrete approximation space corresponding to the variational formulation of choice. In the so-called "conforming" finite element method, we choose that the finite element space to be a finite dimensional subspace of the abstract space that the weak (true) solution of the variational problem belongs to. For some classical functional spaces, their corresponding conforming finite element spaces are well-known. For example, the piecewise discontinuous polynomial space for $L^2(\O)$, the global continuous piecewise polynomial space for $H^1(\O)$, Raviart-Thomas (RT) and Brezzi-Douglas-Marini (BDM) spaces for $H(\divvr;\O)$, and N\'{e}d\'{e}lec edges spaces for $H(\mbox{curl};\O)$. A wrongly chosen space will introduce unnecessary numerical errors on both a priori and a posteriori stages. 

As pointed in almost all partial differential equation books, it is crucial to realize that unlike the elliptic or parabolic equations where the solution is generally smooth, hyperbolic equations commonly have {\bf discontinuous} solutions. Also, it is very common that the boundary conditions can also be discontinuous. When choosing a finite element approximation space and formulation for hyperbolic equations, we should be very careful about the discontinuity of the solution. In \cite{Zhang:19}, it is observed that for a discontinuity aligned mesh, the discontinuous finite element space is the best choice; for a discontinuity non-aligned mesh, in order to reduce the overshooting, the best choice is to use the piecewise constant approximation space combined with adaptive methods, while all linear or higher order continuous or discontinuous finite elements will have non-trivial oscillations.

First, we review some known abstract variational formulations for the linear transport equation. We modified the formulations to the conservative settings here. Let 
$$
W = \{ v \in L^2(\O): \gradt (\bbeta v) \in L^2(\O)\}.
$$

\noindent{(\bf Variational formulation 1)} \cite{DE:12}
Find $u\in W$, such that
\begin{equation} \label{VF1}
(\gradt(\bbeta u),v)+(\gamma u,v)+(\bbeta\cdot\bn u,v)_{\Gamma_-} 
= (f,v) + (\bbeta\cdot\bn g,v)_{\Gamma_-}, \quad \forall v\in L^2(\O).
\end{equation}

\noindent{(\bf Variational formulation 2) (ultra-weak)}  \cite{DHSW:12}
Find $u\in L^2(\O)$, such that
\begin{equation} \label{VF2}
(u,-\bbeta \cdot \nabla v+\gamma v) 
= (f,v) -(g, \bbeta\cdot\bn v)_{\Gamma_-}, \quad \forall v\in Y,
\end{equation}
with $Y =\{v: v\in L^2(\O),\bbeta \cdot \nabla v \in L^2(\O),v|_{\Gamma_+}=0\}$.

\noindent{(\bf Variational formulation 3 (least-squares))}  \cite{CJ:88,Jiang:98,BochevChoi:01,BC:01,DMMO:04,BG:09,BG:16}
Find $u\in W$, such that
\begin{eqnarray} \label{VF3}
&&(\gradt(\bbeta u)+\gamma u,\gradt(\bbeta v)+\gamma v)+(\bbeta\cdot\bn u,v)_{\Gamma_-} 
 \\ \nonumber
 &&\quad \quad \quad = (f,\gradt(\bbeta v)+\gamma v) + (\bbeta\cdot\bn g,v)_{\Gamma_-}, \quad \forall v\in W.
\end{eqnarray}
In formulations 1 and 3, the trial space is $W$, in the formulation 2, the test space is $Y$. It is well-known that if the standard $C^0$ piecewise polynomial space is used as trial and test spaces for the variational formulation 1 \eqref{VF1}, the method frequently does not give reasonable results in contrast to the elliptic and parabolic cases \cite{Johnson:87}.  It is also true for the least-squares variational formulations \eqref{VF3}, for the simplest piecewise constant discontinuity problem, continuous finite element approximations will introduce unnecessary error, since even the simplest piecewise constant solution is not in the approximation space. If an adaptive method is used, the error indicator will always indicate unnecessary big errors for those elements on the discontinuous region, even when the mesh is very fine.  Also, Gibbs phenomena like spurious over-shootings are unavoidable near the discontinuity, see discussions in \cite{Zhang:19}. The reason is simple: continuous finite element space $\subset H^1(\O)\subset W$ is not good for approximating discontinuous functions. 

The method in \cite{DHSW:12} uses $L^2(\O)$ as the trial space, so the standard discontinuous piecewise polynomial space can be used as the discrete trial space, but the test space $Y$ is essentially as complicated as $W$ and needs a very dedicated and complicated construction.

On the other hand, a close look at the space $W$ will find that simple piecewise discontinuous polynomial space is not its subspace since it needs another continuity requirement. For a true solution $u\in W$, the condition of $\gradt(\bbeta u) \in L^2(\O)$ essentially means 
$$
	u \in L^2(\O) \quad \mbox{and}\quad\bbeta u \in H(\divvr;\O).
$$ 
Thus the continuity in the normal direction of $\bbeta u$ needs to be enforced in a strong or weak way. This is probably the reason why continuous finite element spaces are used in \cite{CJ:88,Jiang:98,BochevChoi:01,BC:01,DMMO:04,BG:09,BG:16},  since continuous finite element space is a subspace of $W$. The only problem is that it requires too much continuity: the solution $u$ may not be continuous at all. In the a posteriori error analysis, there is a simple "gold rule": when the numerical solution is exact, the error estimator is zero. For the continuous LSFEM approximation for the problem with a discontinuous solution, the numerical solution will never be exact even when the mesh is aligned with the discontinuity and the solution is as simple as two constants (see our numerical example 7.4).  The a posteriori error estimator will never be zero in this extremely simple case.

Compared to the continuous finite element method, the famous discontinuous Galerkin method is a right method \cite{RH:73,LR:74,BMS:04}. In DG methods, the solution is approximated in piecewise discontinuous polynomial space, while the continuity of the normal component of $\bbeta u$ is enforced weakly. %When the mesh is aligned with discontinuity, this makes perfect sense. But we also need to notice that, when generalized to nonlinear problems, we often do not know the exact location of discontinuity, thus the mesh is not aligned with the discontinuity is an important case we should consider. In this case, even the linear DG method will have the so-called Gibbs phenomena and needs special tricks like limiters and filtering to reduce the effect. The reason is that when the solution is discontinuous inside an element, even a linear function is too much for the approximation. The piecewise constant approximation is the best choice in terms of the approximation.    

%The nonconforming LSFEM in \cite{DHSW:12} and the similar method in \cite{MY:18} 
%are also such methods. The continuity of the normal componenent of the flux $\bbeta u$ is 
%weakly enforced by adding a jump term with $\jump{\bbeta \cdot \bn u}$ into the discrete 
%formulation.

In this paper, we propose new variational formulations with flux reformulation. Introduce the flux $\bsigma = \bbeta u$, then we have a first order system with appropriate boundary conditions:
\beq \label{1stordersys}
\bsigma -\bbeta u =0 \quad \mbox{and}\quad \gradt \bsigma +\gamma u =f. 
\eeq
With the solution $(\bsigma,u) \in H(\divvr;\O)\times L^2(\O)$, in order to develop a variational formulation, we also need the test spaces and their discrete subspaces, and make sure that the discrete problem is well-posed.
% (as least the trial functions and test functions have the same number of unknowns). 
One way to set up a variational problem for a first order system is developing a mixed type of formulation. But the equation here is unusual and non-symmetric, the stability of the mixed formulation is not clear. The other way of developing a numerical method is using a Petrov-Galerkin formulation as in \cite{DHSW:12,DG:10,DG:11}, where special test functions are constructed.  In this paper, we will use the brute-force method by introducing an artificial, externally defined energy-type variational principle, the least-squares variational principle.

Traditionally, new unknowns are introduced in the least-squares finite element method  in order to decrease the order of problem, e.g., changing the problem from a second order equation into a first order system so that the resulting discrete problem can use low order finite elements and has a reasonable condition number.  For the linear transport equation we study here, it seems unnecessary to introduce new unknowns since the problem is already a first order equation. The reason we introduce the new flux $\bsigma$ is that {\bf different continuity requirements can be handled separately}. In \eqref{1stordersys}, the space requirements for the unknowns are two standard spaces: $\bsigma \in H(\divvr;\O)$ and $u\in L^2(\O)$.  Standard Raviart-Thomas $RT_k$ space and the piecewise discontinuous space $P_k$ can be used to approximate them. 

For the inflow boundary condition, since the space for $u$ is now only $L^2$, we enforce it on $\bsigma$. It can be handled strongly or weakly, thus several closely related least-squares finite element methods are developed here. We call the methods LSFEM and LSFEM-B to denote the method that enforces the boundary condition in the space or by a penalty term, separately. Different weights can be chosen to handle the inflow boundary condition weakly, which lead to two different versions of LSFEM-B methods.

The least-squares finite element methods have several attractive properties: the linear system it produced is symmetric positive definite, and it has a natural and sharp a posteriori error estimator that can be used in adaptive mesh refinements. Because the discrete system is naturally SPD, it opens doors for advanced discrete solves like algebraic multigrid \cite{DMMO:04}. 
 
For the a posteriori error estimator for the linear hyperbolic equation, although there are several researches in this direction, the results are less satisfactory compared to the elliptic equations. Normally, only the upper reliability bound is developed, the lower efficiency bound is often not proved \cite{GHM:14} or only proved under a saturation assumption \cite{Burman:09}. In our methods, the least-squares functional is a natural and sharp error indicator. With respect to the least-squares norms, the error indicator is exact with effectivity constant one. It is also the best one can get from a posteriori estimator: the numerical methods minimize the least-squares energy, the error indicators estimate exactly the error in least-squares energy norms and point out the bad approximated elements.

Because of the reformulation, the methods developed in the paper can use the lowest order finite element approximation spaces: $RT_0$ and $P_0$. For discontinuous solutions with unaligned meshes, $P_0$ approximations can reduce the over/under shootings. Combined with the adaptive mesh refinements, we show numerically that over/undershooting effects can be reduced to almost invisible in "the eye-ball norm". This matches the discussions in \cite{Zhang:19} for approximating discontinuous solutions by adaptive continuous or discontinuous  finite elements.

%The computational cost of our method is also comparable to the standard discontinuous Galerkin method. 

Besides the LSFEMs with problematic continuous approximations, the nonconforming LSFEM in \cite{DHSW:12} and the similar method in \cite{MY:18} use discontinuous approximations. The continuity of the normal component of the flux $\bbeta u$ is weakly enforced by adding a jump term into the discrete formulation. Compared with these methods, the first advantage of our method is that we can use $P_0$ approximations while these methods cannot. This makes our method more suitable for approximating discontinuous solutions on a non-aligned mesh. Also, no jump terms on inter-elements faces/edges are needed in our method, which simplifies the implementation. Besides, it is still not very clear what are the right or optimal weight and form of those inter-element jumps, see \cite{DHSW:12,MY:18}. Earlier methods for hyperbolic equations based on minimization principles can be found in \cite{Lavery:88,Lavery:89}.

The paper is organized as follows.  Section 2 describes the model linear hyperbolic transport problem. Based on a flux reformulation, a least-squares variational problem with strong enforced inflow boundary condition is presented in section 3. Corresponding LSFEM is developed in Section 4, a priori and a posteriori error estimates are established. Sections 5 and 6 develop two versions of least-squares variational formulations and corresponding finite element methods with weakly enforced boundary conditions. Section 7 provides numerical results for many test problems.
In Section 8, we make some concluding remarks.
 
\section{Model Linear Hyperbolic Transport Equation}
\setcounter{equation}{0}

Let $\Omega$ be a bounded polyhedral domain in $\Re^d$ with Lipschitz boundary.
We assume the advective velocity field $\bbeta=(\beta_1,\cdots, \beta_d)^{T}$ is a
vector-valued function defined on $\bar{\O}$ with $\bbeta\in [C^1(\overline{\O})]^d$ for simplicity.
We also assume $\gamma \in L^{\infty}(\O)$ satisfying:
$$
\gamma + \dfrac{1}{2}\gradt \bbeta \geq 0.
$$
Note when $\gradt \bbeta =0$ (for example, $\bbeta$ is a constant vector), $\gamma$ can be $0$.

We define the inflow and outflow parts of $\p \O$ in the usual fashion:
\begin{align*}
&\Gamma_{-}=\{x\in\p \O:\bbeta(x)\cdot\bn(x)<0\}=\mbox{inflow boundary}, \\
&\Gamma_{+}=\{x\in\p\O:\bbeta(x)\cdot\bn(x)>0\}=\mbox{outflow boundary},
\end{align*}
where $\bn(x)$ denotes the unit outward normal vector to $\p\O$ at $x\in\p\O$.

\begin{asm} \label{asmbeta}
({\bf Assumptions of $\bbeta$ and $\gamma$})
We assume that one of the following assumptions on the coefficients is true:
\begin{enumerate}[(i)]
\item $0<|\bbeta|<C$.
For every $\hat{\bx} \in \Gamma_-$, let $\bx(r)$ be a streamline of $\bbeta$ with initial condition $\bx(r_0)=\hat{\bx}$. Assume that there exits a transformation to a coordinate system such that the streamlines are lined up with the $r$ coordinates direction and the Jacobian of the transformation is bounded. We also assume that every streamline connects
$\Gamma_-$ and $\Gamma_+$ with a finite length $\ell(\hat{\bx})$ for $\hat{\bx}\in \Gamma_-$. Note that this case includes the case $\bbeta$ is a nonzero constant vector.

\item There exists a positive $\gamma_0$, such that
$$
\gamma + \dfrac{1}{2}\gradt \bbeta \geq \gamma_0 > 0 \quad \mbox{in  } \Omega.
$$
We also assume that the inflow and outflow boundaries are well-separated.

Note that this case does not include an important case that $\bbeta$ is a constant vector and $\gamma =0$.
\end{enumerate}
\end{asm}

%Let $f\in L^2(\O)$. Now
%Consider the following linear transport equation in the conservative form:
%\begin{align} \label{transporteqn}
%\gradt(\bbeta u)+\gamma u &= f \quad \mbox{in} \,\ \O, \\ \nonumber
%                        u &= g \quad \mbox{on} \,\ \Gamma_{-}.
%\end{align}

Define the following trace space
%$$
%L^2(|\bbeta\cdot\bn|;\p\O):= \{
%v \mbox{ is measurable on } \p\O : \int_{\p\O}|\bbeta\cdot\bn| v^2 < \infty\}.
%$$
%and
$$
L^2(|\bbeta\cdot\bn|;\Gamma_-):= \{
v \mbox{ is measurable on } \p\O : \int_{\Gamma_-}|\bbeta\cdot\bn| v^2 < \infty\}.
$$
%Define
%$$
%W = \{ v \in L^2(\O): \gradt (\bbeta v) \in L^2(\O)\}.
%$$
For the inhomogeneous boundary condition $u=g$ on $\Gamma_-$, we assume $g \in L^2(|\bbeta\cdot\bn|;\Gamma_-)$.

\begin{thm}\label{thm_exist} (Existence and uniqueness of the solution of the linear transport equation)
For $g \in L^2(|\bbeta\cdot\bn|;\Gamma_-)$, the linear transport equation (\ref{transporteqn}) has a unique solution
in $W$ assuming Assumption \ref{asmbeta} of $\bbeta$ and $\gamma$ is true.
\end{thm}
The proof of the theorem with the assumption (i) is based on standard ODE theory, and can be founded in \cite{DHSW:12,DMMO:04}. For the case with the assumption (ii), the proof can be founded in \cite{DHSW:12} and Chapter 2 of \cite{DE:12}.

\begin{rem}
In \cite{DMMO:04}, it is showed that the existence and uniqueness still hold
if the requirement of the inflow boundary condition $g$ is relaxed to
$$
\int_{\Gamma_-} g^2 \ell(\bx(s))|\bbeta \cdot\bn|/|\bbeta| ds < \infty,
$$
where $\ell(\bx)$ is the length of the streamline defined by $\bbeta$ connecting the inflow boundary to the outflow boundary.
\end{rem}

\begin{rem}
An equivalent non-conservative reformulation is
\begin{align}
\bbeta\cdot \nabla u+\mu u &= f \quad \mbox{in} \,\ \O, \\ \nonumber
                        u &= g \quad \mbox{on} \,\ \Gamma_{-},
\end{align}
with $\mu = \gamma + \gradt \bbeta$.

All the methods developed in this paper can be applied to this form of equation by
changing it to the conservative formulation.
\end{rem}
\section{Least-Squares Variational Problem Based on Flux Reformulation}
\setcounter{equation}{0}

In this section, a least-squares variational problem based on flux reformulation is introduced.
The boundary condition is strongly enforced in the trial space. The existence and uniqueness of the formulation 
is discussed.
\subsection{Least-squares problem}
Introduce the flux $\bsigma = \bbeta u$, then
$$
\bsigma -\bbeta u =0 \quad \mbox{and}\quad \gradt \bsigma +\gamma u =f. 
$$
And since $\gradt \bsigma =f - \gamma u \in L^2(\O)$, 
the flux $\bsigma \in H(\divvr;\O)$.

The inflow boundary condition $u = g$ on $\Gamma_{-}$ can also be written as
$$
\bsigma \cdot \bn = (\bbeta \cdot \bn) g, \quad \mbox{on} \,\ \Gamma_{-}.
$$
Define the following spaces:
\begin{eqnarray*}
H_{g,-}(\divvr;\O) &:=& \{ \btau \in H(\divvr;\O) : \btau\cdot\bn = (\bbeta \cdot \bn) g \mbox{  on  }\Gamma_-\}, \\
H_{0,-}(\divvr;\O) &:=&  \{ \btau \in H(\divvr;\O) : \btau\cdot\bn = 0 \mbox{  on  }\Gamma_-\}.
\end{eqnarray*}
Then the least-squares variational problem is:
Seek solutions $(\bsigma,u) \in H_{g,-}(\divvr;\O) \times L^2(\O)$, such that
\beq \label{LS}
\cJ(\bsigma,u;f,g) = \inf_{(\btau,v) \in H_{g,-}(\divvr;\O) \times L^2(\O)} \cJ(\btau,v;f,g),
\eeq
with the least-squares functional $\cJ$ defined as
\beq
\cJ(\btau,v;f,g) := \|\btau -\bbeta v\|_0^2 + \|\gradt \btau + \gamma v -f\|_0^2, \quad \forall (\btau,v) \in H_{g,-}(\divvr;\O) \times L^2(\O).
\eeq

Its corresponding Euler-Lagrange formulation is: Find $(\bsigma,u) \in H_{g,-}(\divvr;\O) \times L^2(\O)$, such that
\beq \label{LS_EL}
a(\bsigma,u;\btau,v) = (f, \gradt \btau +\gamma v), \quad \forall (\btau,v) \in H_{0,-}(\divvr;\O) \times L^2(\O),
\eeq
where for all $(\btau,v), (\brho,w) \in H(\divvr;\O) \times L^2(\O)$, the bilinear form is defined as 
$$
a(\btau,v;\brho,w) = (\btau-\bbeta v, \brho -\bbeta w) + (\gradt \btau+\gamma v,\gradt \brho + \gamma w).
%, \quad \forall (\btau,v), (\brho,w) \in H(\divvr;\O) \times L^2(\O).
$$

\begin{lem} \label{norm1}
Assuming Assumption \ref{asmbeta} of $\bbeta$ and $\gamma$ is true, the following defines a norm for $(\btau,v) \in H_{0,-}(\divvr;\O)\times L^2(\O)$:
\beq
\tri (\btau,v)\tri := \left(\|\btau-\bbeta v\|_0^2 + \|\gradt\btau+\gamma v\|_0^2\right )^{1/2} = a(\btau,v;\btau,v)^{1/2}.
\eeq
\end{lem}
\begin{proof}
The linearity and the triangle inequality are obvious for $\tri (\btau,v)\tri$. 
Now if $\tri (\btau,v)\tri=0$, it follows
$$
\btau=\bbeta v \quad\mbox{and}\quad \gradt \btau+\gamma v=0.
$$
Thus, $\gradt(\bbeta v)+\gamma v=0$.
From the facts $\btau=\bbeta v$ and $\btau\cdot\bn = 0$ on $\Gamma_-$, we get
$\bbeta\cdot\bn v = 0$ on $\Gamma_-$. 
Since $\bbeta\cdot\bn \neq 0$ on $\Gamma_-$, 
$v =0$ on $\Gamma_-$. By Theorem \ref{thm_exist}, $v=0$ is the only solution, thus $\btau =0$.
The norm $\tri \cdot\tri$ is well defined.
\end{proof}
\begin{rem}
It is also clear that 
$$
\tri (\btau,v)\tri_K := \left(\|\btau-\bbeta v\|_{0,K}^2 + \|\gradt\btau+\gamma v\|_{K,0}^2\right )^{1/2}$$
is a semi-norm on an element $K\in\cT$.
\end{rem}

Now, we show the existence and uniqueness of solutions of the least-squares problem by an indirect proof.
\begin{thm} \label{EU1}
The least-squares problem (\ref{LS}) has a unique solution $(\bsigma,u) \in H_{g,-}(\divvr;\O) \times L^2(\O)$ with $g\in L^2(|\bbeta\cdot\bn|;\Gamma_-)$ and that Assumption \ref{asmbeta} is true.
\end{thm}
\begin{proof}
For the existence, with the assumption of $g\in L^2(|\bbeta\cdot\bn|;\Gamma_-)$, by the existence Theorem \ref{thm_exist}, there exists a $u_g\in W\subset L^2({\O})$, such that $u_g = g$ on $\Gamma_-$ satisfying  (\ref{transporteqn}). Let
$\bsigma_g = \bbeta u_g$, then
$$
\|\bsigma_g\|_0 \leq \|\bbeta\|_{\infty} \|u_g\|_0, \quad
\|\gradt \bsigma_g \|_0 = \|f-\gamma u_g\|_0 \leq \|f\|_0 + \|\gamma\|_{\infty}\|u_g\|_{0}.
$$
Also, on the inflow boundary, $\bsigma_g\cdot\bn = \bbeta\cdot\bn u_g= (\bbeta\cdot\bn)g$.
Thus $\bsigma_g \in H_{g,-}(\divvr;\O)$. That is, the minimization problem has a minimizer
$(\bsigma_g,u_g)\in H_{g,-}(\divvr;\O)\times L^2(\O)$ with
$
\cJ(\bsigma_g,u_g;f,g) =0.
$

For the proof of uniqueness, let $(\bsigma_1,u_1)\in H_{g,-}(\divvr;\O)\times L^2(\O)$
and $(\bsigma_2,u_2)\in H_{g,-}(\divvr;\O)\times L^2(\O)$ be two solutions of (\ref{LS}) or (\ref{LS_EL}),
and let
$$
E = \bsigma_1 -\bsigma_2 \quad\mbox{and}\quad \quad e=u_1-u_2.
$$
It follows that
$$
a(E,e;E,e)=  a(\bsigma_1,u_1;E,e)-a(\bsigma_2,u_2;E,e)=(f, \gradt E +\gamma e)-(f, \gradt E +\gamma e)=0.
$$
So $\tri (E,e) \tri =0$, thus $E=0$ and $e=0$. The uniqueness is then proved.
\end{proof}
\begin{rem}
From the proofs of the above lemma and theorem, we can even further reduce the requirements of $\bbeta$ and $\gamma$, as long as they ensure the existence and uniqueness of the solution.

This trick of showing the existence and uniqueness of least-squares method is useful when a norm-equivalence is impossible to prove or the existence and uniqueness theory comes from different techniques. Another example of such least-squares method is its application in non-divergence equation, see \cite{QZ:19}.
\end{rem}

\section{Least-Squares Finite Element Method Based on Flux Reformulation}
\setcounter{equation}{0}
In this section, we develop a LSFEM based on 
the least-squares variational problem developed in the previous section 
and derive the a priori and a posteriori error estimates.

\subsection{Least-squares finite element method}
Let $\cT = \{K\}$ be a triangulation of $\O$ using simplicial elements.
The mesh $\cT$ is assumed to be regular. 
Also, we denote the set of edges/faces of the triangulation $\cT$ on inflow boundary $\Gamma_-$
by $\cE_{-}$.
For an element $K\in \cT$ and integer $k\geq 0$, let
$P_k(K)$ be the space of polynomials with degrees less than or equal to $k$.
Define the finite element spaces $RT_k$ and $P_k$ as follows:
$$
RT_k  :=\{\btau \in H(\divvr; \O) \colon
			\btau|_K \in P_k(K)^d +\bx P_k(K),\,\, \forall \, K\in \cT\},
$$
and
$$			
P_k :=	\{v \in L^2(\O) \colon
			v|_K \in P_k(K),\,\, \forall \, K\in \cT\}.	
$$

\begin{asm} \label{asmbdry}
({\bf Assumption on the boundary data})
For simplicity, we assume $(\bbeta\cdot\bn)g$ on $\Gamma_-$ can be approximated exactly by the trace of $RT_k$ space
on $\Gamma_-$, i.e., $g|_F\in P_k(F)$, for all faces/edges $F\in \cE_{-}$. 
\end{asm}
Note that this assumption still allows the discontinuous boundary condition, but it does require that the boundary mesh is aligned with the discontinuity. For an arbitrary $g$, we need to first interpolate or project $(\bbeta\cdot\bn)g$ to the piecewise polynomial space. 

Define
$$
RT_{k,g,-}:=\{\btau \in RT_k \colon
			\btau\cdot\bn = (\bbeta\cdot\bn)g \mbox{  on  } \Gamma_-\},
$$
then our discrete LSFEM problem is:

\noindent ({\bf LSFEM Problem})
We seek solutions $(\bsigma_h,u_h) \in RT_{k,g,-} \times P_k$,
such that
\beq \label{LSFEM}
\cJ(\bsigma_h,u_h;f,g) = \inf_{(\btau,v) \in RT_{k,g,-} \times P_k} \cJ(\btau,v;f,g).
\eeq
Or equivalently, find $(\bsigma_h,u_h) \in RT_{k,g,-} \times P_k$, such that
\beq \label{LSFEM_EL}
a(\bsigma_h,u_h;\btau,v) = (f, \gradt \btau +\gamma v), \quad \forall (\btau,v) \in RT_{k,0,-} \times P_k.
\eeq
%Now we prove the coercivity of the bilinear form $a$ on the discrete level.
%\begin{lem}
%There exists a constant $C_{cor}$ independent of the mesh size $h$ such that
%\beq
%a(\btau_h,v_h;\btau_h,v_h) \geq C_{cor} (\|\btau_h\|_{H(\divvr;\O)}+\|v_h\|_0)^2, \forall (\btau_h,v_h)\in RT_{k,0,-} \times P_k.
%\eeq
%\end{lem}
%\begin{proof}
%The lefthand side is $\tri (\btau_h,v_h) \tri^2$. By the equivalence of norms on discrete space, we get the righthand side is bounded by the lefthand side. By the standard scaling argument, $C_{cor}$ is independent of $h$.
%\end{proof}
\subsection{Interpolations and their properties}
In order to derive  a priori error estimates, we introduce some interpolations 
and their properties. Note that all properties here are local.

Denote by 
$\pi_k:  L^2 (\O) \mapsto P_k$ the $L^2$-projection onto $P_k$, we have: for $v \in H^{s}(K)$, $s>0$,
\beq \label{L2app}
	\|v-\pi_k v\|_{0,K} \leq C h^{\min\{s,k+1\}} |v|_{\min\{s,k+1\},K}, \quad\forall\,\, K\in \cT.
\eeq
For  $s>0$, denote by $I^{rt}_k: \Hdiv \cap [H^s(\O)]^d \mapsto RT_k$ 
the standard $RT$ interpolation operator \cite{BBF:13}. It satisfies 
the following approximation property: for $\btau \in H^{s}(K)^d$, $s>0$,
\beq \label{rti}
  \|\btau - I^{rt}_k \btau\|_{0,K}
  \leq C h_K^{\min\{s,k+1\}} |\btau|_{\min\{s,k+1\},K}, \quad\forall\,\, K\in \cT.
\eeq
(The estimate in (\ref{rti}) is standard for $s\geq 1$ and may be proved by 
the average Taylor series developed in \cite{DuSc:80} and the standard reference 
element technique with Piola transformation for $0<s<1$.) 
The following commutativity property is well-known:
\beq\label{comm}
 \gradt (I^{rt}_{k}\,\btau)=\pi_k\,\gradt\btau, \qquad
 \forallqq\,\btau\in\Hdiv \cap H^s(\O)^d \,\mbox{ with }\, s>0.
\eeq
Thus the following approximation property holds: for $\btau \in H^{s}(K)^d$
and $\gradt\btau \in H^s(K)$, $s>0$, for any $K\in \cT$, 
\beq \label{rti2}
  \|\gradt(\btau - I^{rt}_k \btau)\|_{0,K} 
  = \|\gradt\btau - \pi_k (\gradt\btau)\|_{0,K}
  \leq C h_K^{\min\{s,k+1\}} |\gradt\btau|_{\min\{s,k+1\},K}.%, \quad\forall\,\, K\in \cT.
\eeq
\begin{rem}
We use $\Hdiv \cap [H^s(\O)]^d$ instead of the choice 
$\{\btau\in L^p(\O)^d$  and  $\gradt \btau \in L^2(\O)\}$ for $p>2$ or $W^{1,t}(K)$ for $t>2d/(d+2)$ in  {\em \cite{BBF:13}} because this Hilbert space based version is more suitable for our analysis.
\end{rem}

We also have the following approximation property on edges(2D)/faces(3D) $F$ of $K$: 
for $\btau \in H^{s}(K)^d$ and $\gradt\btau \in H^s(K)$, for any $K\in \cT$, 
\beq \label{rti3}
  \|(\btau - I^{rt}_k \btau)\cdot\bn\|_{0,F} 
  \leq C h_K^{\min\{s,k+1\}-1/2} (|\btau|_{\min\{s,k+1\},K}+h_K^{1/2}|\gradt\btau|_{\min\{s,k+1\},K}).% \quad\forall\,\, K\in \cT.
\eeq
\begin{proof}
The result follows by approximation properties \eqref{rti} and \eqref{rti2} and the following trace inequality: For all $\btau \in \{ \btau\in H(\divvr;K):\btau\cdot\bn \in L^2(F)\}$,
\beq \label{traceineq}
\|\btau\cdot\bn\|_{0,F} \leq C h_K^{-1/2} (\|\btau\|_{0,K} + h_K^{1/2}\|\gradt\btau\|_{0,K}).
%, \quad \forall \btau \in \{ \btau\in H(\divvr;K):\btau\cdot\bn \in L^2(F)\}
\eeq
\end{proof}

\subsection{A priori error estimation}
\begin{thm} \label{thm_bestapp}
(Cea's lemma type of result)
Let $(\bsigma,u)$ be the solution of least-squares variational problem \eqref{LS}, and $(\bsigma_h,u_h)$ be the solution of LSFEM problem \eqref{LSFEM} with Assumption \ref{asmbdry} on the boundary data, 
the following best approximation result holds:
\beq
\tri(\bsigma-\bsigma_h,u-u_h)\tri \leq \inf _{(\btau_h,v_h) \in RT_{k,g,-} \times P_k}
\tri(\bsigma-\btau_h,u-v_h)\tri.
\eeq
\end{thm}
\begin{proof}
Let $(\btau_h, v_h) \in RT_{k,0,-} \times P_k$, the following error equation holds:
$$
a(\bsigma-\bsigma_h,u-u_h; \btau_h,v_h) =0, \quad \forall (\btau_h, v_h) \in RT_{k,0,-} \times P_k.
$$
From the definition of the norm $\tri \cdot \tri$, the error equation, and Cauchy-Schwarz inequality, we have
\begin{align*}
\tri(\bsigma-\bsigma_h,u-u_h)\tri^2 & = a(\bsigma-\bsigma_h,u-u_h; \bsigma-\bsigma_h,u-u_h)  \\
& =a(\bsigma-\bsigma_h,u-u_h; \bsigma-\btau_h,u-v_h) \\
 & \leq \tri(\bsigma-\bsigma_h,u-u_h)\tri \tri(\bsigma-\btau_h,u-v_h)\tri,
\end{align*}
so $\tri(\bsigma-\bsigma_h,u-u_h)\tri \leq  \tri(\bsigma-\btau_h,u-v_h)\tri$. 
Since $(\btau_h,v_h)$ is chosen arbitrarily, the theorem is proved.
\end{proof}

Define the following piecewise function space on the triangulation $\cT$,
 \begin{align*}
H^{s}(\cT)&=\{v\in L^2(\O):v|_K\in H^{s_K}(K) \,\ \forall K\in\cT \}, \\
H^s(\divvr; \cT)&=\{\btau \in (L^2(\O))^d: \btau|_K\in (H^{s_K}(K))^d, \gradt\btau|_K\in H^{s_K}(K)\,\ \forall K\in\cT \},
\end{align*}
where $s$ is a piecewisely defined function with $s|_K=s_K>0$.
\begin{thm} \label{thm_LSFEM_apriori}
Assume the solution $(\bsigma,u)\in  H^s(\divvr; \cT)\times H^{s}(\cT)$, for $s>0$ defined piecewisely,
and $(\bsigma_h,u_h)\in RT_k\times P_k$ is the solution of the LSFEM problem \eqref{LSFEM} 
with Assumption \ref{asmbdry} on the boundary data, 
then there exists a constant $C>0$ independent of the mesh size $h$, such that
\begin{align}
\tri (\bsigma-\bsigma_h,u-u_h) \tri
 \leq C
 \sum_{K\in \cT}h^{od_K}_K \left(\|u\|_{od_K,K}+\|\bsigma\|_{od_K,K}+\|\gradt\bsigma\|_{od_K,K}
 \right),
\end{align}
where $od_K = \min(k+1,s_K)$.
\end{thm}
\begin{proof}
By the triangle inequality, it is easy to see that
$$
\tri (\btau,v) \tri_K \leq \|\btau\|_{0,K} + \|\gradt \btau\|_{0,K} + (\|\bbeta\|_{\infty,K} +\|\gamma\|_{\infty,K})\|v\|_{0,K}, \quad \forall K\in \cT.
$$
Then the theorem follows immediately after Theorem \ref{thm_bestapp}, 
the triangle inequality, 
and the approximation properties \eqref{L2app}, \eqref{rti}, and \eqref{rti2}.
\end{proof}

\begin{rem} \label{remark_apriori}
\begin{enumerate}
\item  Similar as we did in \cite{CHZ:17} for elliptic problems,
the above a priori result is local with respect to reluralities. 
It establishes the "equip-distribution of errors" foundation of adaptive mesh refinement algorithms. 
With different local regularities and different local sizes of 
the solution in respected $s_K$ norms, 
the mesh size $h_K$ can be modified to ensure an almost equal-distribution of the error. 

\item
Assume that $\bbeta$, $\gamma$, and $f$ are sufficiently smooth in an element $K$, if $u|_K \in H^{s_K}(K)$, then
$$
\bsigma|_K=(\bbeta u)|_K\in (H^{s_K}(K))^d \quad\mbox{and}\quad \gradt\bsigma|_K=(f-\gamma u)|_K\in H^{s_K}(K),
$$
so we can safely assume that $\bsigma|_K$ and $\gradt\bsigma|_K$ have the same smoothness under the condition that the data in each element are sufficiently smooth.

\item  For piecewise smooth solutions, the above theorem covers two cases. For the case that the mesh is aligned with discontinuity,
the solution is still smooth in each element $K$ with some $s_K\geq 1$, we can get optimal 
convergence result in least-squares norms with respect to the local regularity $s_K$.

For the more general case that  the finite element mesh is not aligned with discontinuity, 
$u|_K$ belongs to $H^{1/2-\epsilon}(K)$ for those elements $K$ with a passing though discontinuity for some $\epsilon>0$ as pointed out in \cite{DMMO:04}. 
This means that we cannot get order $1$ on those discontinuous elements. % is $1/2-\epsilon$.
Also $RT_0\times P_0$ should be used on those elements since higher order elements will not contribute more but will introduce much more severe overshooting. And it suggests that there will be many mesh refinements along the discontinuity when an adaptive algorithm is used.

\item It is also clear that we should use $RT_k\times P_k$ pair to ensure the same order of approximation. 
For the $BDM_k\times P_k$ or $BDM_{k+1}\times P_k$, the approximation order will not be balanced and suboptimal 
like the mixed case with the non-zero diffusion \cite{Demlow:02}.

\item  For the extreme case that no smoothness is assumed, i.e., 
the exact solutions satisfy $u\in L^2(\O)$ and $\bsigma \in H(\divvr;\O)$ only, 
we can still prove the convergence without an order by the standard density argument. 
Introduce a smooth 
$\bsigma_{\epsilon} \in H_{g,-}(\divvr;\O)\cap C^{\infty}(\O)^d$ and 
a smooth $u_{\epsilon} \in C^{\infty}(\O)$ such that
$\tri (\bsigma-\bsigma_\epsilon, u-u_\epsilon) \tri \leq \epsilon$ 
for an arbitrary small $\epsilon >0$. 
The smooth $(\bsigma_\epsilon,u_\epsilon)$ can be well-approximated with a small $h$. 
Thus we can show 
$$
\tri (\bsigma-\bsigma_h, u-u_h) \tri \longrightarrow 0, \ \mbox{ as } \ h\longrightarrow 0.
$$
This analysis can also be localized element-wisely as above.

\item  In the theorem, the a priori error estimate is derived for the least-squares energy norm $\tri(\cdot,\cdot)\tri$. 
Our numerical test will also disprove the possibility of a coercivity with respect to the standard norm:
$$
\tri (\btau, v) \tri^2 \geq C \left(\|\btau\|_{H(\divvr;\O)}^2 + \|v\|_0^2\right), \quad \forall (\btau,v) \in H_{0,-}(\divvr;\O) \times L^2(\O),
$$
or the weak discrete version with an $h$-independent $C>0$,
$$
\tri (\btau, v) \tri^2 \geq C \left(\|\btau\|_{H(\divvr;\O)}^2 + \|v\|_0^2\right), \quad \forall (\btau,v) \in RT_{k,0,-} \times P_k.
$$
Because if one of such coercivity results hold, 
one can show that the error measured in $H(\divvr;\O)\times L^2(\O)$ norm will 
be optimal for piecewise smooth solutions with discontinuity aligned mesh, which is not the case in our numerical tests 7.4, 7.6 and the Peterson example 7.5. 
\end{enumerate}
\end{rem}

\subsection{A posteriori error estimation}
The least-squares functional can be used 
to define the following fully computable a posteriori  local indicator 
and global error estimator:
$$
\eta_K^2 := \|\bsigma_h-\bbeta u_h\|_{0,K}^2 + \|\gradt \bsigma_h+\gamma u_h -f\|_{0,}^2, 
\quad \forall K\in \cT,
$$
and
$$
\eta^2 := \sum_{K\in\cT}\eta_K^2 =  \|\bsigma_h-\bbeta u_h\|_0^2 + \|\gradt \bsigma_h+\gamma u_h -f\|_0^2.
$$

\begin{thm}
The a posteriori error estimator $\eta$ is exact with respect to the least-squares norm $\tri (\cdot,\cdot)\tri$:
$$
\eta = \tri (\bsigma-\bsigma_h,u-u_h) \tri \quad \mbox{and} \quad \eta_K = \tri (\bsigma-\bsigma_h,u-u_h) \tri_K.
$$
The following local efficiency bound is also true with $C>0$ independent of the mesh size $h$:
$$
C \eta_K \leq \|\bsigma-\bsigma_h\|_{H(\divvr;K)} + \|u-u_h\|_{0,K}, \quad \forall K\in\cT.
$$
\end{thm}
\begin{proof}
Note that the exact solutions satisfy $\bsigma = \bbeta u$ and $f = \gradt \bsigma + \gamma u$, so
\begin{eqnarray*}
\eta^2 &=& 
\|\bsigma_h-\bbeta u_h\|_0^2 + \|\gradt \bsigma_h+\gamma u_h -f\|_0^2\\
&=& \|\bsigma-\bsigma_h+\bbeta (u-u_h)\|_0^2 + \|\gradt (\bsigma-\bsigma_h)+\gamma (u-u_h)\|_0^2\\
&=&\tri (\bsigma-\bsigma_h,u-u_h) \tri^2. 
\end{eqnarray*}
The proof of the local exactness is identical.

With the triangle inequality, the local efficiency bound for the standard norms can be easily proved.
\end{proof}

\begin{rem}
Due to the fact that the least-squares functional norm is not equivalent to the standard $H(\divvr)$-$L^2$ norm, it is 
impossible to get the corresponding reliability result w.r.t. the $H(\divvr)$-$L^2$ norm.
\end{rem}

%%%%%%%

\section{Least-Squares Variational Problems with Boundary Functional}
\setcounter{equation}{0}
In this section, in stead of treating the inflow boundary condition as an essential condition, we develop 
a least-squares method with boundary functional in the free space.

In this section, we assume the inflow boundary condition is not degenerate, 
$$
|\bbeta\cdot\bn| \geq c>0 \mbox{   on   } \Gamma_-.
$$
%Since it is obvious that $|\bbeta\cdot\bn| < |\bbeta|$,
%this assumption is the same as assuming that the weighted norms
%$\|v\|_{|\bbeta\cdot\bn|,\Gamma_-}$, $\|v\|_{1/|\bbeta\cdot\bn|,\Gamma_-}$
%and the standard $L^2$ norm $\|v\|_{0,\Gamma_-}$ are equivalent on $\Gamma_-$,
%where
%$$
%\|v\|_{\omega,\Gamma_-}^2 := \int_{\Gamma_-} \omega v^2  ds, \quad \omega = |\bbeta\cdot\bn| \mbox{  or  } 1/|\bbeta\cdot\bn|.
%$$
\begin{rem}
This assumption is essential to guarantee the optimal convergence rate, see 
the proof of Theorem \ref{theorem_LSFEMB_apriori}.
\end{rem}

Define a weight-dependent inner product and its corresponding norm:
$$
(v,w)_{\omega,\Gamma_-} := \sum_{F\in \cE_{-}} \int_F \dfrac{\omega}{|\bbeta\cdot\bn|} vw dx \quad\mbox{and}\quad
\|v\|_{\omega,\Gamma_-} := (v,v)_{\omega,\Gamma_-}^{1/2}.
$$
We use two choices here:
$$
\omega_1 = 1 \quad \mbox{and} \quad \omega_2 = \a_F h_F.
$$
The following notation is also used to denote the norm on an edge(2D)/face(3D) of an element $K$:
$$
\|v\|_{\omega,F} := \left(\int_F \dfrac{\omega}{|\bbeta\cdot\bn|} v^2 dx\right)^{1/2}.
$$

\begin{rem}
Here, $\a_F>0$ is a big enough but $h$-independent constant to ensure the balance of terms.
In general, $\a_F$ can be chosen depending on $F$ and possibly also depending on the coefficients.  
The constant $\a_F$ comes from the constant 
that appears in the trace inequality \eqref{traceineq}.
%In most cases, it is fine to choose $\a_F$ to be $1$. 
In some extreme cases, we find it is necessary to
choose $\a_F$ to be some constant large enough ($10$ is large enough in our numerical tests) 
to ensure that the boundary condition is not too weakly enforced. 
See detailed discussion in our numerical test 7.6.1.

In this paper, the choice $\a_F = 10$ is suggested and used in numerical tests.

The choice of $\omega_1=1$ does not have the above issues, but the convergence order 
is less optimal near the inflow boundary, see our discussion in the a priori error estimates Theorem \ref{theorem_LSFEMB_apriori}.
\end{rem}
Let
$$
\Sigma := \big\{\btau\in H(\divvr;\O) : \btau \cdot\bn \in L^2(\Gamma_-)  \big\}
$$
with the weight-dependent norm
$$
\|\btau\|_{\Sigma}^2 := \|\btau\|_{0,\O}^2 + \|\gradt\btau\|_{0,\O}^2 + \|\btau\cdot\bn\|_{\omega,\Gamma_-}^2.
$$
Define the following least-squares functional $\cL$ for all $(\btau,v) \in \Sigma \times L^2(\O)$,
\beq
\cL_i(\btau,v;f,g) := \|\btau -\bbeta v\|_0^2 + \|\gradt \btau + \gamma u -f\|_0^2 +\| \btau\cdot \bn - \bbeta\cdot\bn g\|_{\omega_i,\Gamma_-}^2, i =1,2.
\eeq
%We use notations $\cL_1$ and $\cL_2$ to denote the cases of $\omega = \omega_1$ or $\omega_2$, separately.

\noindent({\bf Least-Squares Problems with Boundary Functional})
We seek solutions $(\bsigma,u) \in \Sigma \times L^2(\O)$, such that
\beq \label{LSB}
\cL_i(\bsigma,u;f,g) = \inf_{(\btau,v) \in \Sigma \times L^2(\O)} \cL_i(\btau,v;f,g), \quad i=1,2.
\eeq
Its corresponding Euler-Lagrange formulation is: Find $(\bsigma,u) \in \Sigma \times L^2(\O)$, such that
\beq \label{LS2_EL}
b_i(\bsigma,u;\btau,v) = (f, \gradt \btau +\gamma v)+(\bbeta\cdot\bn g,\btau\cdot\bn)_{\omega_i,\Gamma_-}, \quad \forall (\btau,v) \in \Sigma \times L^2(\O), i=1,2,
\eeq
where, for all  $(\btau,v), (\brho,w) \in \Sigma \times L^2(\O)$, the bilinear form is: 
$$
b_i(\btau,v;\brho,w) := (\btau-\bbeta v, \brho -\bbeta w) + (\gradt \btau+\gamma v,\gradt \brho + \gamma w)
+ (\bsigma\cdot\bn,\btau\cdot\bn)_{\omega_i,\Gamma_-}, i =1,2.
$$
Note that for $F\in \cE_-$, $\bbeta \cdot \bn <0$, so 
$$
(\bbeta\cdot\bn g,\btau\cdot\bn)_{\omega,\Gamma_-} = 
\sum_{F\in \cE_{-}} \int_F \dfrac{\omega}{|\bbeta\cdot\bn|} (\bbeta\cdot\bn g)(\btau\cdot\bn) dx 
= -\sum_{F\in \cE_{-}}\omega \int_F\btau\cdot\bn g dx.
$$

\begin{lem} \label{norm2}
Assuming that the data $\bbeta$ and $\gamma$ satisfy Assumptions \ref{asmbeta},
the following defines a norm for $(\btau,v) \in \Sigma\times L^2(\O)$:
\beq
\tri (\btau,v)\tri_{B} := (\|\btau-\bbeta v\|_0^2 + \|\gradt\btau+\gamma v\|_0^2 + \|\btau\cdot\bn\|_{\omega,\Gamma_-}^2)^{1/2}.
\eeq
\end{lem}
\begin{proof}
The proof of the lemma is almost identical to that of Lemma \ref{norm1} by realizing that
if $\tri (\btau,v)\tri_B=0$, we have
$$
\btau=\bbeta v \quad\mbox{and}\quad \gradt \btau+\gamma v=0 \mbox{ in  } \O, \quad
\btau \cdot\bn = 0 \mbox{  on  } \Gamma_-.
$$
\end{proof}
\begin{rem}
Similarly,
$$
\tri (\btau,v)\tri_{B,K} := \left(\|\btau-\bbeta v\|_{0,K}^2 + \|\gradt\btau+\gamma v\|_{0,K}^2
+ \sum_{F\in\p K \cap \cE_-}\|\btau\cdot\bn\|_{\omega,F}^2\right)^{1/2}
$$
is a semi-norm on an element $K\in\cT$.

Notations  $\tri (\btau,v)\tri_{B,i}$ and $\tri (\btau,v)\tri_{B,i,K}$ with $i=1$ or $2$ are used to denote the (semi-)norms with weights $\omega = \omega_i$, $i= 1$ or $2$.
\end{rem}

\begin{thm}
The least-squares problem (\ref{LSB}) has a unique solution $(\bsigma,u) \in \Sigma \times L^2(\O)$ with
the assumption $g\in L^2(|\bbeta\cdot\bn|;\Gamma_-)$ and the data $\bbeta$ and $\gamma$ satisfying Assumptions \ref{asmbeta}.
\end{thm}
\begin{proof}
The proof of existence and uniqueness is very similar to that of Theorem \ref{EU1} and thus 
we omit it here.
\end{proof}

\section{LSFEMs with Boundary Functional}
\setcounter{equation}{0}
In this section, we develop LSFEMs based on 
the least-squares variational problems with boundary functional 
developed in the previous section 
and derive the a priori and a posteriori error estimates.

\subsection{LSFEM-B problems}
We seek solutions $(\bsigma_h,u_h) \in RT_k \times P_k$,
such that
\beq \label{LSFEMB}
\cL_i(\bsigma_h,u_h;f,g) = \inf_{(\btau,v) \in RT_k \times P_k} \cL_i(\btau,v;f,g),\quad i=1,2.
\eeq
Or equivalently, find $(\bsigma_h,u_h) \in RT_k \times P_k$, such that
\beq \label{LSFEMB_EL}
b_i(\bsigma_h,u_h;\btau,v) = (f, \gradt \btau +\gamma v)+(\bbeta\cdot\bn g,\btau\cdot\bn)_{\omega_i,\Gamma_-}, \quad \forall (\btau,v) \in RT_k \times P_k, i=1,2.
\eeq

\subsection{A priori error estimation}
\begin{thm} \label{thm_bestapp_B}
(Cea's lemma type of result)
Let $(\bsigma,u)$ be the solution of least-squares variational problem 
with boundary term \eqref{LSB}, 
and $(\bsigma_h,u_h)$ be the solution of LSFEM problem \eqref{LSFEMB}, 
the following best approximation result holds:
\beq
\tri (\bsigma-\bsigma_h,u-u_h)\tri_B \leq \inf _{(\btau_h,v_h) \in RT_{k} \times P_k}
\tri (\bsigma-\btau_h,u-v_h)\tri_B
\eeq
\end{thm}
\begin{proof}
The proof is identical to that of Theorem \ref{thm_bestapp}.
\end{proof}

Define the collections of elements with edges(2D)/faces(3D) on the inflow boundary as:
$$
\cT_- = \{ K : K\in \cT, \p K \cap \Gamma_- \neq \emptyset  \}.
$$

\begin{thm} \label{theorem_LSFEMB_apriori}
Assume the exact solution $(\bsigma,u)\in  H^s(\divvr; \cT)\times H^{s}(\cT)$, for $s>0$ defined piecewisely.
Assume $(\bsigma_{h,i},u_{h,i}) \in RT_k\times P_k$ is the solution of LSFEM-B problem \eqref{LSFEMB} with weight $\omega_i$, $i=1$ or $2$, then
there exists a constant $C>0$ independent of the mesh size $h$, such that
\begin{align} \label{apriori_B1}
\tri (\bsigma-\bsigma_{h,1},u-u_{h,1}) \tri_{B,1}
 \leq C
 \sum_{K\in \cT }h^{od_K}_K \left(\|u\|_{od_K,K}+\|\bsigma\|_{od_K,K}+\|\gradt\bsigma\|_{od_K,K}  \right) \\ \nonumber
 +C\sum_{K\in \cT_-}h^{od_K-1/2}_K \|\bsigma\|_{od_K,K},
 \\ \label{apriori_B2}
\tri (\bsigma-\bsigma_{h,2},u-u_{h,2}) \tri_{B,2}
 \leq C
 \sum_{K\in \cT}h^{od_K}_K \left(\|u\|_{od_K,K}+\|\bsigma\|_{od_K,K}+\|\gradt\bsigma\|_{od_K,K}
 \right).
\end{align}
where $od_K = \min(k+1,s_K)$.
\end{thm}
\begin{proof}
We only need to handle the boundary term, the rest of terms are identical to that of 
Theorem \ref{thm_LSFEM_apriori}.

Let $\btau_h = I_k^{rt}\bsigma$,
by the trace inequality \eqref{traceineq} and approximation property \eqref{rti3}, we have
\begin{eqnarray*}
\|(\bsigma -\btau_h)\cdot\bn\|_{0,F} 
&\leq & h_F^{-1/2}(\|\bsigma -\btau_h\|_{0,K}+h_K^{1/2}\|\gradt(\bsigma -\btau_h)\|_{0,K}) \\
&\leq& C h^{od_K-1/2}_K (\|\bsigma\|_{od_K,K}+h_K^{1/2}\|\gradt\bsigma\|_{od_K,K}).
\end{eqnarray*}

By our assumption on $\bbeta\cdot\bn$, there exits a constant  $C>0$ independent of the mesh size $h$,
\begin{eqnarray*}
\|(\bsigma -\btau_h)\cdot\bn\|_{\o_1,\Gamma_-}  &= &\sum_{F\in\cE_-} 
\|(\bsigma-\btau_h)\cdot\bn\|_{\o_1,F}^2 \leq 
C\sum_{F\in\cE_-}\|(\bsigma -\btau_h)\cdot\bn\|_{0,F}^2 \\
&\leq& C \sum_{K \in \cT_-} h^{2 od_K-1}_K \|\bsigma\|_{od_K,K}^2+
 \sum_{K \in \cT_-} h^{2 od_K}_K \|\gradt\bsigma\|_{od_K,K}^2.
\end{eqnarray*}
Combined with interior terms, we proved \eqref{apriori_B1}.

By our assumptions on $\a_F$ and $\bbeta\cdot\bn$, there exits a constant  $C>0$ independent of the mesh size $h$,
\begin{eqnarray*}
\|(\bsigma -\btau_h)\cdot\bn\|_{\omega_2,\Gamma_-}  &= &\sum_{F\in\cE_-} 
\|(\bsigma-\btau_h)\cdot\bn\|_{\o_2,F}^2 \leq 
C\sum_{F\in\cE_-}h_F \|(\bsigma -\btau_h)\cdot\bn\|_{0,F}^2 \\
&\leq& C \sum_{K \in \cT} h^{2 od_K}_K (\|\bsigma\|_{od_K,K}^2+h_K\|\gradt\bsigma\|_{od_K,K}^2).
\end{eqnarray*}
Combined with interior terms, we proved \eqref{apriori_B2}.
\end{proof}
\begin{rem}
For the case the weight $\omega=1$, we see there is a half-order loss in the error analysis 
for those elements in $\cT_-$. Compared with the number of elements in $\cT$, the number 
of elements in $\cT_-$ is small 
and such sub-optimality often is non-observable in our numerical tests. 

For the case the weight $\omega=\omega_2$, even though the convergence order is optimal, we do add an uncertainty of choosing  $\a_F$. A too small $\a_F$ will lead to imbalance of terms and will cause the boundary condition un-resolved, which will make the  adaptive algorithms fail, see our numerical test 7.7.1.

For the case that the mesh is not aligned with the discontinuity, which probably is 
the interesting case, the elements with discontinuity are the major source of the error, and will dominate the inflow half order loss since we can always make sure the mesh on the inflow  boundary condition is aligned. In this case, the simple choice $\omega=1$ is probably the better choice.

The discussions in Remark \ref{remark_apriori} are also true for the methods in this section.
\end{rem}
\subsection{A posteriori error estimation}
The least-squares functional can be used to define the following fully computable a posteriori local indicator and global error estimator:
$$
\xi_K^2 := \|\bsigma_h-\bbeta u_h\|_{0,K}^2 + \|\gradt \bsigma_h+\gamma u_h -f\|_{0,K}^2 + \sum_{F\in \p K \cap \cE_{-}}\|\bsigma_h\cdot\bn-\bbeta\cdot\bn g\|_{\omega,F}^2, \forall K\in \cT,
$$
and
$$
\xi^2 := \sum_{K\in\cT}\xi_K^2 =  \|\bsigma_h-\bbeta u_h\|_0^2 + \|\gradt \bsigma_h+\gamma u_h -f\|_0^2+\|\bsigma_h\cdot\bn-\bbeta\cdot\bn g\|_{\omega,\Gamma_-}^2.
$$

\begin{thm}
The a posteriori error estimator $\eta$ is exact with respect to $\tri (\cdot,\cdot)\tri_B$-norm:
$$
\xi = \tri (\bsigma-\bsigma_h,u-u_h) \tri_B \quad \mbox{and} \quad \xi_K = \tri (\bsigma-\bsigma_h,u-u_h) \tri_{B,K}.
$$
The following local efficiency bounds are also true with a constant $C>0$ independent of the mesh size $h$. For the method and indicators with $\omega = \omega_1=1$,
$$
C \xi_K \leq \|\bsigma-\bsigma_h\|_{H(\divvr;K)} + \|u-u_h\|_{0,K}, \quad \forall K\in\cT\backslash\cT_-,
$$
$$
C \xi_K \leq h_K^{-1/2}\|\bsigma-\bsigma_h\|_{0,K} + \|\gradt(\bsigma-\bsigma_h)\|_{0,K} + \|u-u_h\|_{0,K}, \quad \forall K\in\cT_-,
$$
and for the method and indicators with $\omega = \omega_2$,
$$
C \xi_K \leq \|\bsigma-\bsigma_h\|_{H(\divvr;K)} + \|u-u_h\|_{0,K}, \quad \forall K\in\cT.
$$
\end{thm}
\begin{proof}
The local and global exactness results are trivial as the case without the boundary functional.
For the local efficiency bounds, the result follows from the triangle inequality and the trace inequality \eqref{traceineq} if the element belongs to $\cT_-$.
\end{proof}

\begin{rem} In our LSFEM-B method with weight $\omega_2$, 
the boundary condition of $\bsigma$ is treated 
by mesh size weighting to ensure the optimal convergence order. The more complicated $-1/2$ norm version similar to that in \cite{Sta:06} can also be developed. 
\end{rem}

%%%%%%%%%%%%%%%%%%%%%%%
\section{Computational Examples}
\setcounter{equation}{0}

\subsection{C-LSFEM for comparison}
In our computational examples, we also compare our new flux based LSFEMs with  existing LSFEM with a continuous finite element approximation (C-LSFEM). 

%\subsubsection{C-LSFEM: LSFEM with a continuous finite element approximation}
Define the standard linear continuous finite element space as $S_1 := \{v\in H^1(\O), v|_{K}\in P_1(K), \forall \ K\in \cT \}$, define the abstract spaces as
\begin{eqnarray*}
H^1_{g,-}(\O) &:=& \{v\in H^1(\O), v=g \mbox{  on  } \Gamma_-\}, \\ \quad \mbox{and}\quad
H^1_{0,-}(\O) &:=& \{v\in H^1(\O), v=0 \mbox{  on  } \Gamma_-\}, 
\end{eqnarray*}
and define the corresponding continuous finite element spaces:
$$
V_{g,-} := H^1_{g,-}(\O) \cap S_1 \quad \mbox{and}\quad
V_{0,-} := H^1_{0,-}(\O) \cap S_1.
$$
Here, we assume that $g$ is smooth enough to make the above definitions meaningful and it can be exactly approximated by the linear finite elements on the inflow boundary. For our many numerical examples, $g$ might be discontinuous and we use a smoothed version of it in the computation, specifically, we choose $u_h\in V_{g,-}$, such that $u_h(z) = g(z)$ if $g$ is continuous at a node $z$, and  $u_h(z) = (g(z^-)+g(z^+))/2$ for $g$ with a jump discontinuity at a node $z$. 

For the C-LSFEM, we use the equivalent non-conservative formulation to define the least-squares minimization problem:
\beq
\cL(v;f,g) := \|\bbeta \cdot \nabla v+\mu v- f\|_0^2, \quad v \in H^1_{g,-}(\O).
\eeq
The corresponding finite element problem is \cite{BochevChoi:01,DMMO:04,BG:09,BG:16}: find $u_h\in V_{g,-}$, such that,
$$
(\bbeta \cdot \nabla u_h + \mu u_h, \bbeta \cdot \nabla v_h + \mu v_h) = (f, \bbeta \cdot \nabla v_h+ \mu v_h), \; \forall \, v_h\in V_{0,-}.
$$
We can use the LS functional as the a posteriori error estimator and error indicator:
$$
\zeta = \|\bbeta \cdot \nabla u_h +\mu u_h- f\|_0, \quad \zeta_K =  \|\bbeta \cdot \nabla u_h +\mu u_h- f\|_{0,K}.
$$
Note that, when the true solution is discontinuous, the error estimator will never be zero. And, it will have non-trivial oscillations near discontinuity, as pointed out in \cite{Zhang:19}.

\subsection{Computational setting}
In most of our numerical examples, the lowest order approximations are used, i.e., $P_0$ for $u$ and $RT_0$ for the flux $\bsigma$. We will explicitly state out if $RT_1 \times P_1$ pair is used.

We use the name LSFEM to denote the methods we developed in Section 4,
and use LSFEM-B1 and LSFEM-B2 to denote the methods with weight $\omega_1$ and $\omega_2$ developed in Section 6, separately. 
If not stated explicitly,  $\a_F=10$ is used in LSFEM-B2 in our numerical tests.

In the adaptive mesh refinement algorithm,  the D\"ofler's bulk marking strategy with $\theta =0.5$ is used and  the algorithm is stopped when the total number of  nodes reaches $10^5$. All refinements are based on the longest edge bisection algorithm.

\begin{figure}[!htb]
    \centering
   \begin{minipage}[!hbp]{0.5\linewidth}
        \includegraphics[width=0.99\textwidth,angle=0]{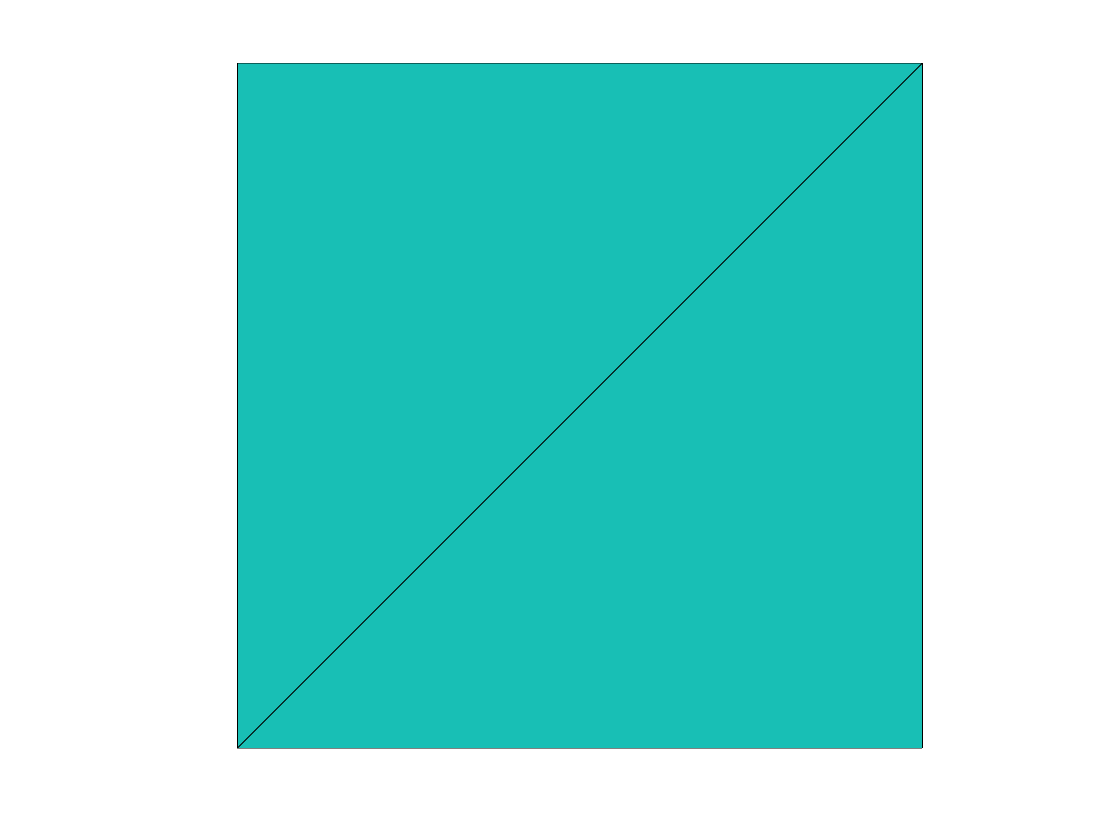}
        \end{minipage}%
        \caption{Initial mesh for all examples with a $(0,1)^2$ domain}%
            \label{initialmesh}
\end{figure}

For all the numerical examples with domain $(0,1)^2$, except for the Peterson problem, the mesh shown in Fig. \ref{initialmesh} is used as an initial mesh.

Although we have three versions of fluxed-based least-squares methods, in our numerical experiments, we find they have almost identical performance. We only show the figures of all three methods in Examples 7.4 and 7.6. For all other test problems, only LSFEM are shown unless stated explicitly.

\subsection{An example with a constant advection field and a piecewise constant solution on a matching grid}
%In this example, we only need do the thought experiment, although the actual computation does confirm our result.

Consider the following problem, $\O = (0,1)^2$ with $\bbeta = (1,1)^T/\sqrt{2}$. The inflow boundary is $\{x=0, y\in (0,1)\} \cup \{x\in (0,1), y=0\}$, i.e., the west and south boundaries 
of the domain. 
%Let region be I $= \{ (x,y) \in (0,1)^2 \colon y>x\}$ and region II $= \{ (x,y) \in (0,1)^2 \colon y<x\}$.
Let $\gamma =1$ and choose $g$ and $f$ such that the exact solution $u$ is
$$
u = \left\{ \begin{array}{lll}
1 & \mbox{in} & y>x, \\[2mm]
0 & \mbox{in} & y<x.
\end{array} \right.
$$
If we choose the mesh aligned with the discontinuity, for example, any refinements of the mesh in Fig. \ref{initialmesh}. Note that the true solutions $u\in P_0$ and $\bsigma \in RT_0$. By the best approximation properties Theorems \ref{thm_bestapp} and \ref{thm_bestapp_B},  the numerical solutions of the flux-based LSFEMs $u_h$ and $\bsigma_h$ are identical to the exact solutions.  So no further refinements are needed.  

This is not true for the C-LSFEM, where $C^0$ finite elements are used to approximate the discontinuous $u$. Many unnecessary refinements are needed, and a mesh aligned with the discontinuity is useless here. Severe overshooting is observed. In Fig. \ref{afem_pwc_matching}, we show a final adaptive mesh, a numerical solution, and convergence histories for the C-LSFEM with $\zeta$ as the a posteriori error estimator. Compared with the natural discontinuous approximations such as the flux-based LSFEMs, the C-LSFEM is a bad choice for such cases.
%
%\begin{figure}[!ht]
%    \label{afem_pwc_matching}
%    \begin{minipage}[!hbp]{0.33\linewidth}
%        \includegraphics[width=0.99\textwidth,angle=0]{mesh_afem_pwc_matching_CLSFEM}
%        \end{minipage}%
%    \begin{minipage}[!hbp]{0.33\linewidth}
%        \includegraphics[width=0.99\textwidth,angle=0]{sol_afem_pwc_matching_CLSFEM}
%        \end{minipage}%
%    \begin{minipage}[!htbp]{0.33\linewidth}
%        \includegraphics[width=0.99\textwidth,angle=0]{error_afem_pwc_matching_CLSFEM}
%        \end{minipage}
%        \caption{Piecewise constant solution on a matching mesh with an adaptive C-LSFEM:  a adaptive mesh (left), numerical solution (center), convergence history (right)}%
%\end{figure}

\begin{figure}[!htb]
\centering 
\subfigure[adaptive mesh]{ 
\includegraphics[width=0.3\linewidth]{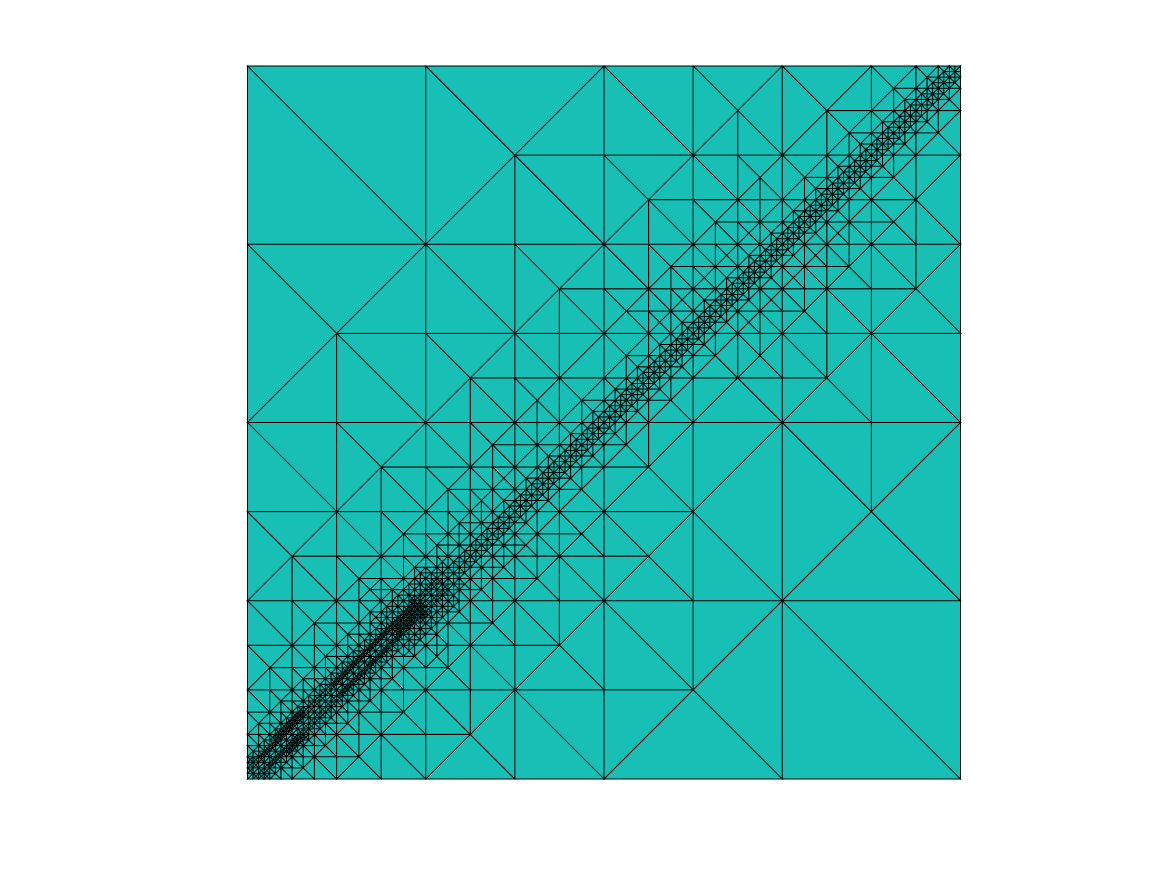}}
%\hspace{0.01\linewidth}
~
\subfigure[numerical solution]{
\includegraphics[width=0.3\linewidth]{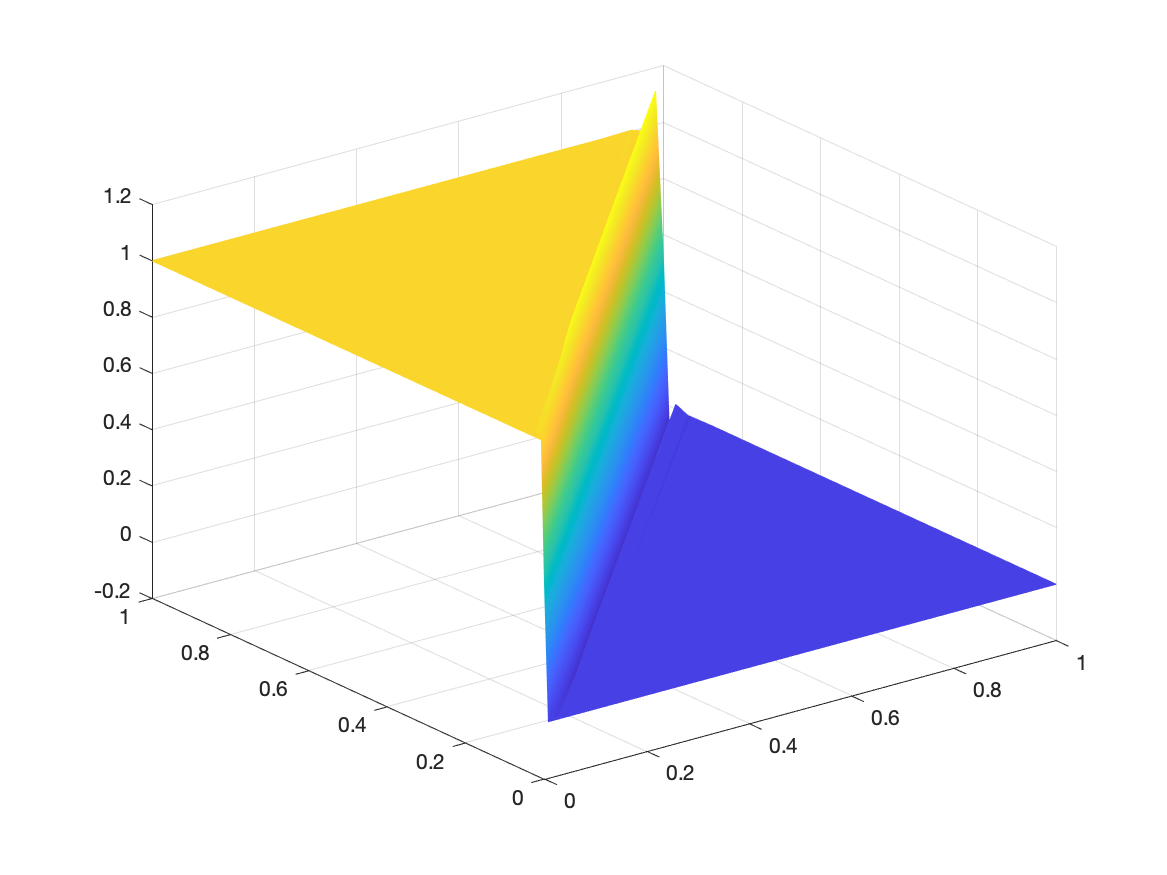}}
~
\subfigure[convergence histories]{
\includegraphics[width=0.3\linewidth]{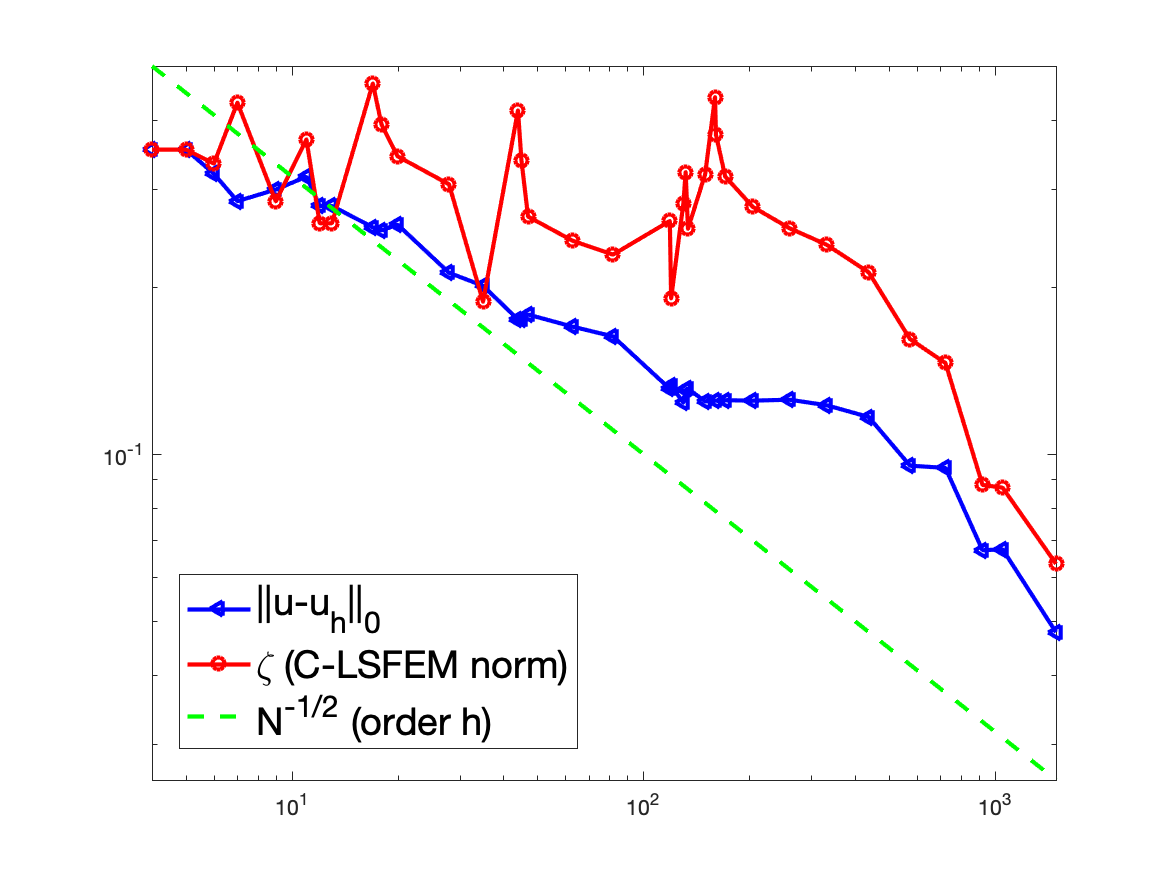}}
\caption{Piecewise constant solution: a matching mesh with the adaptive C-LSFEM}
 \label{afem_pwc_matching}
\end{figure}

\subsection{An example with a global smooth solution}
Consider the following simple problem: $\O = (0,1)^2$ with $\bbeta = (1,1)^T$. The inflow boundary is $\{x=0, y\in (0,1)\} \cup \{x\in (0,1), y=0\}$, i.e., the west and south boundaries of the domain. Let $\gamma=1$. Choose $f$ and $g$ such that the exact solution is $u =\sin(x+y)$.

In Fig. \ref{error_smooth}, the convergence histories of flux-based LSFEMs on uniformly refined meshes are shown. Errors measured in least-squares norms and $\|u-u_h\|_0$ are all of order $1$. The optimal convergence order in $\tri(\cdot,\cdot) \tri_{B,1}$ norm suggests that the half order loss on those inflow boundary elements is neglectable. 

%Similar results should also be true for the LSFEM-B2 formulation.

\begin{figure}[!ht]
\centering 
\subfigure[LSFEM]{ 
\includegraphics[width=0.3\linewidth]{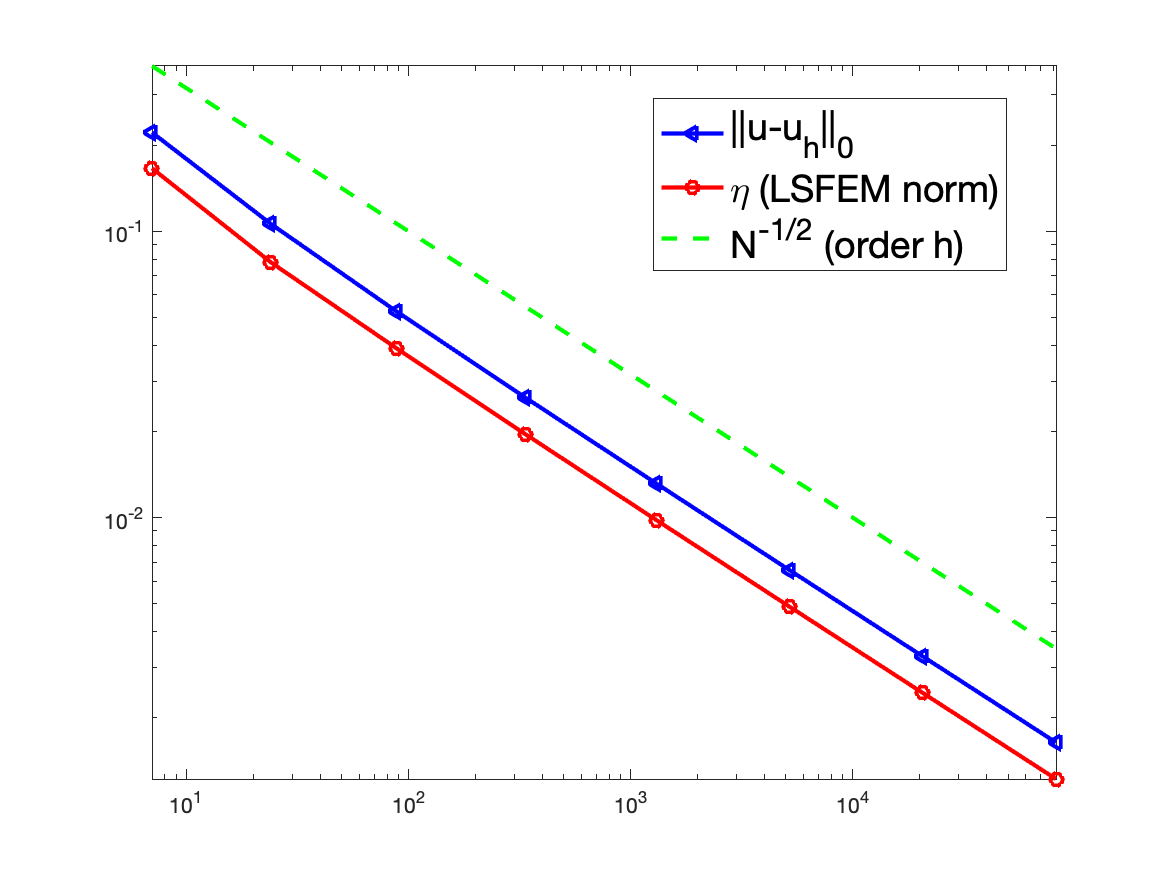}}
%\hspace{0.01\linewidth}
~
\subfigure[LSFEM-B1]{
\includegraphics[width=0.3\linewidth]{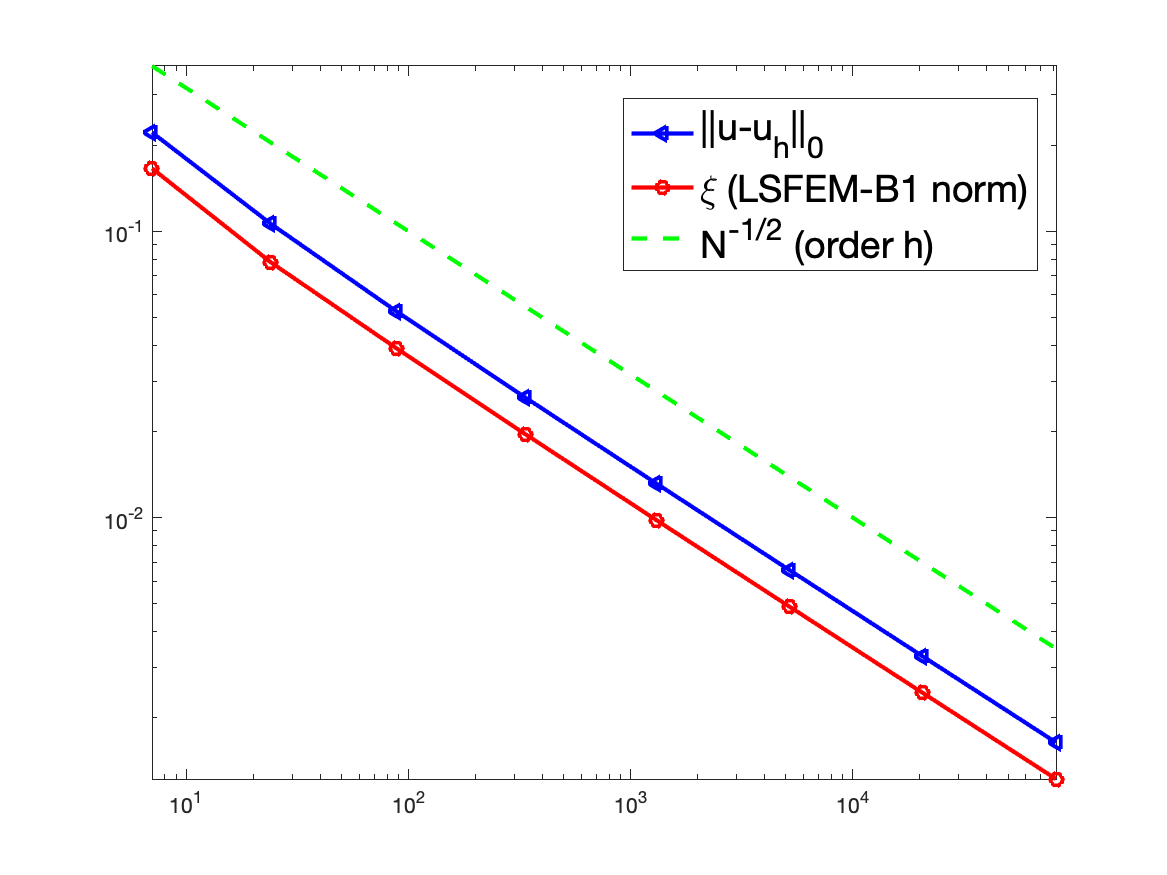}}
~
\subfigure[LSFEM-B2]{
\includegraphics[width=0.3\linewidth]{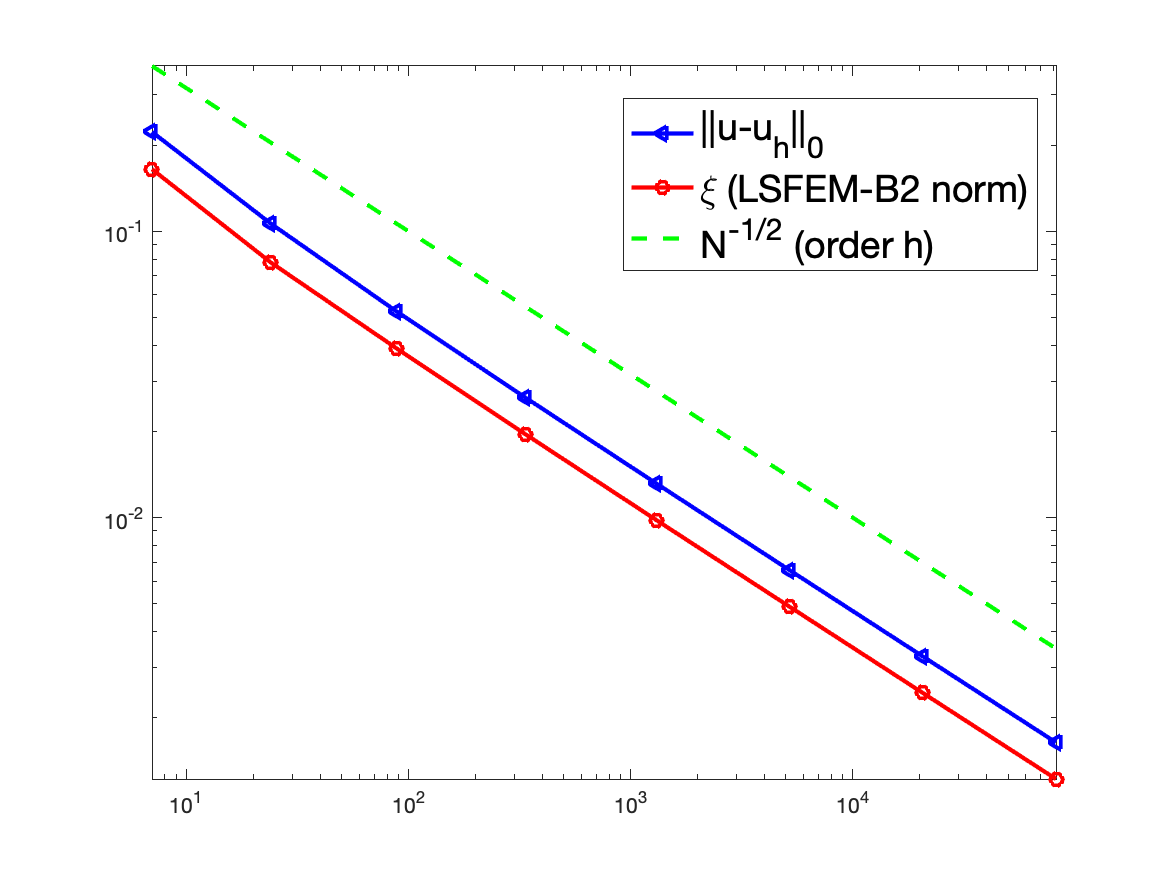}}
\caption{Global smooth solution: convergence histories on uniformly refined meshes of flux-based LSFEMs}
 \label{error_smooth}
\end{figure}

%\begin{figure}[!ht]
%\centering 
%\subfigure[LSFEM RT1P1]{ 
%\includegraphics[width=0.45\linewidth]{error_uni_gsmooth_LSFEM_RT1P1}}
%\caption{Global smooth solution: convergence histories on uniformly refined meshes with RT1-P1 pair}
% \label{error_smooth_rt1}
%\end{figure}

%In Fig. \ref{error_smooth_CLSFEM_DG}, the convergence histories of the C-LSFEM and P1-DGFEM on uniformly refined meshes are shown. It is clear that for this specific globally continuous problem, the C-LSFEM and P1-DGFEM both work well and obtain optimal convergence. Thus, for a globally smooth solution, it seems that the new flux-based LSFEMs have no obvious advantages.
%
%\begin{figure}[!ht]
%\centering 
%\subfigure[C-LSFEM]{ 
%\includegraphics[width=0.45\linewidth]{error_uni_gsmooth_CLSFEM}}
%%\hspace{0.01\linewidth}
%~
%\subfigure[P1-DGFEM]{
%\includegraphics[width=0.45\linewidth]{error_uni_gsmooth_P1DGFEM}}
%\caption{Global smooth solution: convergence histories on uniformly refined meshes}
% \label{error_smooth_CLSFEM_DG}
%\end{figure}

\subsection{Peterson example}
In Peterson \cite{Peterson:91}, a famous example is suggested to show that even for a smooth solution, discontinuous Galerkin methods cannot have optimal convergence order. The $L^2$-norm error estimate $\|u-u_{dg}\|_0 \leq C h^{k+1/2}\|u\|_{k+1}$ of the discontinuous Galerkin method using $P_k$ as approximation space cannot be improved. 

We have a similar situation in our case. Consider the test problem from section 3 of Peterson \cite{Peterson:91}: Let $\O = (0,1)^2$ and $\bbeta = (0,1)^T$. The inflow boundary $\Gamma_-$ is $\{x\in (0,1), y=0\}$, i.e., the south boundary of the domain.
\begin{eqnarray}
u_y = \gradt(\bbeta u)&=& 0 \quad \mbox{ in } \O,\\
u|_{\Gamma_-} &=& x \quad \mbox{ on } \Gamma_-.
\end{eqnarray} 
The exact solution is $u=x$. The mesh is chosen to be in the pattern on the left of Fig. \ref{error_peterson}. We compute a series of solutions by the LSFEM on meshes with $h$ from $1/6$, $1/12$, $\cdots$, to $1/768$. The convergence result is plotted on the right of Fig. \ref{error_peterson}. It is observed that the error in LS norm still converges in the order of $1$, but the $L^2$-norm of $u-u_h$ only converges in the order of $3/4$.

The rate difference of Examples 7.4 and 7.5 suggests that even for globally smooth solutions, the following norm equivalence (even in discrete spaces) does {\it not} hold:
$$
\tri (\btau, v) \tri \approx \|\btau\|_{H(\divvr;\O)} + \|v\|_0  \quad \forall (\btau,v) \in H_{0,-}(\divvr;\O) \times L^2(\O).
$$
Otherwise, we will have a uniform convergence order of $1$ for $\|u-u_h\|_0$.

\begin{figure}[!htb]
\centering 
\subfigure[Peterson mesh with $h=1/6$]{ 
\includegraphics[width=0.45\linewidth]{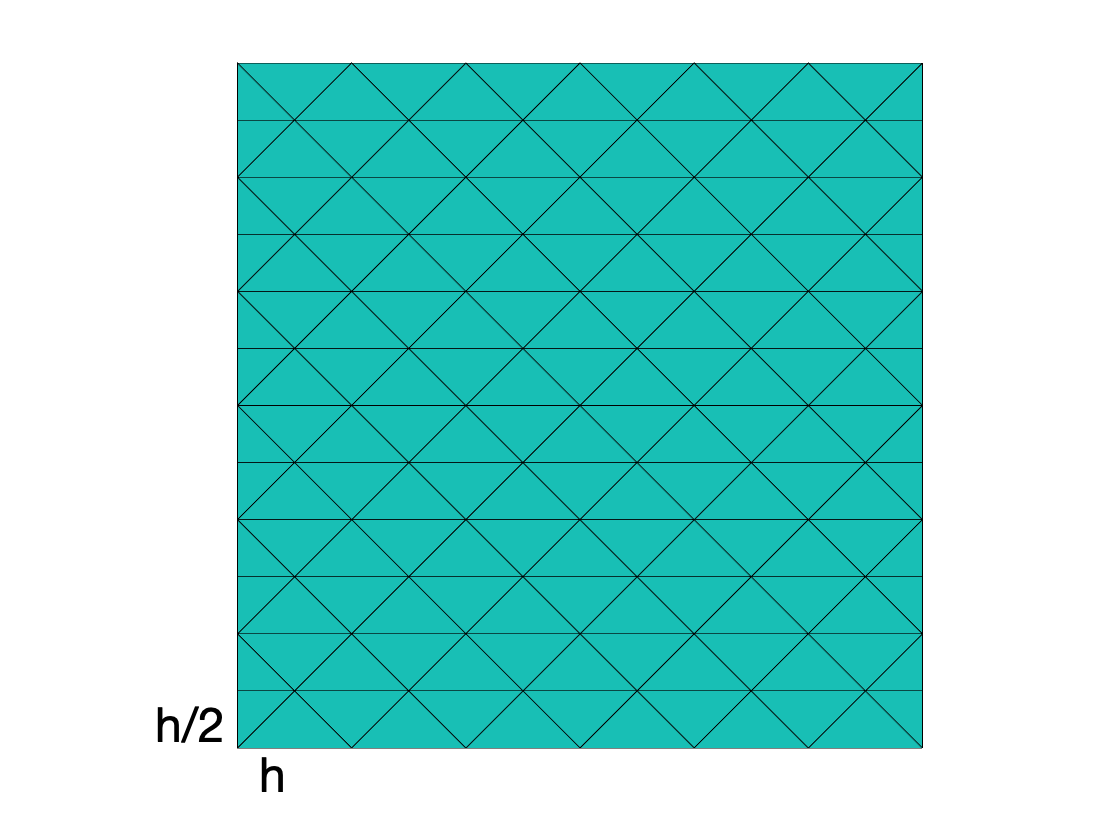}}
%\hspace{0.01\linewidth}
~
\subfigure[LSFEM convergence]{
\includegraphics[width=0.45\linewidth]{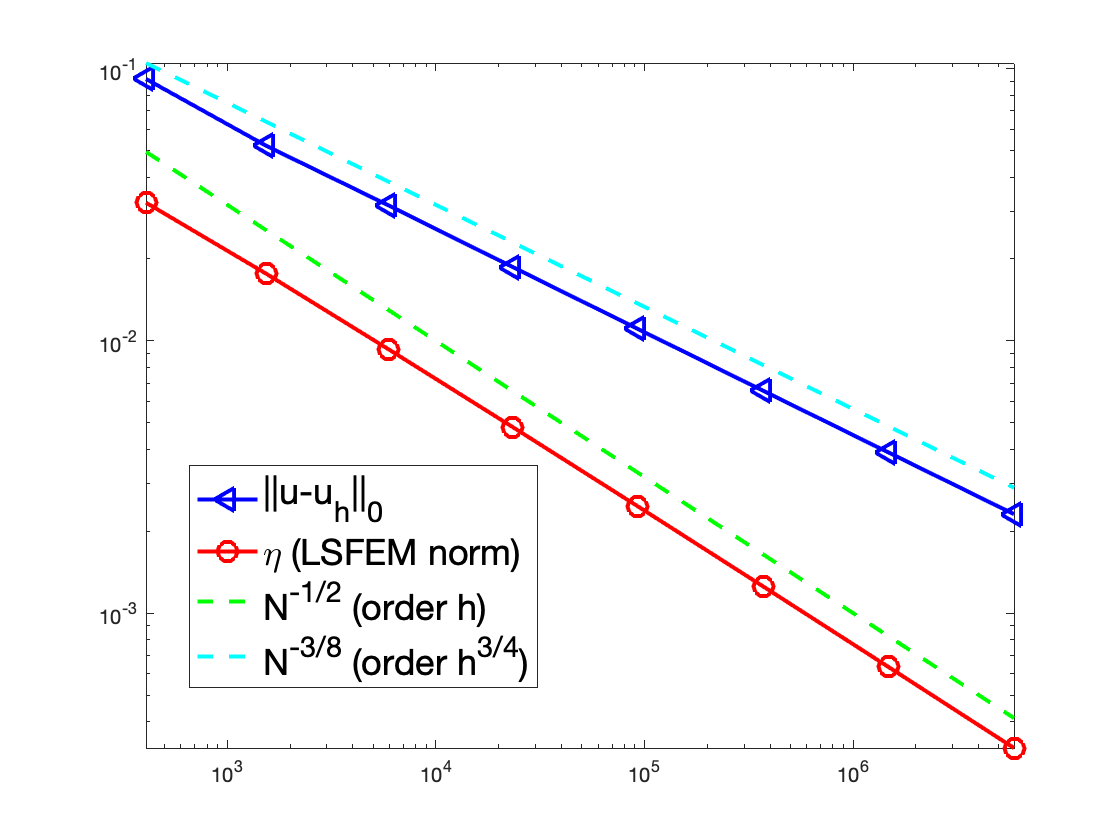}}
\caption{Peterson example}
 \label{error_peterson}
\end{figure}

\subsection{An example with a piecewise smooth solution, matching grid}
Consider the following problem: $\O = (0,1)^2$ with $\bbeta = (1/\sqrt{2},1/\sqrt{2})^T$. The inflow boundary is $\{x=0, y\in (0,1)\} \cup \{x\in (0,1), y=0\}$, i.e., the west and south boundary of the domain. 
%Let region be I $= \{ (x,y) \in (0,1)^2 \colon y>x\}$ and region II $= \{ (x,y) \in (0,1)^2 \colon y<x\}$.
Let $\gamma =1$. Choose $g$ and $f$ such that
%$$
%f = \left\{ \begin{array}{lll}
%\sin(x+y)+\sqrt{2}\cos(x+y) & \mbox{in} & \mbox{I}, \\[2mm]
%\cos(x+y) - \sqrt{2}\sin(x+y) & \mbox{in} & \mbox{II}.
%\end{array} \right.
%\quad\mbox{and}\quad
%g = \left\{ \begin{array}{lll}
%\sin y & \mbox{on} & \{0\}\times (0,1), \\[2mm]
%\cos x & \mbox{on} & (0,1)\times \{0\}.
%\end{array} \right.
%$$
the exact solution $u$ is
$$
u = \left\{ \begin{array}{lll}
\sin(x+y) &\mbox{if} &y>x, \\[2mm]
\cos(x+y) &\mbox{if} &y<x.
\end{array} \right.
$$

We choose an initial mesh that matches the discontinuity (Fig. \ref{initialmesh}) and uniformly refine it for $8$ times. In Fig. \ref{error_pws}, we show the convergence histories. For all three formulations, the convergence order of the errors in their corresponding least-squares norms is $1$. The order of $\|u-u_h\|_0$ is less than $1$ (about $0.6$ at late stages). This again suggests that the norm equivalence (or in discrete sub-spaces):
$$
\tri (\btau, v) \tri \approx \|\btau\|_{H(\divvr;\O)} + \|v\|_0  \quad \forall (\btau,v) \in H_{0,-}(\divvr;\O) \times L^2(\O),
$$
does not be true for the discontinuous solutions. %Similar results should also be true for the LSFEM-B2 formulation.

\begin{figure}[!htb]
\centering 
\subfigure[LSFEM]{ 
\includegraphics[width=0.3\linewidth]{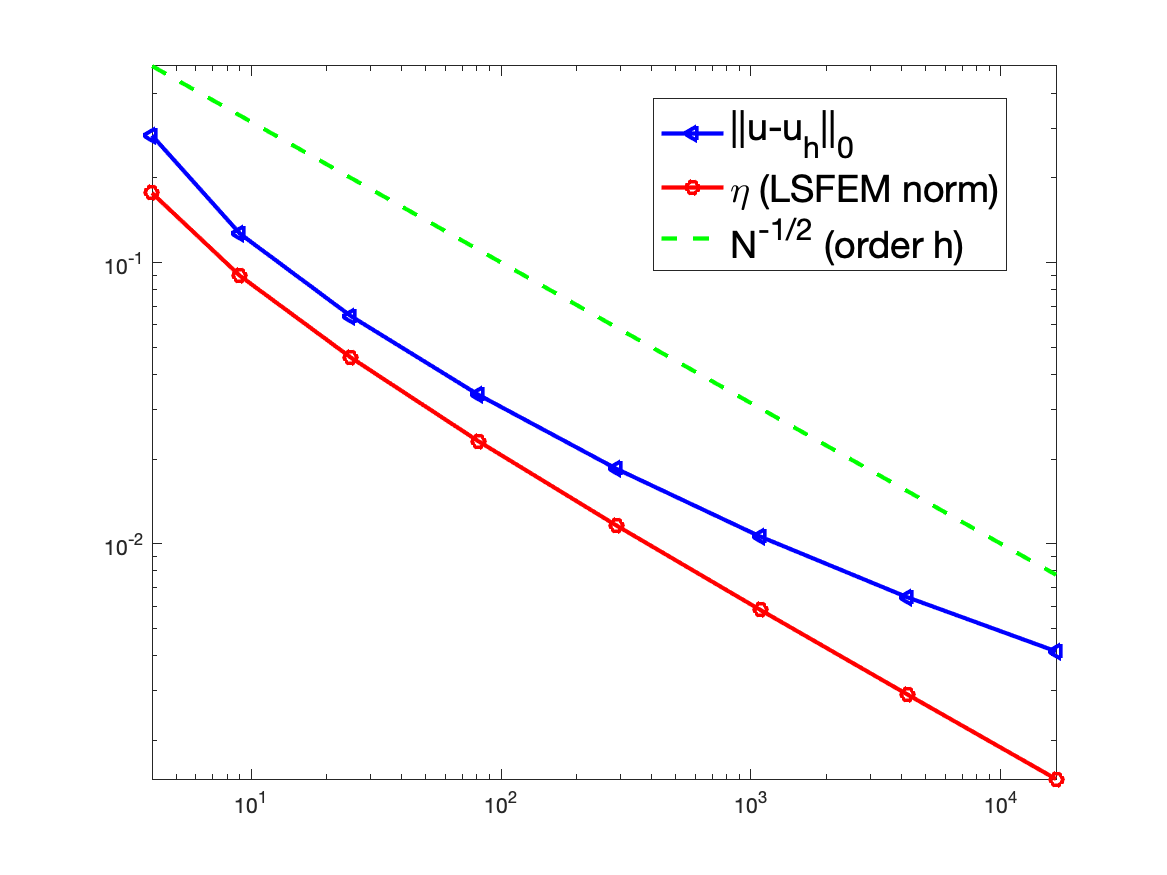}}
%\hspace{0.01\linewidth}
~
\subfigure[LSFEM-B1]{
\includegraphics[width=0.3\linewidth]{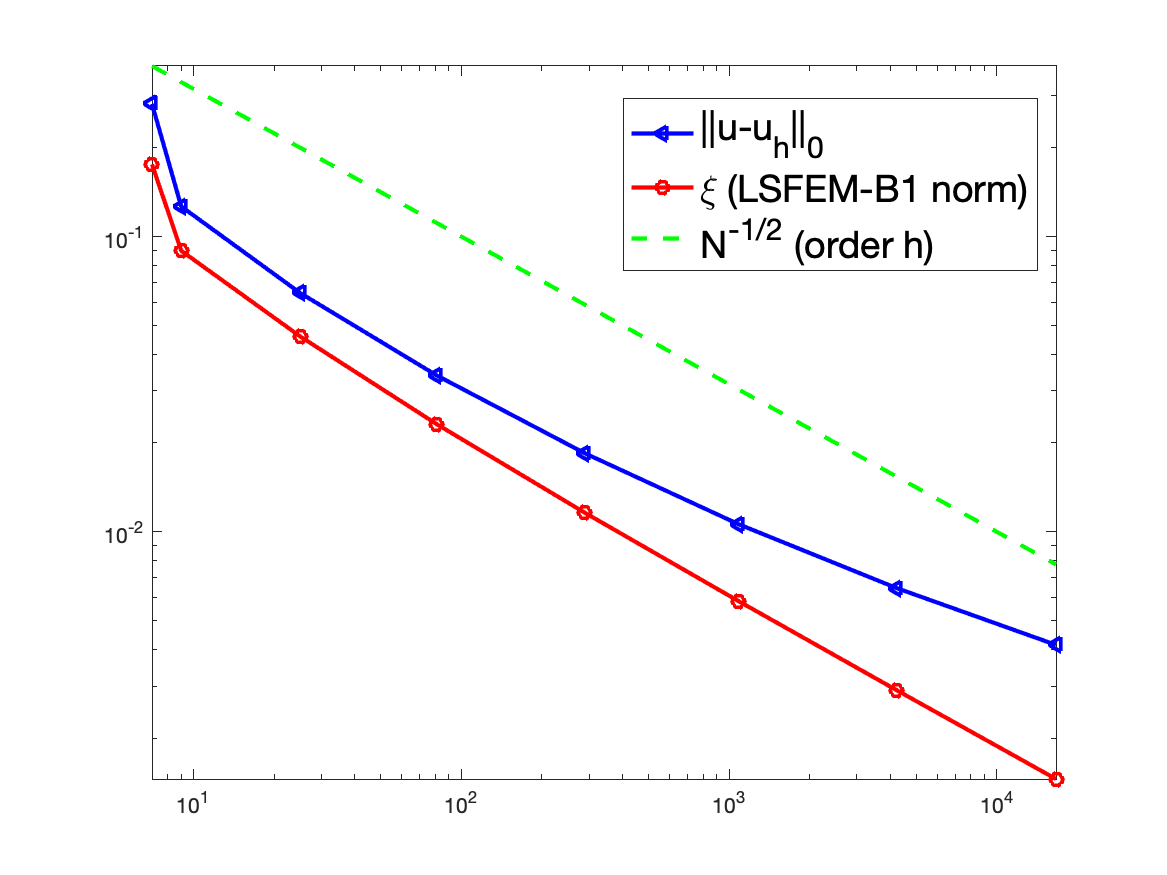}}
~
\subfigure[LSFEM-B2]{
\includegraphics[width=0.3\linewidth]{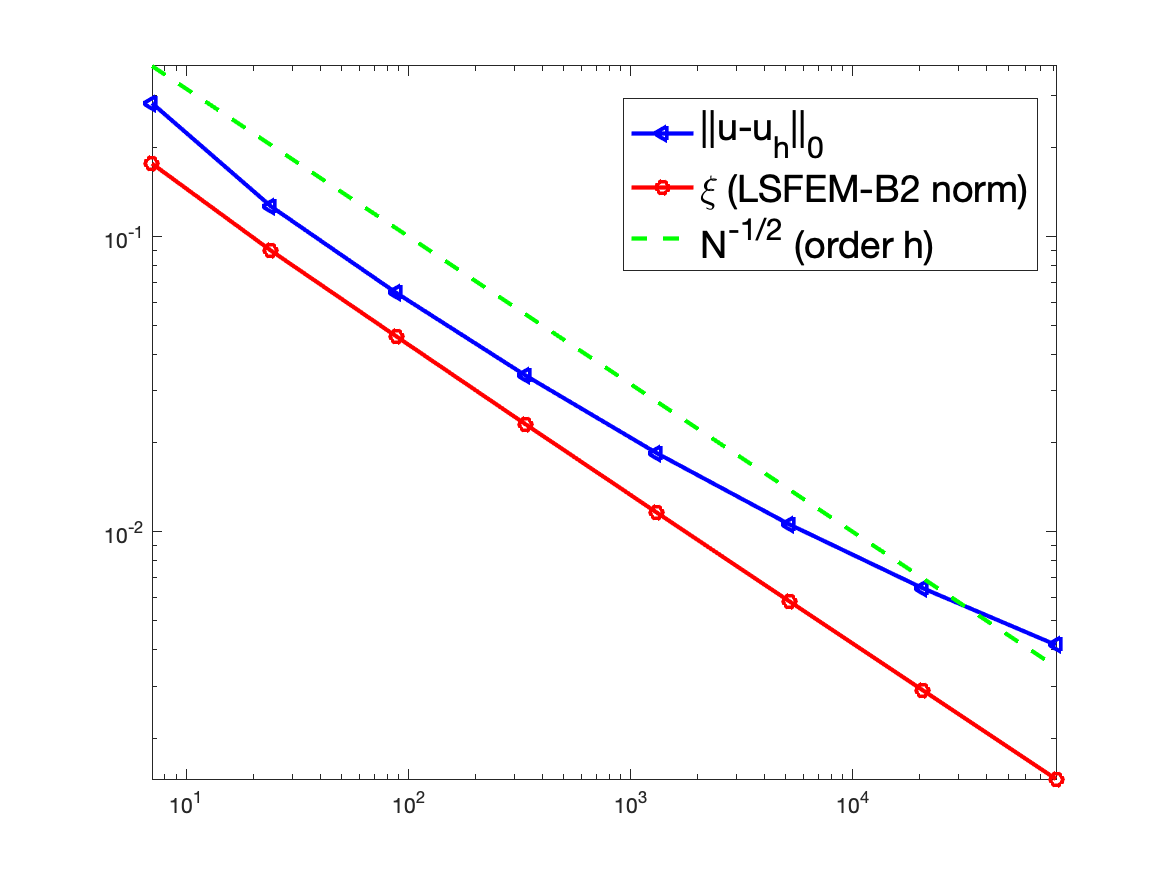}}
\caption{Piecewise smooth solution on a matching mesh: convergence histories of flux-based LSFEMs}
 \label{error_pws}
\end{figure}

%\begin{figure}[!ht]
%\centering 
%\subfigure[LSFEM RT1P1]{ 
%\includegraphics[width=0.45\linewidth]{error_pws_gsmooth_LSFEM_RT1P1}}
%%\hspace{0.01\linewidth}
%%~
%%\subfigure[P1-DGFEM]{
%%\includegraphics[width=0.45\linewidth]{error_uni_gsmooth_P1DGFEM}}
%\caption{Piecewise smooth solution on a matching mesh: convergence histories on uniformly refined meshes}
% \label{error_pws_LSFEM_RT1P1}
%\end{figure}
%
%In Fig. \ref{error_pwsmooth_CLSFEM_DG}, the convergence histories of the C-LSFEM and P1-DGFEM on uniformly refined meshes are shown. For this piecewise continuous but globally discontinuous problem, the P1-DGFEM both work well and obtain optimal convergence. Thus, for a globally smooth solution, it seems that the new flux-based LSFEMs have no obvious advantages.
%
%\begin{figure}[!ht]
%\centering 
%\subfigure[C-LSFEM]{ 
%\includegraphics[width=0.45\linewidth]{error_uni_pwsmooth_CLSFEM}}
%%\hspace{0.01\linewidth}
%~
%\subfigure[P1-DGFEM]{
%\includegraphics[width=0.45\linewidth]{error_uni_pwsmooth_P1DGFEM}}
%\caption{Piecewise smooth solution on a matching mesh: convergence histories on uniformly refined meshes of C-LSFEM and P1-DGFEM}
% \label{error_pwsmooth_CLSFEM_DG}
%\end{figure}

%%%%%%%%%
\subsection{An example with a piecewise constant solution, non-matching grid}
In this example, we discuss the over/undershootings of the solution when the mesh is not matched with discontinuity.

Consider the problem: $\O = (0,2)\times (0,1)$ with $\bbeta = (0,1)^T$. 
The inflow boundary is $\{x\in (0,1), y=0\}$, i.e., the south boundary of the domain. Let $\gamma =0$ and $f=0$. Choose the inflow boundary condition such that
%$$
%g(x) = \left\{ \begin{array}{lll}
%0 & \mbox{if} & x< \pi/3, \\[2mm]
%1 & \mbox{if} & x > \pi/3.
%\end{array} \right.
%$$
the exact solution is
$$
u(x,y) = \left\{ \begin{array}{lll}
0 & \mbox{if} & x< \pi/3, \\[2mm]
1 & \mbox{if} & x > \pi/3.
\end{array} \right.
$$
We set the initial mesh to be as shown on the left of Fig. \ref{inimesh_pi}. 
The bottom central node is $(\pi/3,0)$ and the top central node is $(1,1)$. 
So the inflow boundary mesh is matched with the inflow boundary condition 
while the mesh is not aligned with the discontinuity in general 
and will never match with it if bisection mesh refinement is used.

On the right of Fig. \ref{inimesh_pi}, 
we show the solution computed by LSFEM on a mesh after 8 uniform refinements of the initial mesh. Since it essentially is  a 1D problem, we project the graph of the solution onto the plane $y=0$,  that is, we plot the numerical solution value at the midpoint of x-axis of each elements. We do see some under/overshooting. 
%Note that even the inflow boundary condition is exact, we still see some small undershooting. 
The maximum of $u_h$ is $1.0629$ and the minimum of $u_h$ is $-0.0339$. 
%The overshooting in the outflow boundary is about the double size of the inflow undershooting. %with $h=0.0027$

\begin{figure}[!htb]
\centering 
\subfigure[initial mesh]{ 
\includegraphics[width=0.45\linewidth]{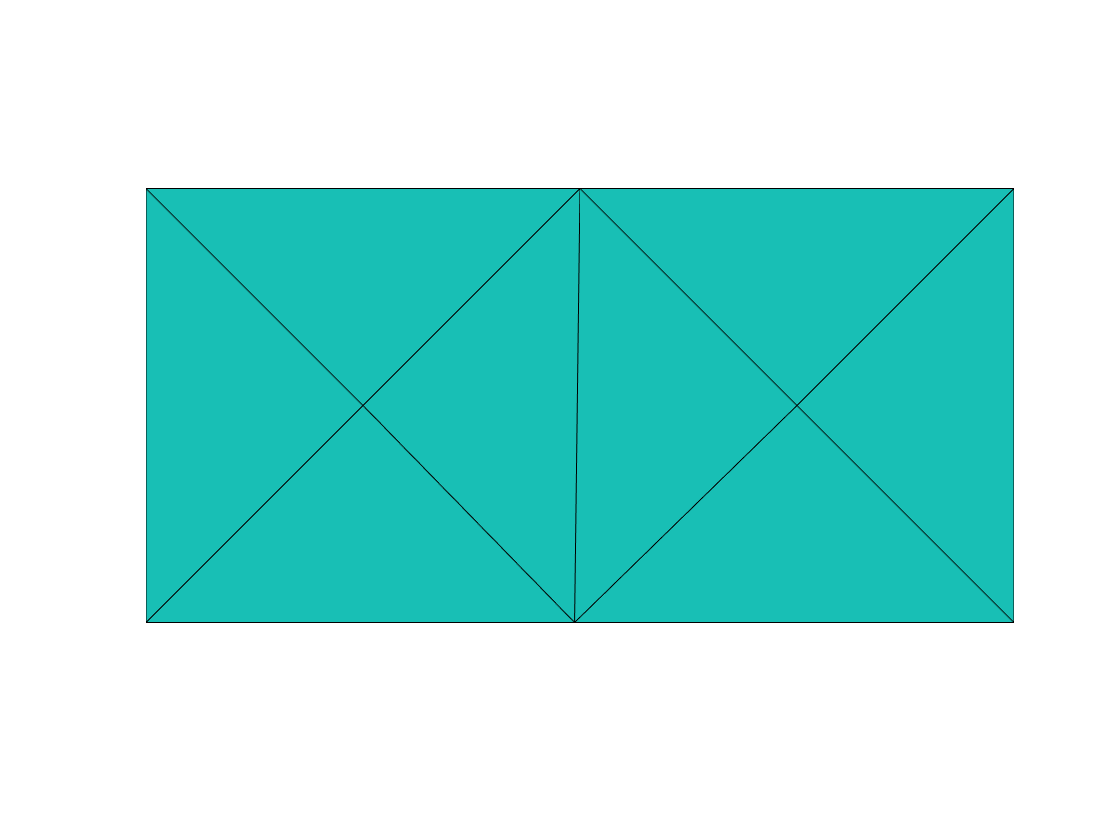}}
%\hspace{0.01\linewidth}
~
\subfigure[projected LSFEM solution on the uniform mesh]{
\includegraphics[width=0.45\linewidth]{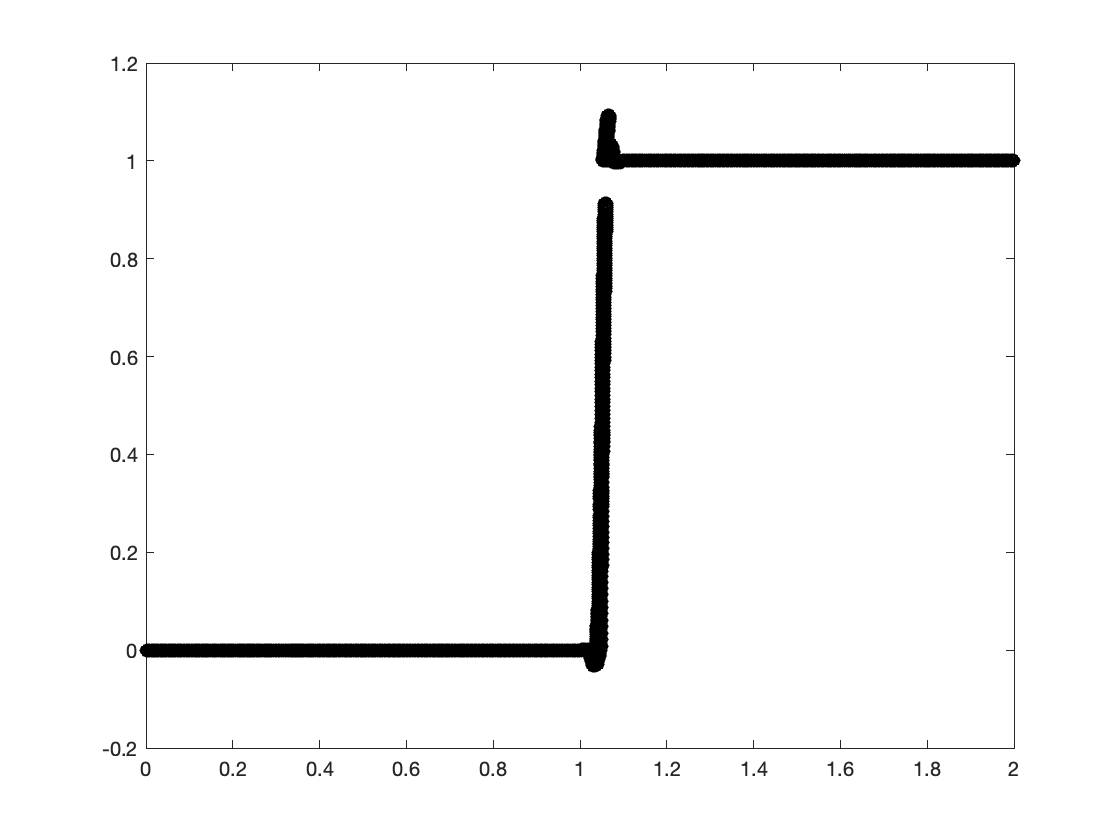}}
\caption{Piecewise constant solution with a non-matching grid test problem}
 \label{inimesh_pi}
\end{figure}

On the left of Fig. \ref{error_10_uniform}, we plot the convergence results of uniform refinements.
The decay rate of the error measured in the least-squares norm is about $0.7$. 
The reason that the rate is less than $1$
is that the discontinuity is cutting though those interface elements 
so that $u$, $\bsigma$, and $\gradt \bsigma$ are not of
$H^1$ in those elements. But the rate is apparently better than $1/2-\epsilon$, 
even though all those true solutions are only in $H^{1/2-\epsilon}(K)$ for those interface elements. 
The possible reason for the better rate can be that 
the Sobolev space $H^{1/2-\epsilon}$ may not be the best space to
characterize the piecewisely discontinuous function space. The order of $\|u-u_h\|_0$ is about $1/2$.

We then test the problem by adaptive mesh refinements. 
On the center of Fig. \ref{error_10_uniform}, adaptive refined meshes after some iterations are shown. Clearly, the refinements are along the discontinuity. 
On the right of Fig. \ref{error_10_uniform}, we show the convergence histories. 
The error measured in the LS norm is optimal with order $1$, while the order of $\|u-u_h\|_0$ is about $1/2$, 
which is about the same order as the uniform refinement.

%
%\begin{figure}[!ht]
%	\label{error_10_uniform}
%    \begin{minipage}[!hbp]{0.33\linewidth}
%  %  \label{error_10_uniform}
%        \centering
%        \includegraphics[width=0.99\textwidth,angle=0]{error_uni_pwc_LSFEM}
%%        \caption{Convergence histories for the piecewise constant solution with a non-matching grid on uniformly refined meshes}
%        \end{minipage}%
%    \begin{minipage}[!htbp]{0.33\linewidth} %\label{mesh10}
%        \centering
%        \includegraphics[width=1\textwidth,angle=0]{mesh_amr_10_LS.png}
%%        \caption{Adaptive refined meshes after some iterations for the piecewise constant solution with a non-matching grid
%%        }
%    \end{minipage}%
%    \begin{minipage}[!htbp]{0.33\linewidth} %\label{convergence10}
%        \centering
%        \includegraphics[width=1\textwidth,angle=0]{error_afem_pwc_LSFEM.png}
%    \end{minipage}%
%    \caption{Piecewise constant solution with a non-matching grid test problem: convergence history on uniformly refined meshes(left), adaptive refined meshes after some iterations(center), convergence history 
% on adaptive refined meshes(right)}
%\end{figure}

\begin{figure}[!htb]
\centering 
\subfigure[convergence history on uniformly refined meshes]{ 
\includegraphics[width=0.3\linewidth]{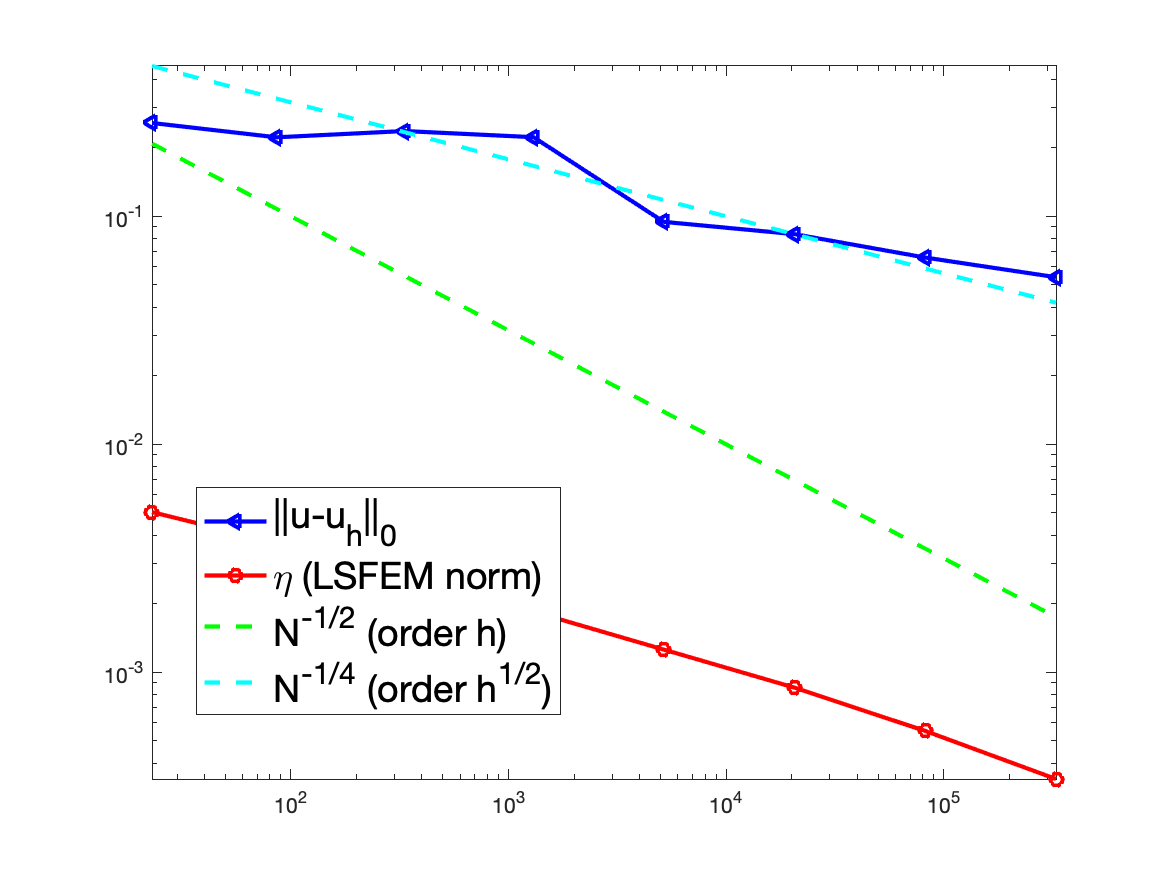}}
%\hspace{0.01\linewidth}
~
\subfigure[an adaptive refined mesh]{
\includegraphics[width=0.3\linewidth]{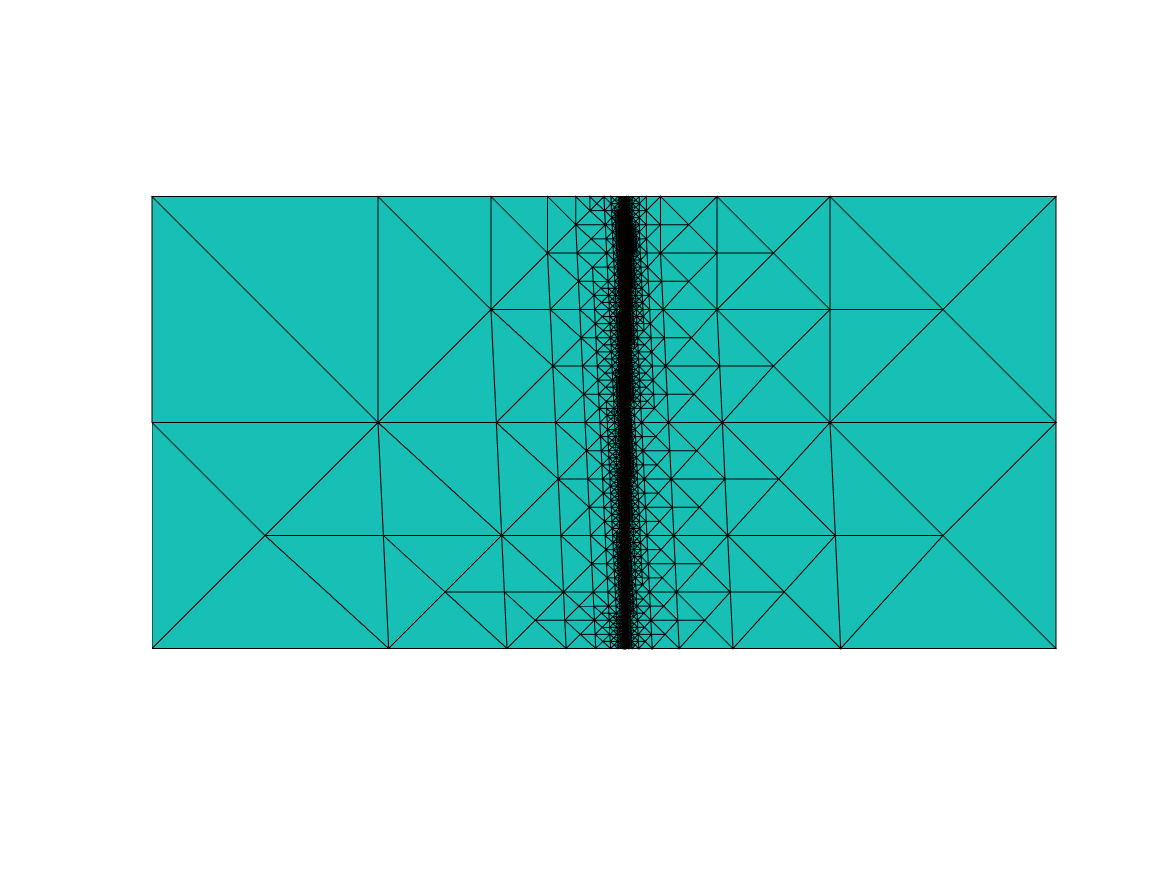}}
~
\subfigure[convergence history on adaptive refined meshes]{
\includegraphics[width=0.3\linewidth]{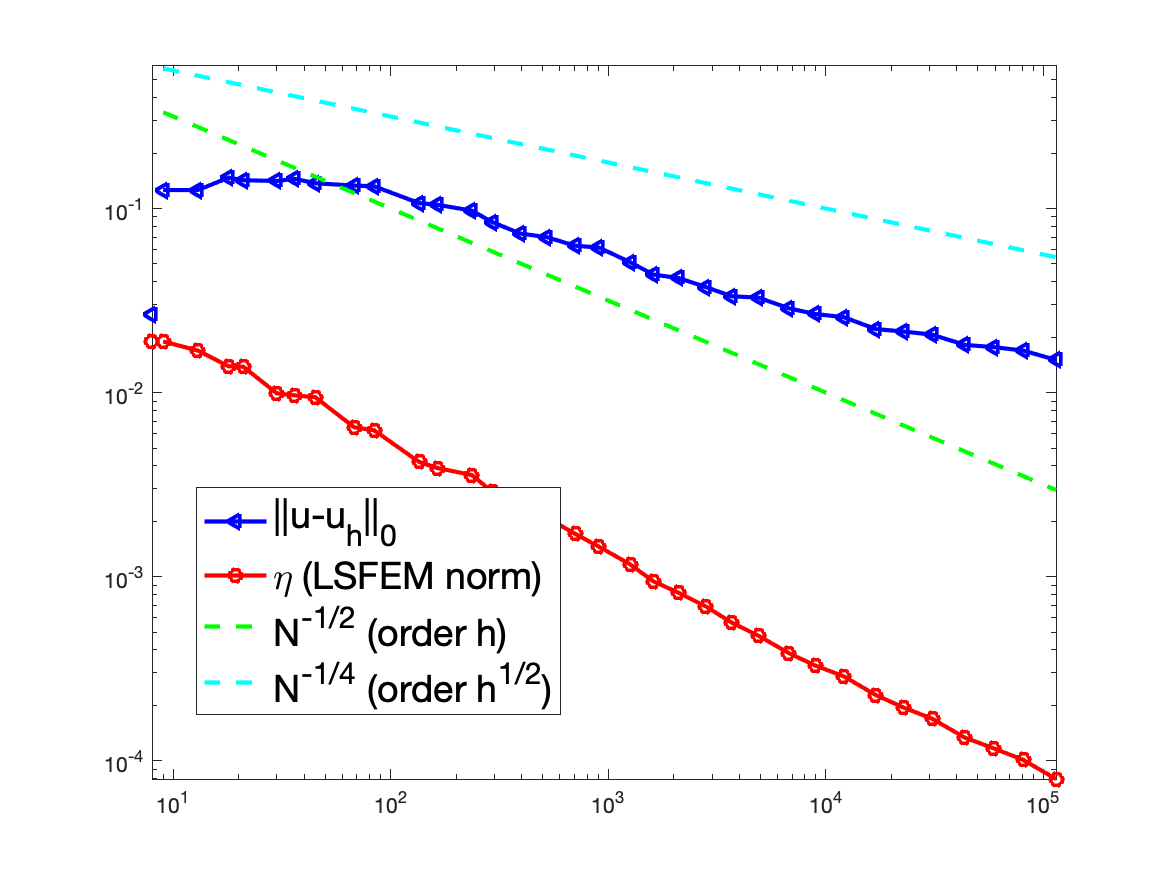}}
\caption{Piecewise constant solution with a non-matching grid test problem}
 \label{error_10_uniform}
\end{figure}

On the left of Fig. \ref{overshhot_amr_LS}, we show the decreasing of the overshooting values by adaptive mesh refinements. Here, the overshooting value is defined as $\max(\max(u_h-1), -\min(u_h))$. We clearly see after the mesh is reasonably fine (when the mesh is coarse, the overshooting is actually not very severe since we approximate $u$ by $P_0$), the overshooting value begins to decrease. On the right of Fig. \ref{overshhot_amr_LS}, we show a projected solution on the final mesh. It is clear that when the mesh is fine, the overshooting is almost neglectable with $RT_0\times P_0$ approximation. 

We also test the same LSFEM with $RT_1\times P_1$ approximations, the result can be found in Fig. \ref{pwc_nonmatching_LSFEM_RT1P1}. It is clear that if we use $P_1$ functions to approximate the discontinuous solutions on a non-matching adaptive mesh, the refinements cannot reduce the overshooting. This matches the discussions we have in \cite{Zhang:19}, that on a non-matching mesh, only piecewise constant approximation can reduce the overshooting, other higher order continuous or discontinuous finite elements cannot.

On Fig. \ref{afem_sol_lsfem_dg}, we show the numerical solutions computed by the RT0P0-LSFEM and the RT1P1-LSFEM on the same final adaptive mesh. The overshooting is obvious for the RT1P1-LSFEM. 

For this problem, C-LSFEM will get a disastrous result.

\begin{figure}[!htb]
\centering 
\subfigure[reduction of overshootings]{ 
\includegraphics[width=0.45\linewidth]{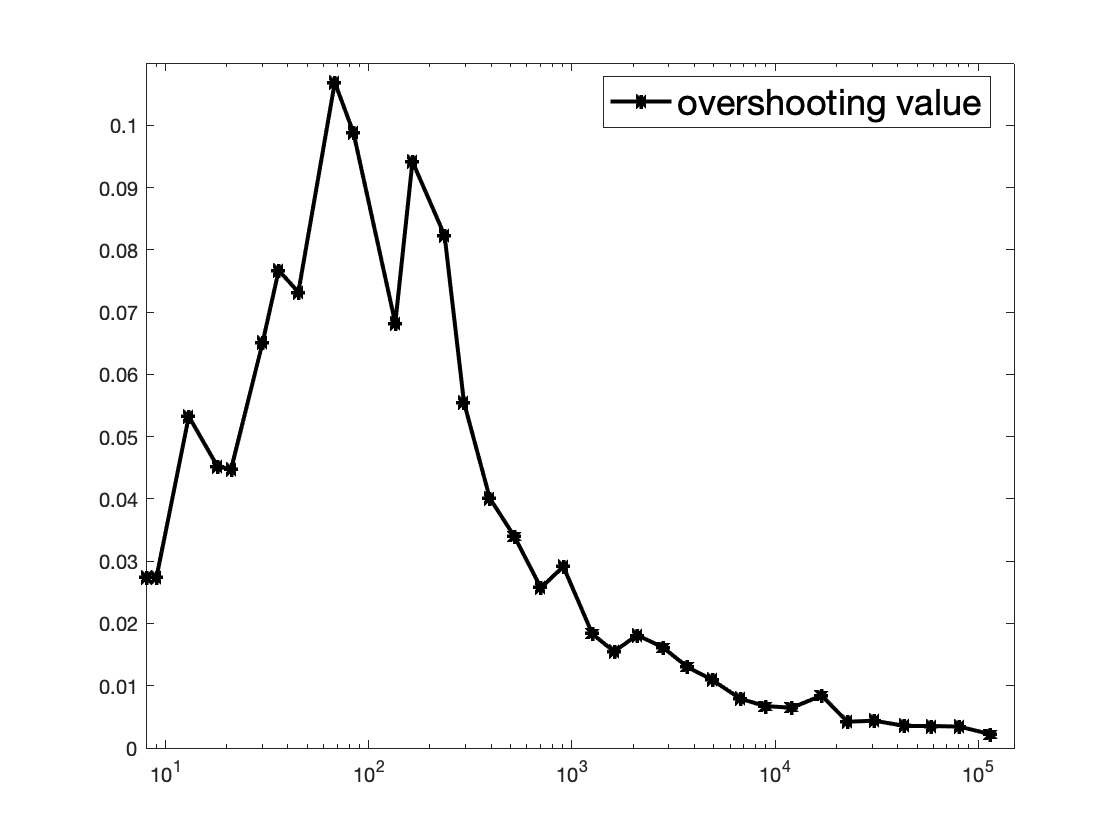}}
~
\subfigure[projected solution]{
\includegraphics[width=0.45\linewidth]{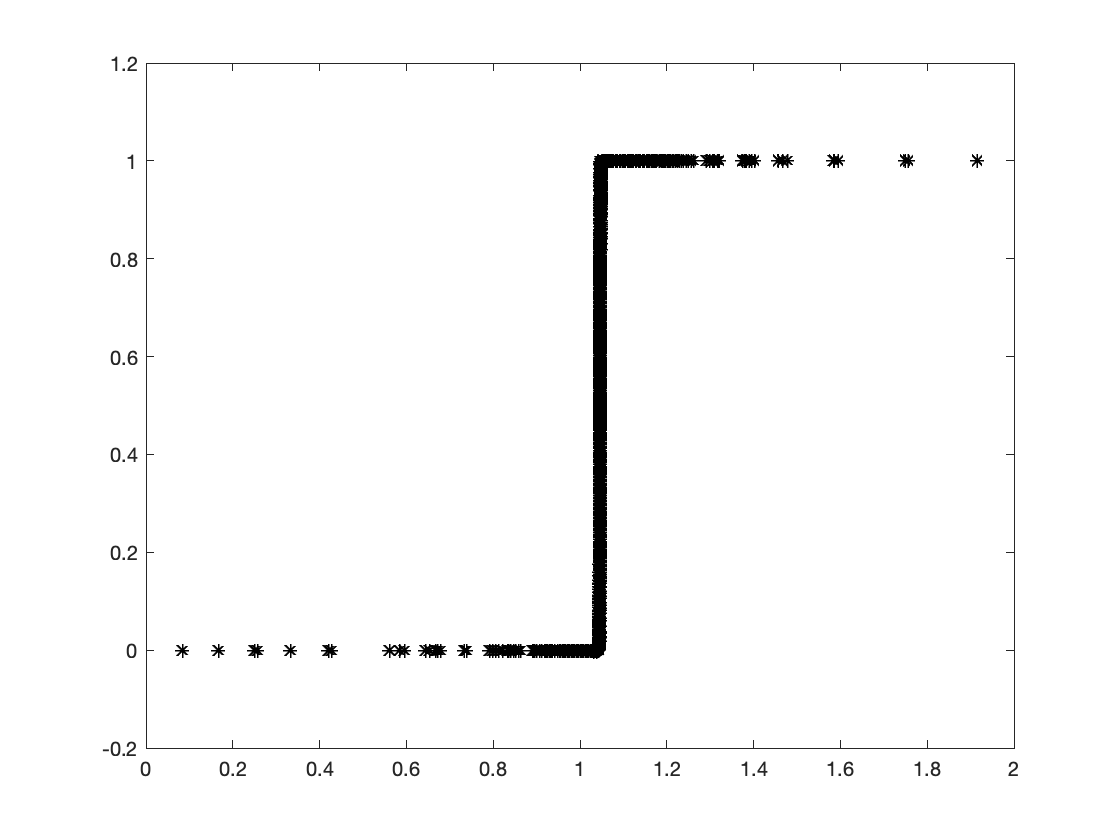}}
\caption{Piecewise constant solution with a non-matching grid test problem by RT0P0-LSFEM on adaptive meshes}
 \label{overshhot_amr_LS}
\end{figure}

\begin{figure}[!ht]
\centering 
\subfigure[overshooting]{ 
\includegraphics[width=0.45\linewidth]{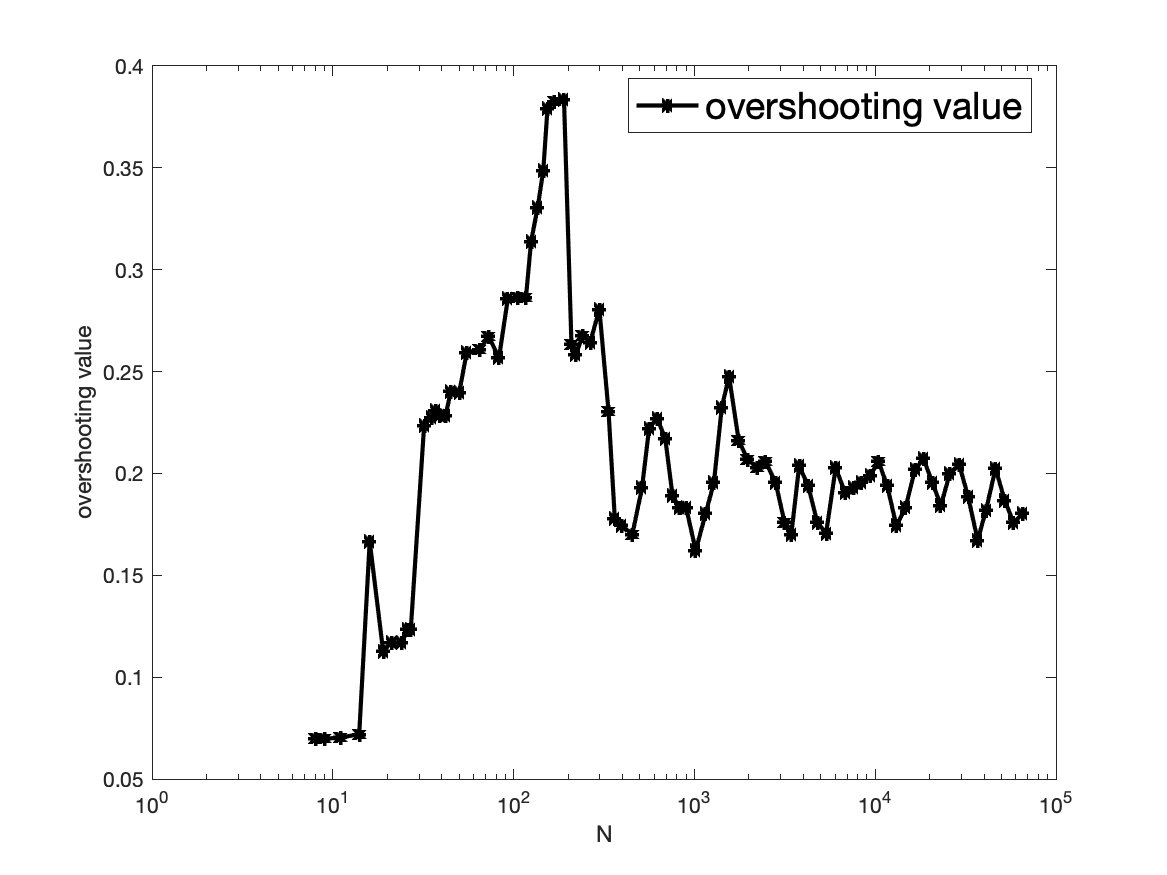}}
%\hspace{0.01\linewidth}
~
\subfigure[projected solution]{
\includegraphics[width=0.45\linewidth]{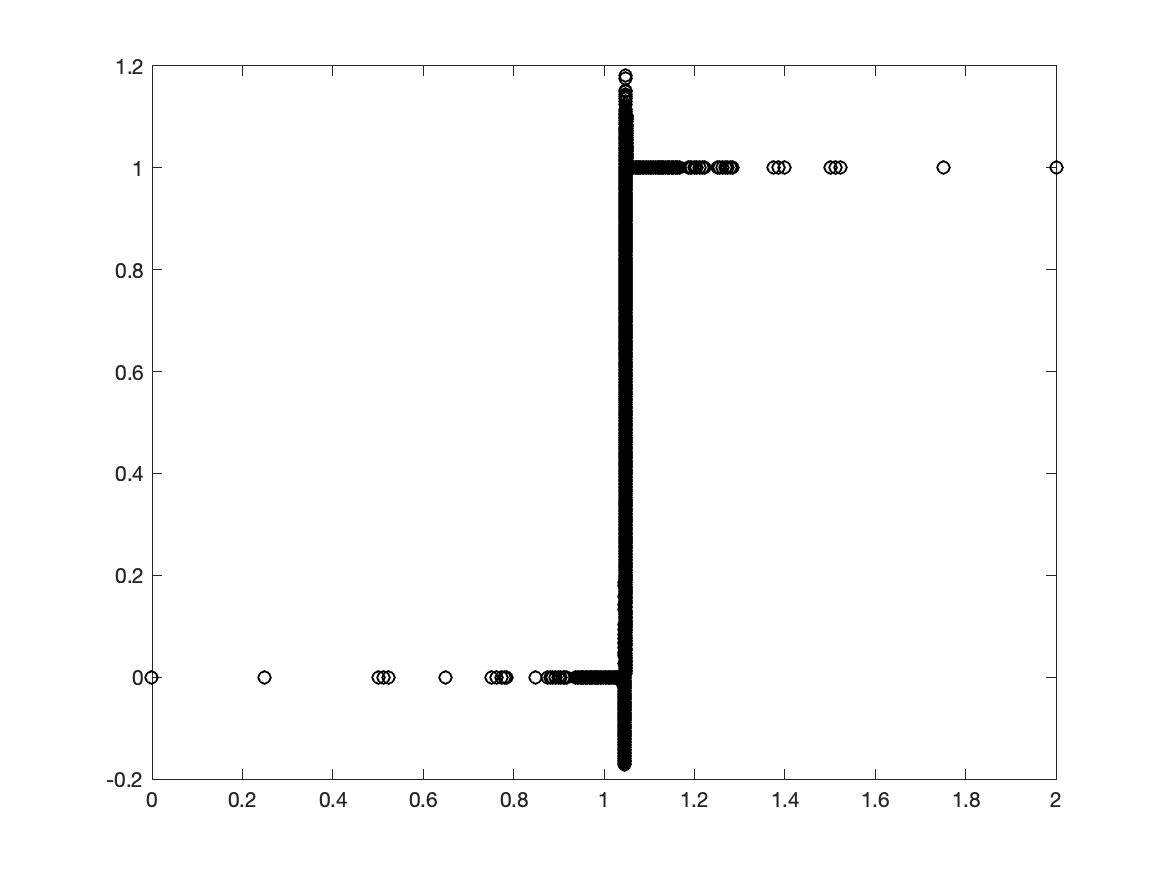}}
\caption{Piecewise constant solution with a non-matching grid test problem by RT1P1-LSFEM on adaptive meshes}
 \label{pwc_nonmatching_LSFEM_RT1P1}
\end{figure}

\begin{figure}[!htb]
\centering 
\subfigure[RT0P0-LSFEM]{ 
\includegraphics[width=0.45\linewidth]{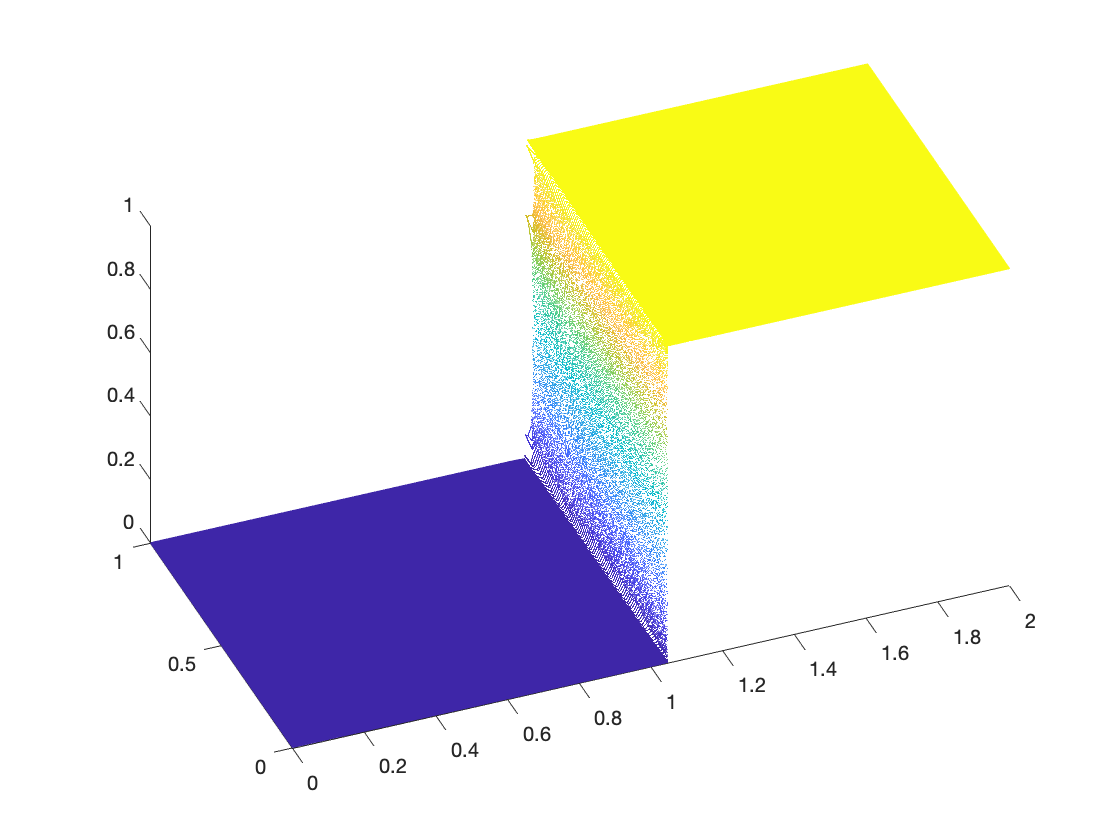}}
~
\subfigure[RT1P1-LSFEM]{
\includegraphics[width=0.45\linewidth]{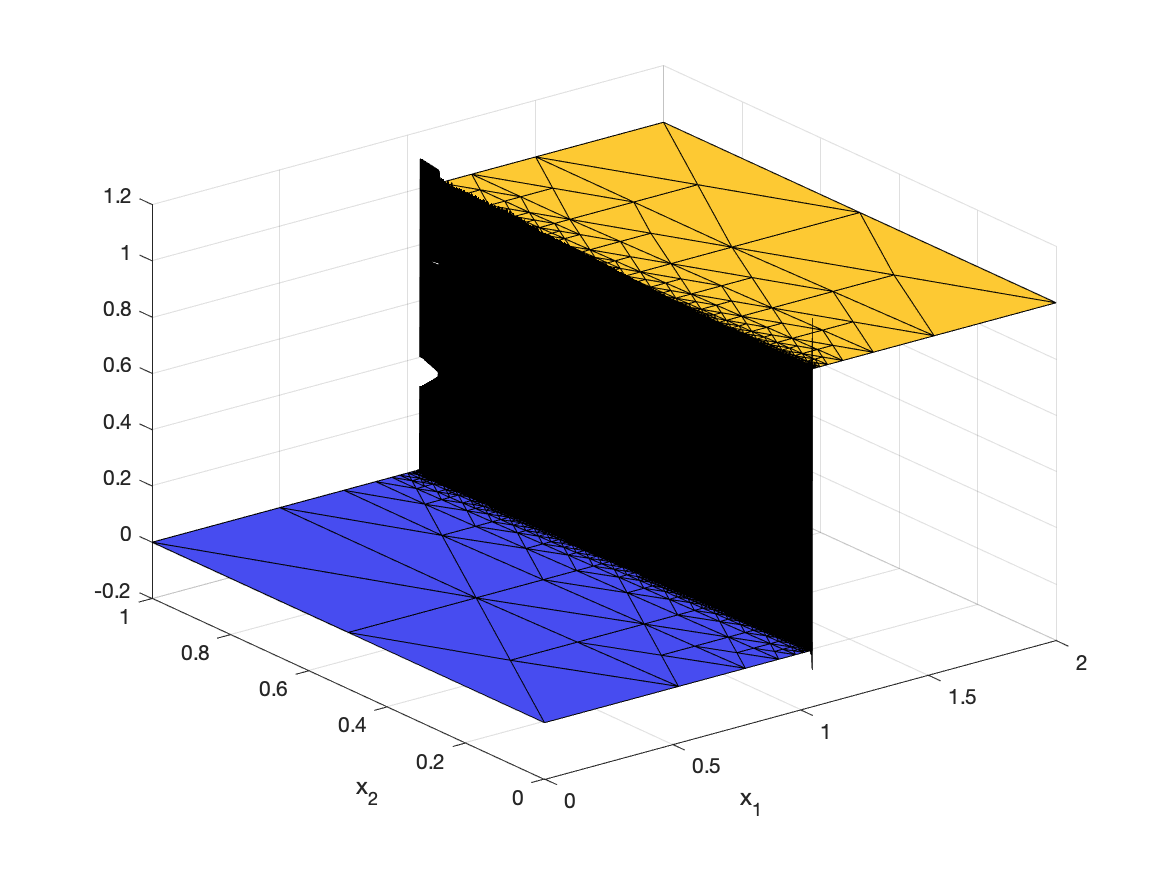}}
\caption{Piecewise constant solution with a non-matching grid test problem: numerical solutions with RT0P0 and RT1P1-LSFEMs}
 \label{afem_sol_lsfem_dg}
\end{figure}

%\begin{figure}[!ht]
%  \label{
%  convergence_10_DG}
%    \begin{minipage}[!hbp]{0.47\linewidth}
%    %\label{overshhot_amr}
%        \centering
%        \includegraphics[width=1\textwidth,angle=0]{convergence_10_DG.png}
%        \end{minipage}%
%       \caption{Piecewise constant solution with a non-matching grid test problem by P1-DG: convergence}
%\end{figure}
%
%\begin{figure}[!ht]
%\centering 
%\subfigure[overshooting]{ 
%\includegraphics[width=0.45\linewidth]{overshotting_pwc_gnonmatching_LSFEM_RT1P1}}
%%\hspace{0.01\linewidth}
%~
%\subfigure[projected solution]{
%\includegraphics[width=0.45\linewidth]{projected_amr_10_RT1P1}}
%\caption{Piecewise constant solution with a non-matching grid test problem by LSFEM RT1-P1 ion an adaptive mesh: overshooting and projected solution}
% \label{pwc_nonmatching_LSFEM_RT1P1}
%\end{figure}

\subsection{An example with a piecewise smooth solution, non-matching grid}
Consider the following simple problem with $\bbeta = (\cos(1/8),\sin(1/8))^T$ and $\O = (0,1)^2$. The inflow boundary is $\{x=0, y\in (0,1)\} \cup \{x\in (0,1), y=0\}$, i.e., the west and south boundaries 
of the domain. Let $\gamma =1$. Choose $g$ and $f$ such that the exact solution $u$ is
$$
u = \left\{ \begin{array}{lll}
\sin(x+y) &\mbox{if}& y>\tan(1/8)x, \\[2mm]
\cos(x+y) &\mbox{if}& y< \tan(1/8)x.
\end{array} \right.
$$
Note that with an initial mesh as in Fig. \ref{initialmesh}, any refinement of it will never match the discontinuity. 

We show the uniform convergence result on the left of Fig. \ref{error_pws_nm_uniform}. 
The convergence oder in LS norms is about $0.8$. 
Similar to the piecewise constant solution on non-matching grids,
it is worse than order $1$ but better than order $1/2$.
The convergence order for $\|u-u_h\|_0$ is about $0.3$, 
which is worse than the piecewise constant non-matching 
case.

On the center Fig. \ref{error_pws_nm_uniform}, an adaptive mesh by LSFEM is shown. Many refinements are generated near the discontinuity. 
On the right of Fig. \ref{error_pws_nm_uniform}, convergence history of adaptive LSFEM is shown. The rate of convergence of error in the LS norm is about order $1$, and $\|u-u_h\|_0$ is about order $0.5$.

On Fig. \ref{sol_pws_nm_afem}, we show the numerical solutions computed by RT0P0-LSFEM and RT1P1-LSFEM. The overshooting is quite severe on the numerical solution obtained by RT1P1-LSFEM .

\begin{figure}[!htb]
\centering 
\subfigure[convergence history on uniformly refined meshes]{ 
\includegraphics[width=0.3\linewidth]{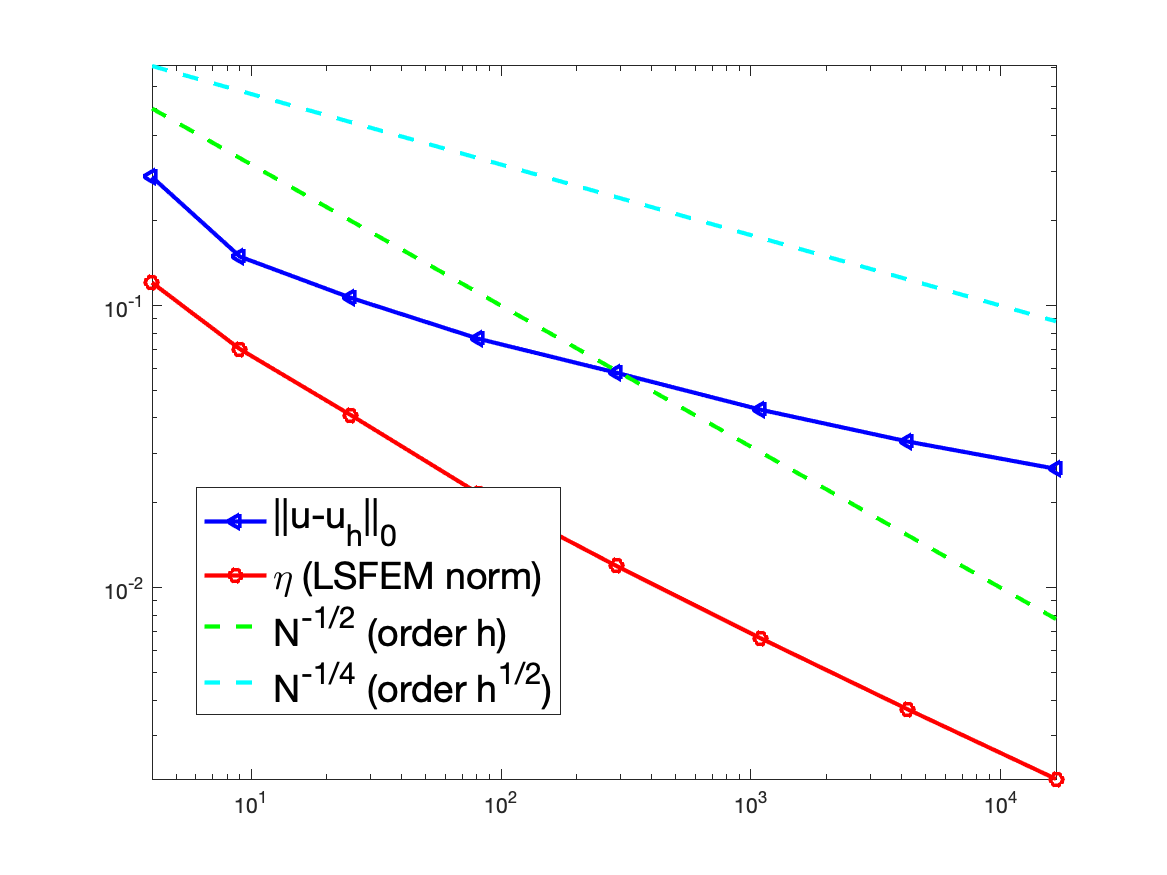}}
%\hspace{0.01\linewidth}
~
\subfigure[adaptive refined mesh]{
\includegraphics[width=0.3\linewidth]{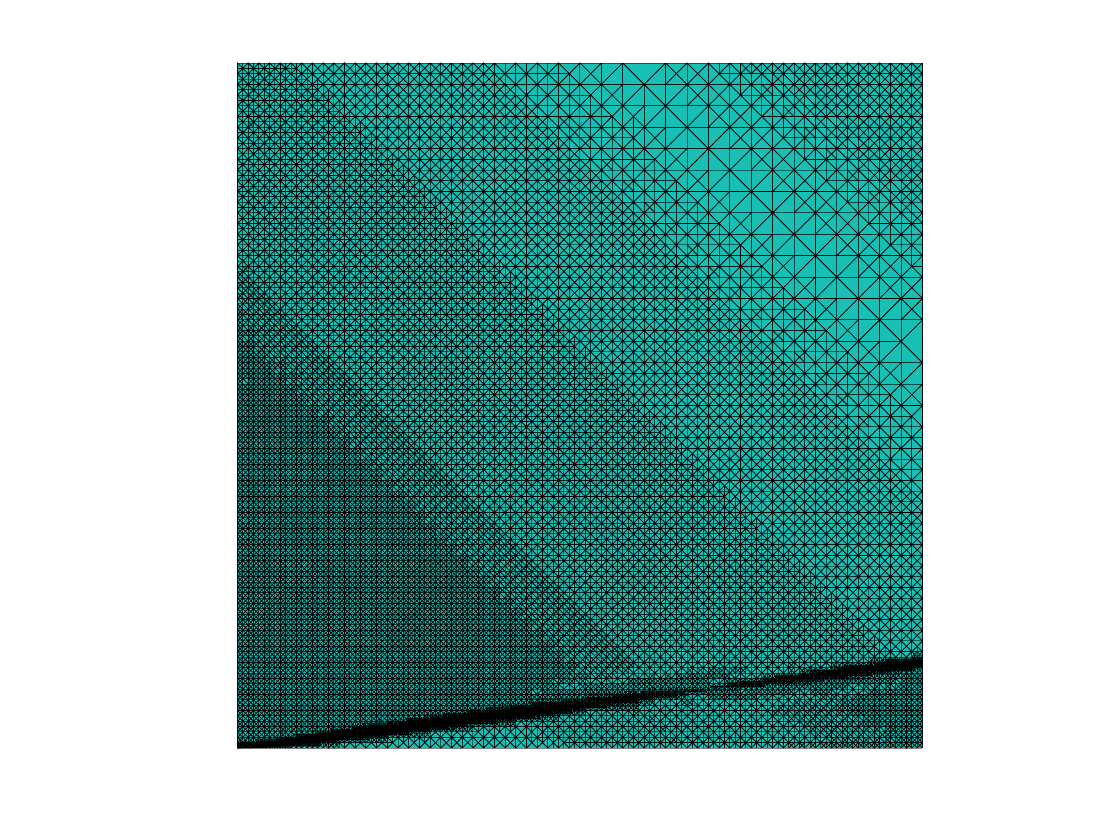}}
~
\subfigure[convergence history on adaptive refined meshes]{
\includegraphics[width=0.3\linewidth]{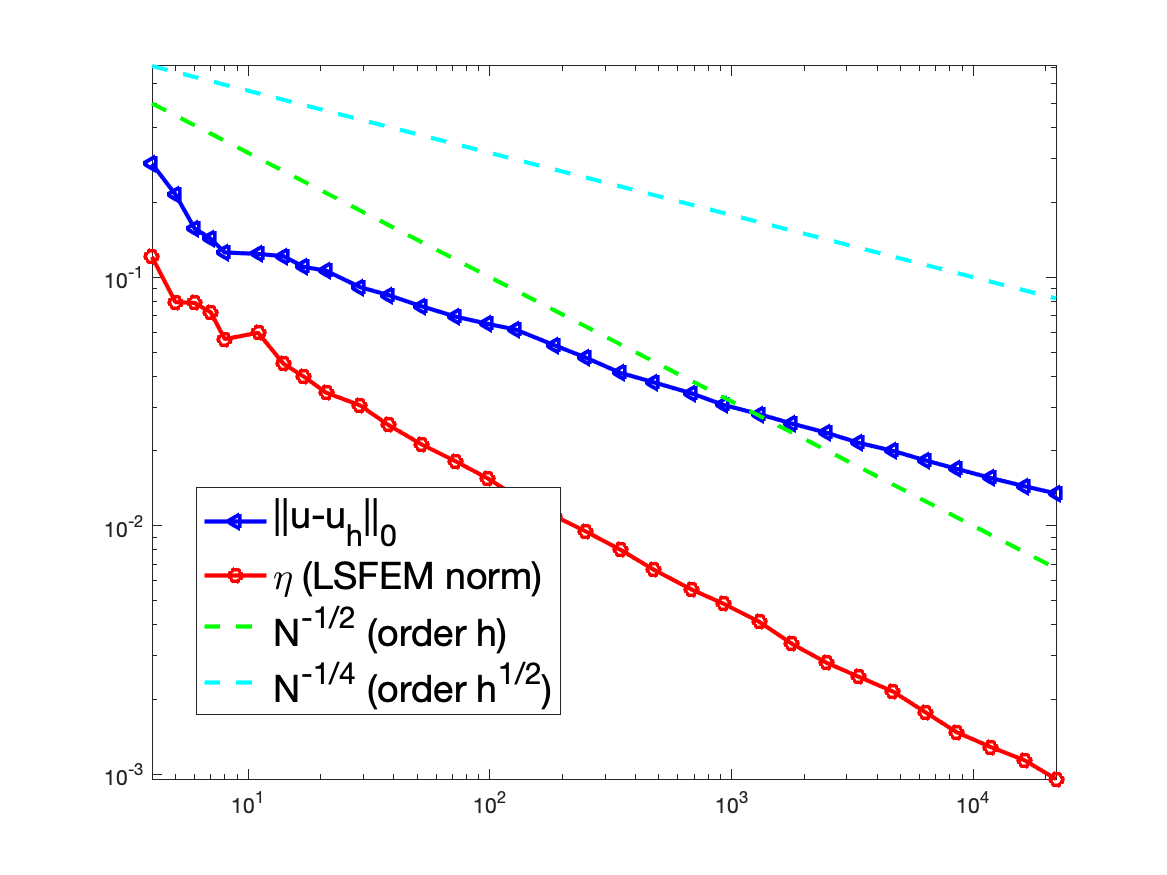}}
\caption{Piecewise smooth solution with a non-matching grid test problem by LSFEM}
 \label{error_pws_nm_uniform}
\end{figure}

%\begin{figure}[!htb]
%\centering 
%\subfigure[numerical solution by CLSFEM]{ 
%\includegraphics[width=0.3\linewidth]{solution_amr_PWS_nonmatching_CLSFEM}}
%%\hspace{0.01\linewidth}
%~
%\subfigure[adaptive refined mesh by CLSFEM]{
%\includegraphics[width=0.3\linewidth]{mesh_amr_PWS_nonmatching_CLSFEM}}
%~
%\subfigure[convergence history on adaptive refined meshes by CLSFEM]{
%\includegraphics[width=0.3\linewidth]{convergence_PWS_nonmatching_CLSFEM}}
%\caption{Piecewise smooth solution with a non-matching grid test problem by CLSFEM}
% \label{error_pws_nm_CLSFEM}
%\end{figure}

\begin{figure}[!htb]
\centering 
\subfigure[RT0P0-LSFEM]{ 
\includegraphics[width=0.45\linewidth]{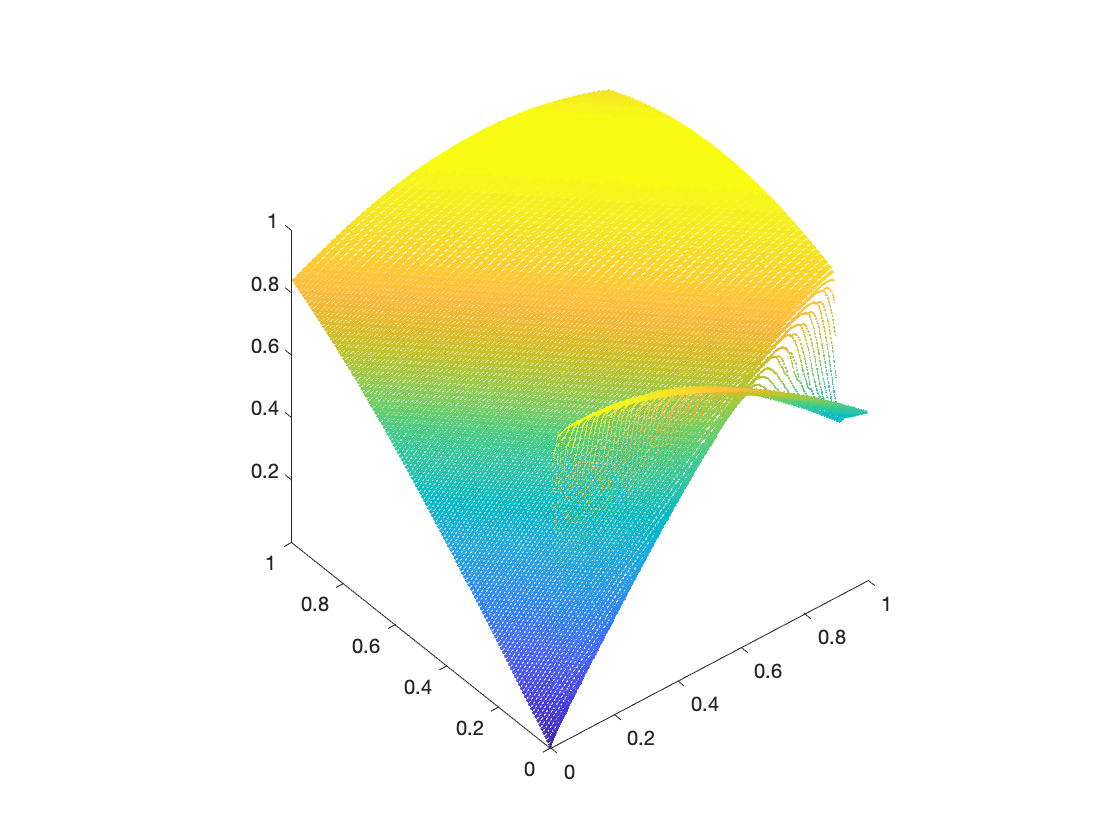}}
%\hspace{0.01\linewidth}
~
\subfigure[RT1P1-LSFEM]{
\includegraphics[width=0.45\linewidth]{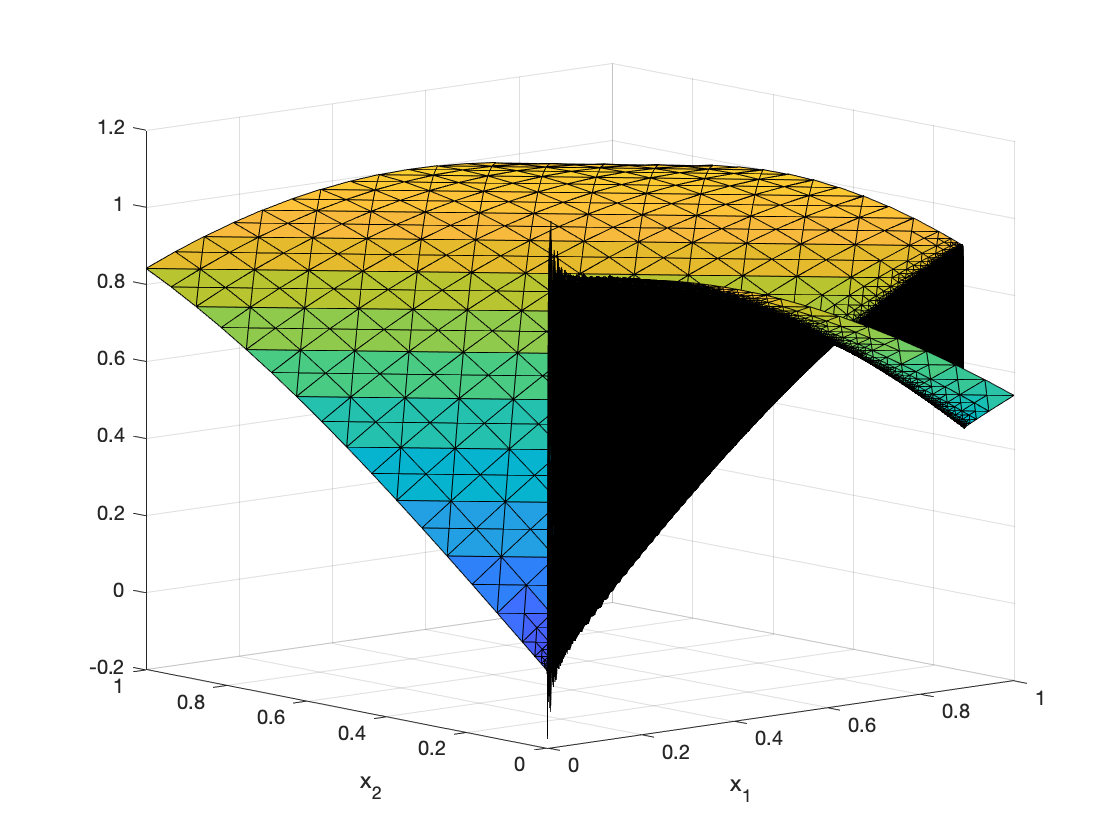}}
\caption{Piecewise smooth solution with a non-matching grid test problem: numerical solutions with RT0P0 and RT1P1-LSFEMs}
 \label{sol_pws_nm_afem}
\end{figure}

%\begin{figure}[!ht]
%	\label{error_pws_nm_uniform}
%    \begin{minipage}[!hbp]{0.33\linewidth}
%        \centering
%        \includegraphics[width=0.99\textwidth,angle=0]{convergence_uni_PWS_nonmatching_LSFEM}
%        \end{minipage}%
%    \begin{minipage}[!htbp]{0.33\linewidth} 
%       \centering
%        \includegraphics[width=1\textwidth,angle=0]{mesh_amr_smooth_nonmatching_LS}
%    \end{minipage}%
%    \begin{minipage}[!htbp]{0.33\linewidth}
%       \centering
%        \includegraphics[width=1\textwidth,angle=0]{convergence_PWS_nonmatching_LSFEM}
%    \end{minipage}%
%    \caption{Piecewise smooth solution with a non-matching grid test problem: convergence history on uniformly refined meshes(left), adaptive refined meshes after some iterations(center), convergence history on adaptive refined meshes(right)}
%\end{figure}

\subsection{Curved transport examples}
\subsubsection{Curved transport problem 1: zero-one example}
We consider an example similar to an example in 4.4.2 of \cite{Guermond:04}. Consider the problem on the half disk 
$\O = \{(x,y) \colon x^2+y^2<1; y>0\}$. Let the inflow boundary be $\{-1<x<0; y=0\}$. Choose the advection field 
$\bbeta = (\sin \theta, -\cos \theta)^T = (y/\sqrt{x^2+y^2}, -x/\sqrt{x^2+y^2})^T$, with $\theta$ being the polar angle.
Let $\gamma=0$, $f=0$, and the inflow condition and the exact solution be
$$
g = \left\{ \begin{array}{lll}
1 & \mbox{if}  &-1<x<-0.5, \\[2mm]
0 & \mbox{if}  &-0.5<x<0,
\end{array} \right.
\mbox{and}\quad
u = \left\{ \begin{array}{lll}
1 & \mbox{if  }  x^2+y^2 > 0.25, \\[2mm]
0 & \mbox{otherwise}. 
\end{array} \right.
$$ 

\begin{figure}[!htb]
\centering 
\subfigure[initial mesh]{ 
\includegraphics[width=0.45\linewidth]{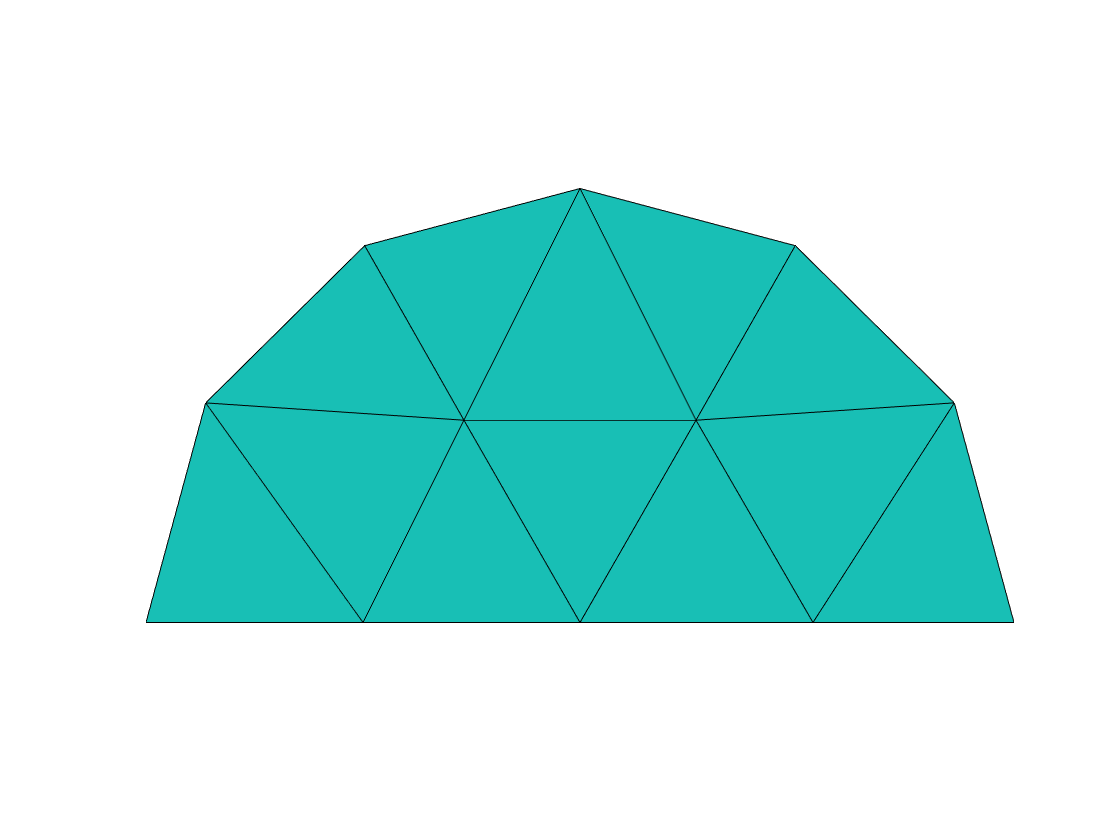}}
%\hspace{0.01\linewidth}
~
\subfigure[numerical solution (LSFEM) on an almost uniform mesh]{
\includegraphics[width=0.45\linewidth]{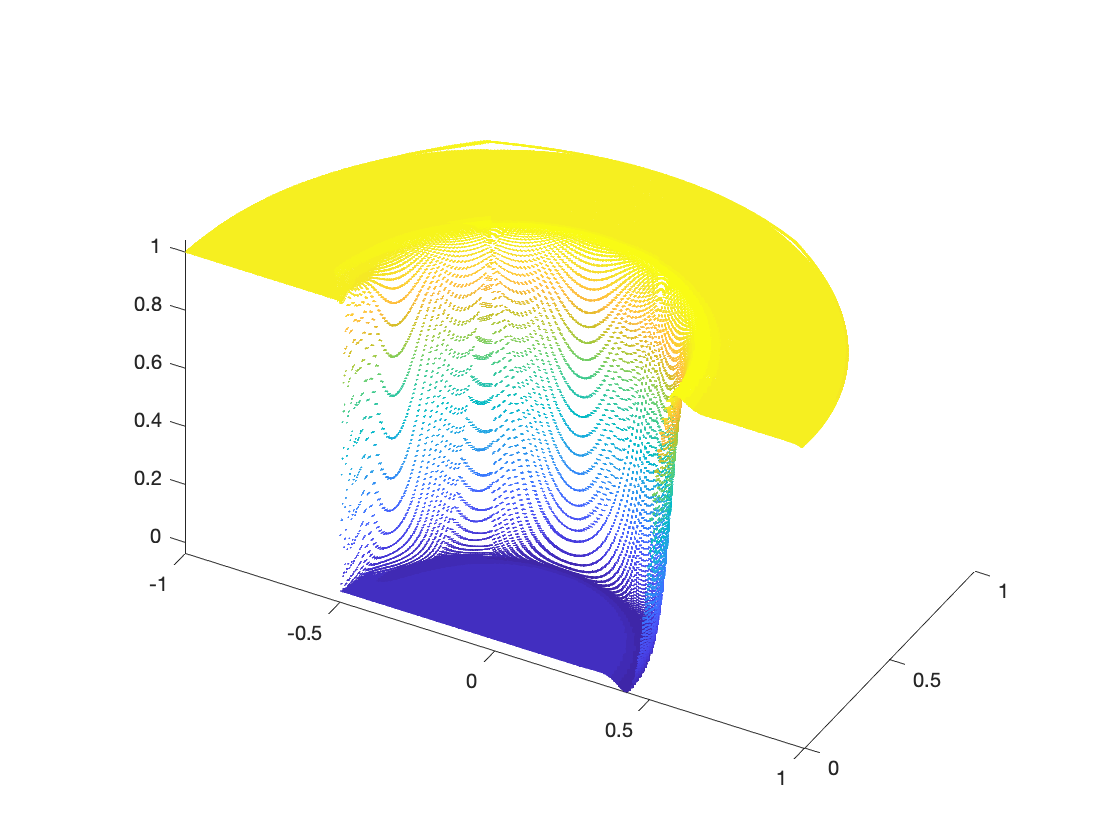}}
\caption{Curved transport problem 1}
 \label{curved_initialmesh}
\end{figure}

We choose an initial mesh to be as shown on the left of Fig. \ref{curved_initialmesh}. 
We choose the bottom central node to be $(0,0)$ and the node left of it  to be $(-0.5,0)$. 
So the inflow boundary mesh is matched with the inflow boundary condition. 
Since the advection field is curved and so is the discontinuity, 
the mesh will never be aligned with the discontinuity even after refinements.
Since the boundary is a half circle, when the mesh refinement is performed, an extra step is taken to map those boundary nodes to the right positions on the circle.

We show the numerical solution computed by LSFEM  on a mesh after 8 uniform refinements 
of the initial mesh 
on the right of Fig.  \ref{curved_initialmesh} (LSFEM-B solutions are similar). 
Small overshooting can be observed near the discontinuity. 
Along the radius, the solution is essentially one dimensional,
we project the graph of the solution onto the radius, see the left of Fig. \ref{curved_solutionuniform_projected}.
We do see the small under and overshooting.  
The maximum and minimum values of numerical solution $u_h$ are $1.0401$ and $-0.0381$, respectively.

With uniform refinements,
the convergence rate of the error in the least-squares norm is about $0.81$ 
and the rate of $\|u-u_h\|_0$ is about $0.25$, 
see the right of Fig. \ref{curved_solutionuniform_projected}. 
Since the mesh is not aligned with the discontinuity, 
the convergence order of the LS energy norm is smaller than $1$.

\begin{figure}[!htb]
\centering 
\subfigure[: projected numerical solutions on an almost uniform mesh]{ 
\includegraphics[width=0.45\linewidth]{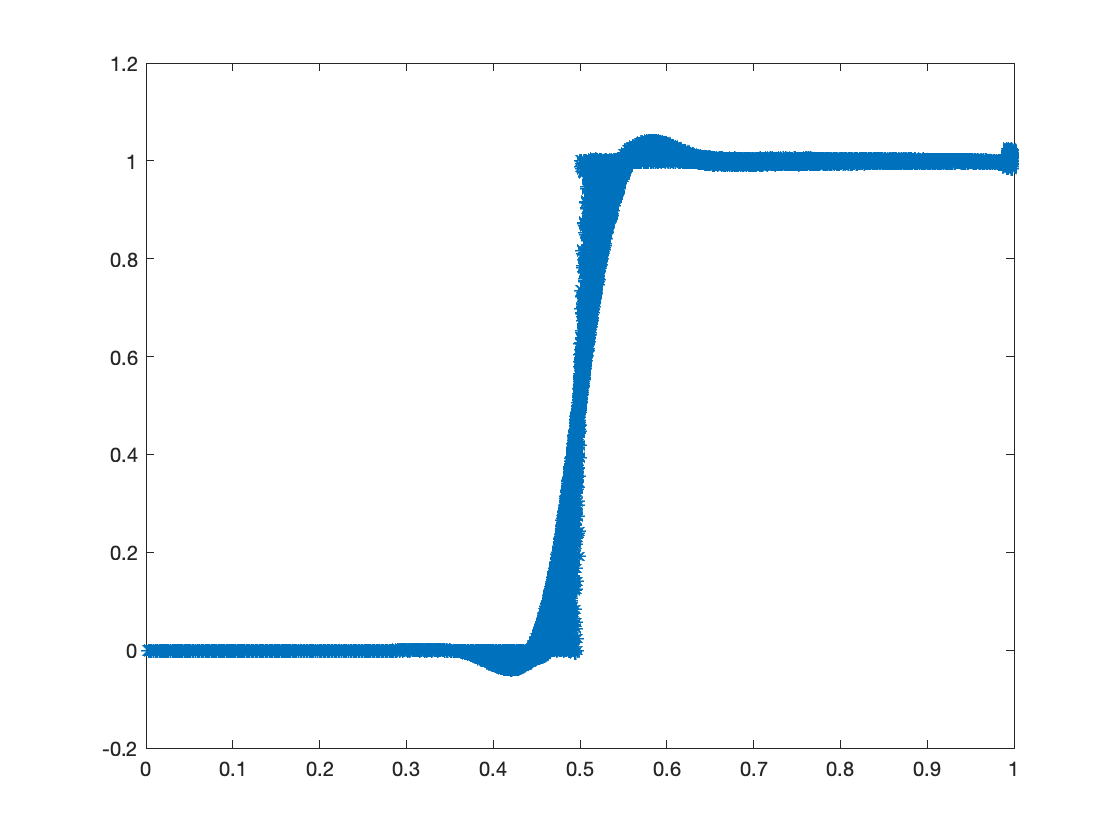}}
%\hspace{0.01\linewidth}
~
\subfigure[convergence history on uniform refined meshes]{
\includegraphics[width=0.45\linewidth]{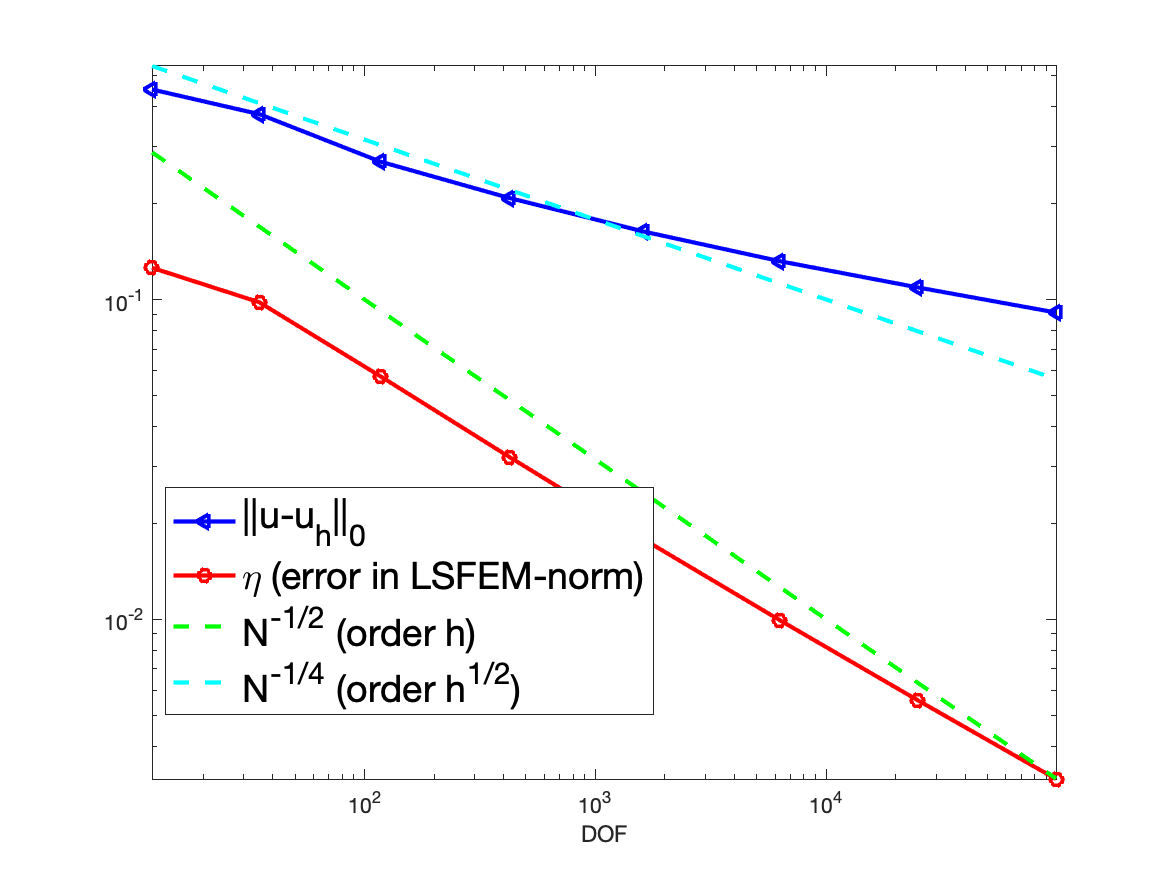}}
\caption{Curved transport problem 1}
 \label{curved_solutionuniform_projected}
\end{figure}

On the left of Fig. \ref{curved_os}, we show the adaptive mesh generated by LSFEM after several iterations. We see many refinements along the discontinuity which is very natural. Also, almost uniform refinements can be found in the half ring where $u=1$. The reason is that even $u$ is a constant $1$, the flux $\bsigma = \bbeta$ is not a constant vector and has approximation errors. On the other hand, in the region where $u=0$, the flux is also a zero vector and can be exactly computed. So no refinement is needed in the inner half circle. 

On the right of  Fig. \ref{curved_os}, we show the convergence history of the adaptive method. With adaptive refinements, the convergence order of the error in the LS norm is about $1$ and is optimal, and the rate of $\|u-u_h\|_0$ is about $0.5$.

\begin{figure}[!htb]
\centering 
\subfigure[an adaptive refined mesh]{ 
\includegraphics[width=0.45\linewidth]{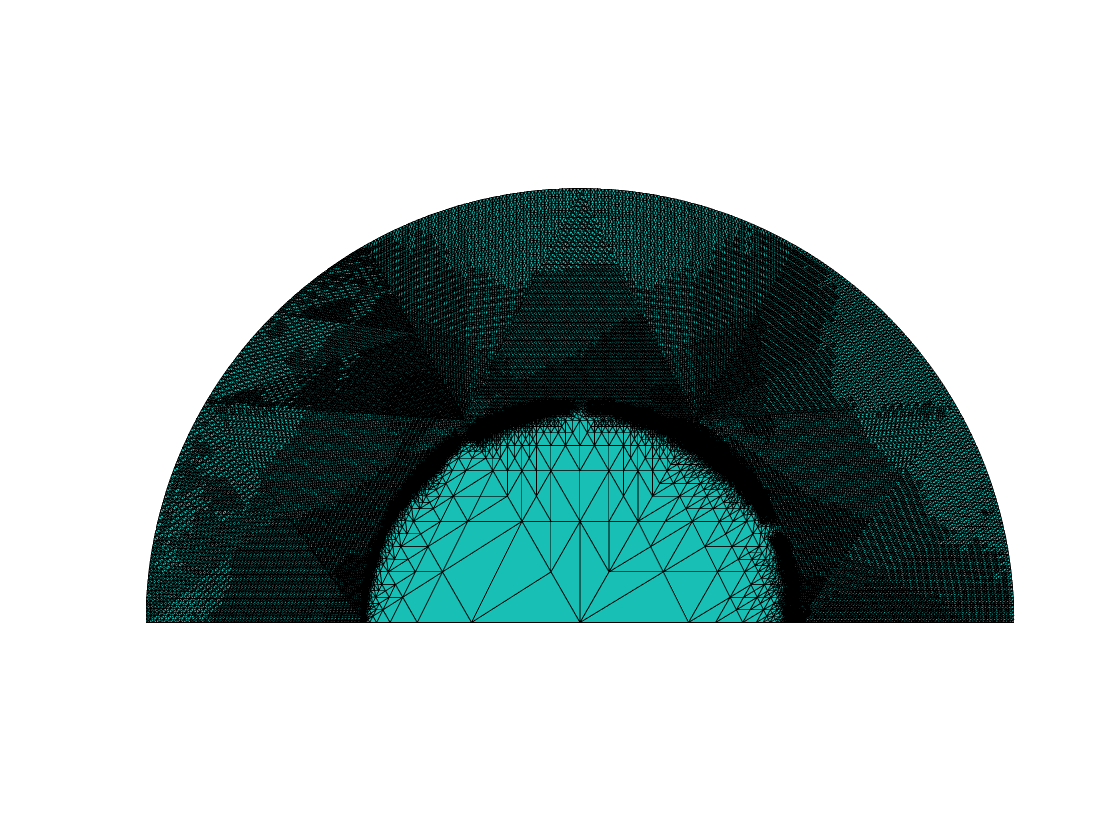}}
%\hspace{0.01\linewidth}
~
\subfigure[convergence history]{
\includegraphics[width=0.45\linewidth]{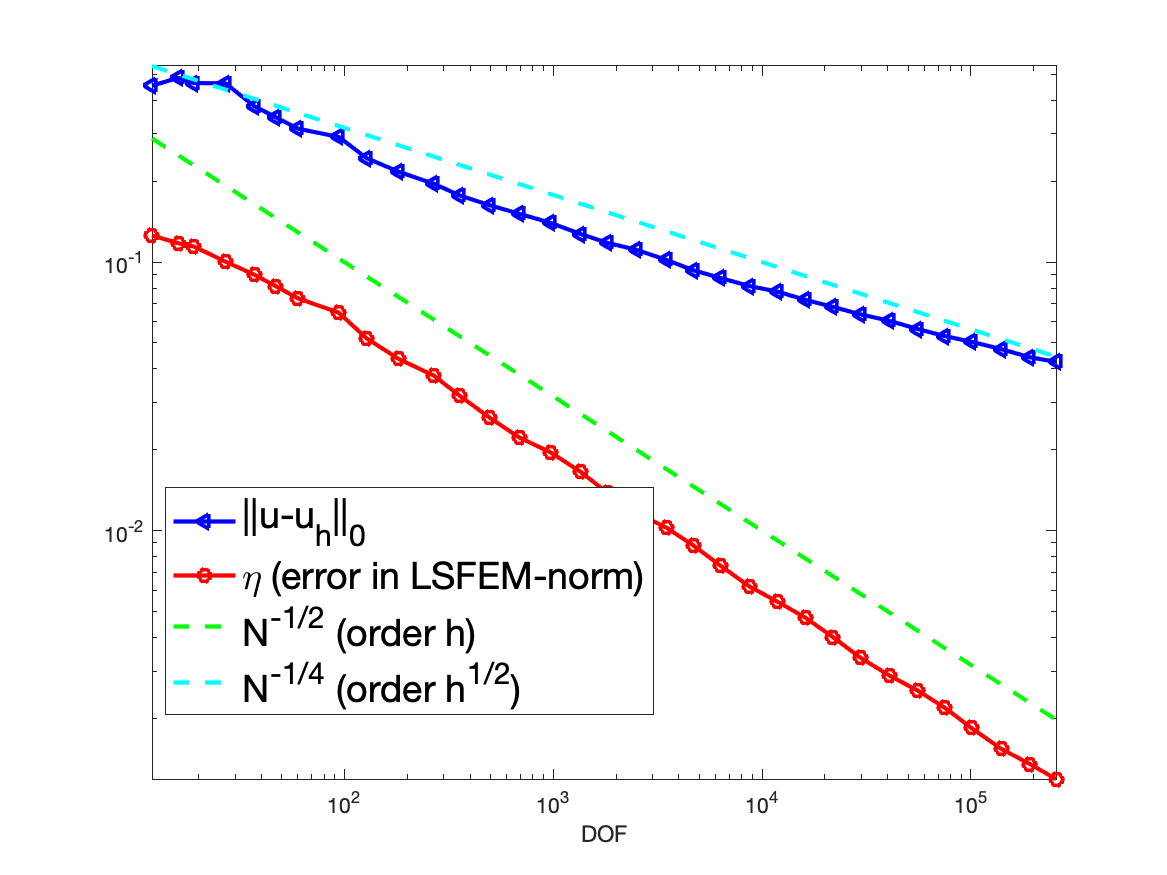}}
\caption{Curved transport problem 1 with adaptive LSFEM}
 \label{curved_os}
\end{figure}

On the left of Fig. \ref{curved_adaptivemesh}, we show the reduction of overshooting values of the $RT_0\times P_0$ LSFEM solution. After the initial stages, the overshooting values is decreasing with refined meshes along the discontinuity (although not strictly monotonically).

On the right of Fig. \ref{curved_adaptivemesh}, the projected solution is shown on the final mesh. We can see that the overshooting is very small compared with the uniform refinements. Thus the Gibbs phenomena is not observed.

\begin{figure}[!htb]
\centering 
\subfigure[reduction of overshooting]{ 
\includegraphics[width=0.45\linewidth]{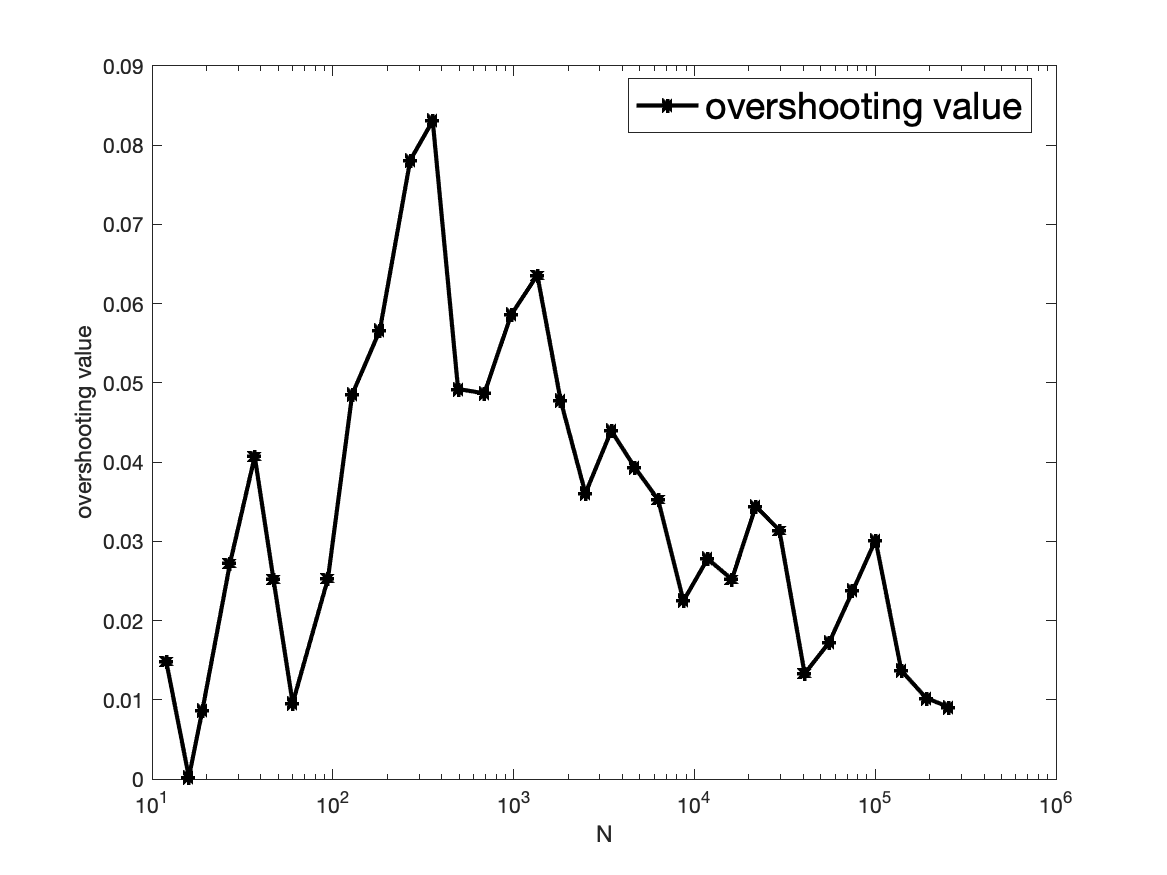}}
%\hspace{0.01\linewidth}
~
\subfigure[projected solution]{
\includegraphics[width=0.45\linewidth]{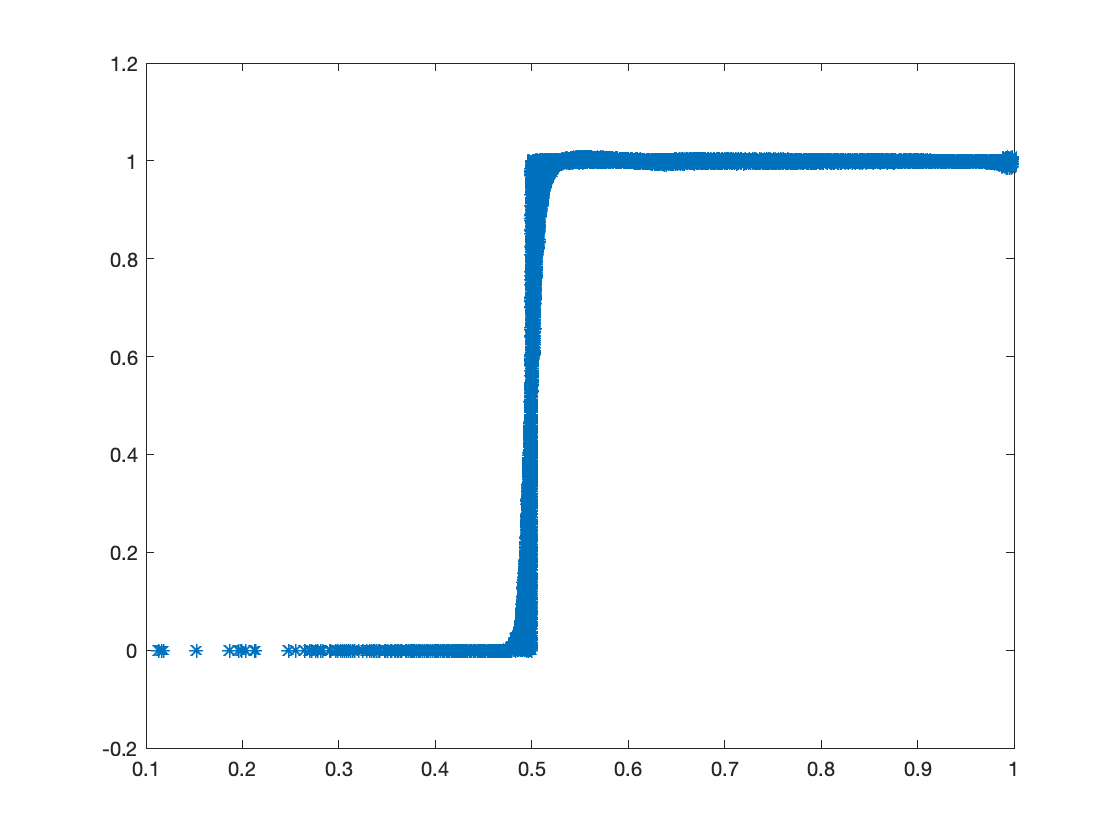}}
\caption{Curved transport problem 1 with adaptive LSFEM}
 \label{curved_adaptivemesh}
\end{figure}

If we choose $\a_F =1$ in the LSFEM-B2 formulation, the numerical computation is not right for this problem. On the left of Fig. \ref{curved_bad_1}, the refined mesh generated by LSFEM-B2 and error estimator $\xi$ is shown. Many unnecessary refinements along the inflow boundary are seen. On the right of Fig. \ref{curved_bad_1}, we show the convergence histories. For the error measured in the LS norm $\tri \cdot \tri_B$ the order is optimal, but $\|u-u_h\|_0$ is not decreasing.  On Fig. \ref{curved_bad_2}, the numerical solution and its projected version are shown. It is very clear the solution is not accurate under this mesh and LSFEM-B with $\a_F=1$.

These all suggest that if we simply choose $\a_F =1$ in LSFEM-B2, the $\tri \cdot \tri_B$ norm is not well balanced, the weight on the boundary term is too weak. The choice $\a_F=10$ is big enough to have enough boundary weight.

\begin{figure}[!htb]
\centering 
\subfigure[refined mesh generated by LSFEM-B2 ]{ 
\includegraphics[width=0.45\linewidth]{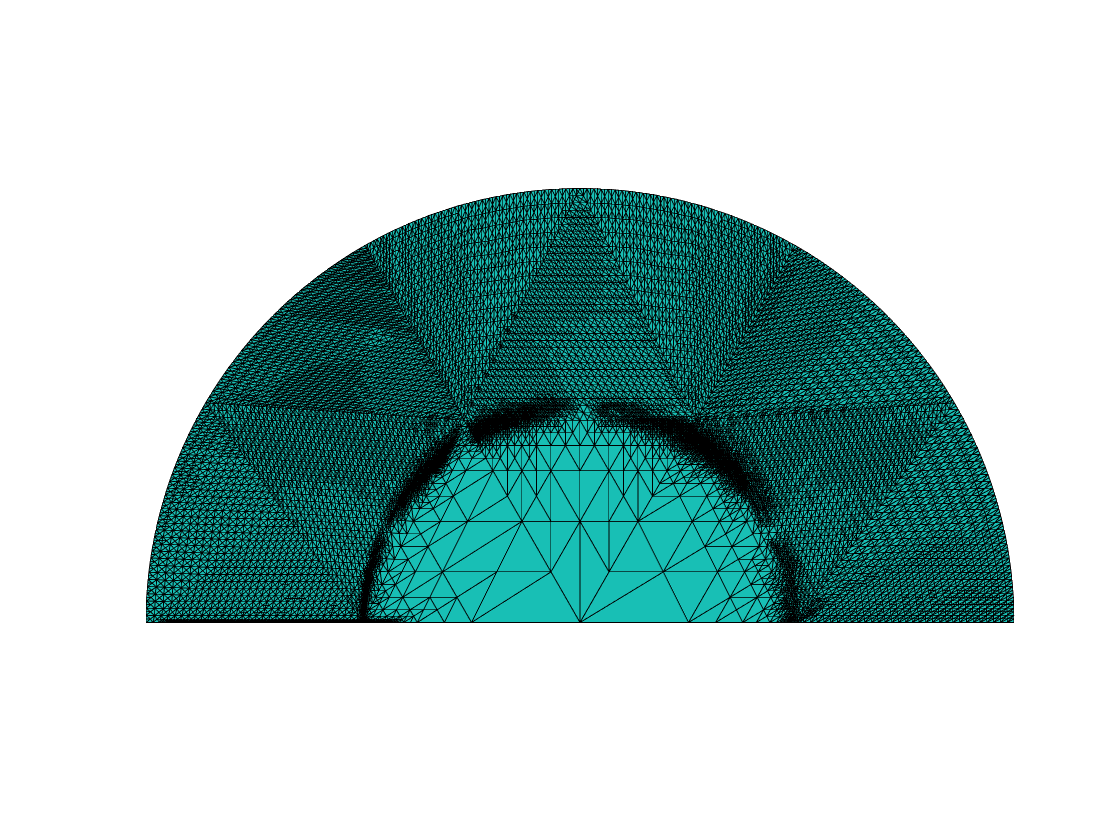}}
%\hspace{0.01\linewidth}
~
\subfigure[convergence history]{
\includegraphics[width=0.45\linewidth]{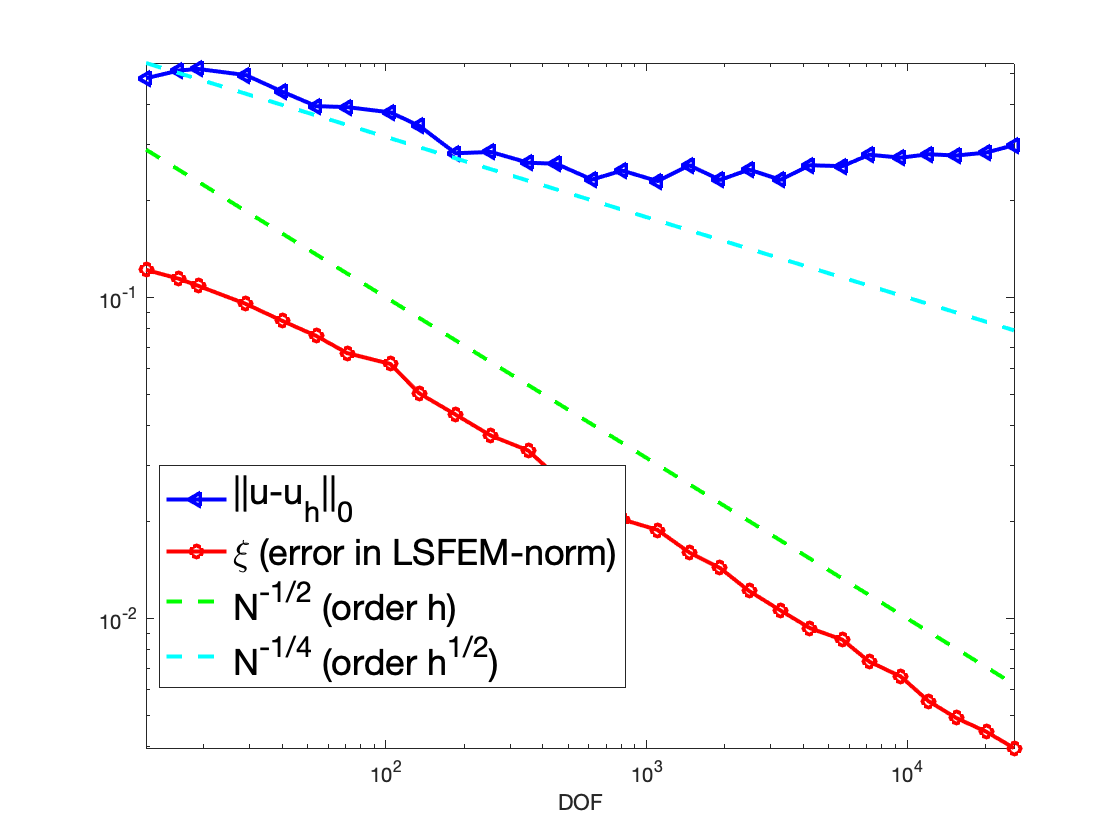}}
\caption{Curved transport problem 1: LSFEM-B2 with $\a_F=1$}
 \label{curved_bad_1}
\end{figure}

\begin{figure}[!htb]
\centering 
\subfigure[numerical solution]{ 
\includegraphics[width=0.45\linewidth]{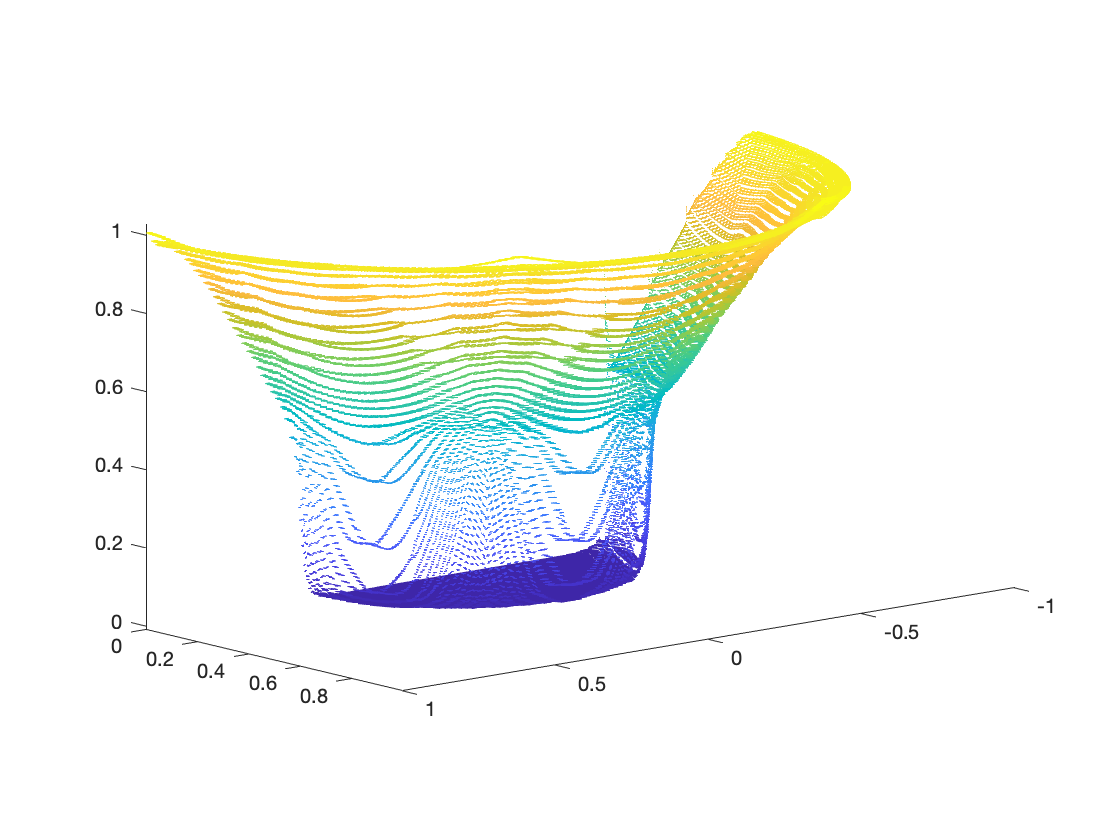}}
%\hspace{0.01\linewidth}
~
\subfigure[projected solution]{
\includegraphics[width=0.45\linewidth]{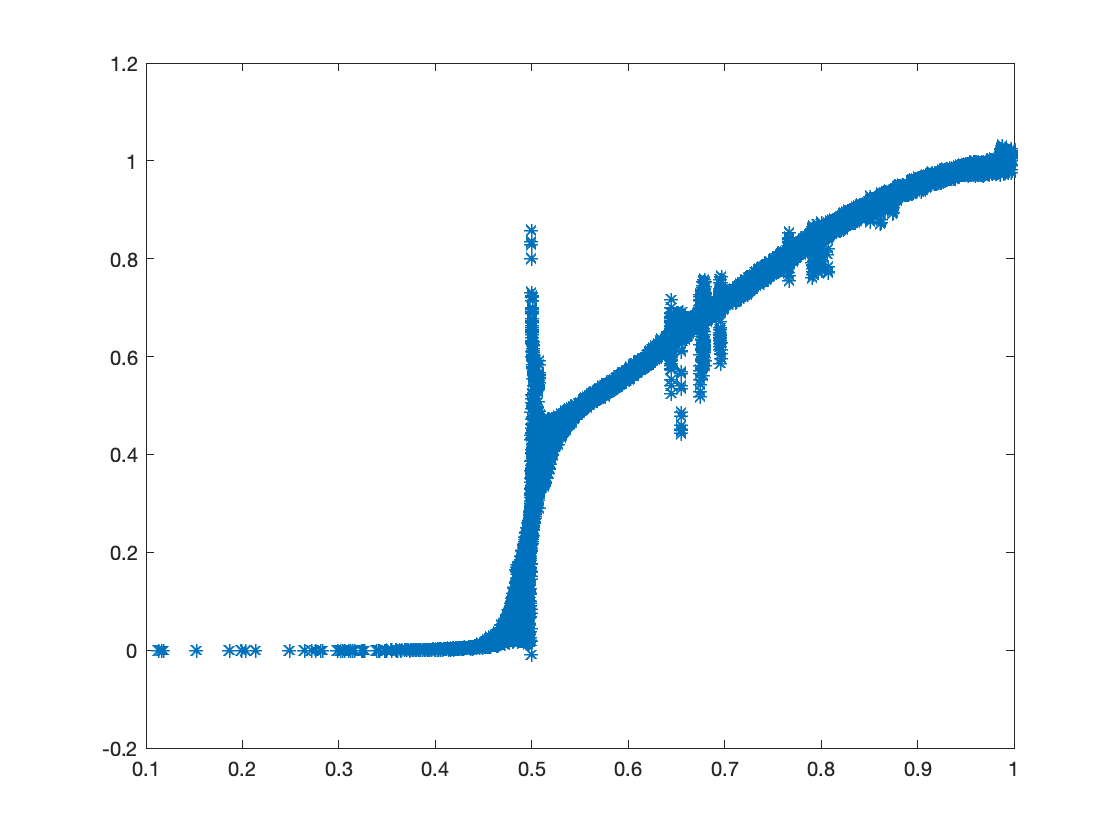}}
\caption{Curved transport problem 1: LSFEM-B2 with $\a_F=1$}
 \label{curved_bad_2}
\end{figure}

\subsubsection{Curved transport problem 2: negative-one-one example} We modify the previous example by letting the inflow condition and the exact solution be
$$
g = \left\{ \begin{array}{lll}
1 & \mbox{if}  &-1<x<-0.5, \\[2mm]
-1 & \mbox{if}  &-0.5<x<0,
\end{array} \right.
\mbox{and}\quad
u = \left\{ \begin{array}{lll}
1 & \mbox{if  }  x^2+y^2 > 0.25, \\[2mm]
-1 & \mbox{otherwise}. 
\end{array} \right.
$$
Note that even the solution $u$ in the inner half disk $\{x^2+y^2<0.25, y>0\}$ is still a constant vector, the flux $\bsigma = \bbeta u = -\bbeta$ is not. At the origin 
$(0,0)$, the flux is singular, so it is expected that there are many refinements around the 
origin. 
%
%For simplicity, we only present the results from LSFEM, 
%the results from LSFEM-B1 and LSFEM-B2 are similar.

The left of  Fig. \ref{curved_adaptivemesh_11_eta} is a refined mesh. 
It is clear that the mesh is refined around the origin and the discontinuities. 
The right of Fig. \ref{curved_adaptivemesh_11_eta} shows the convergence history.
The order of LS energy norm is $1$ and that of $\|u-u_h\|_0$ is $1/2$.

\begin{figure}[!htb]
\centering 
\subfigure[a refined mesh]{ 
\includegraphics[width=0.45\linewidth]{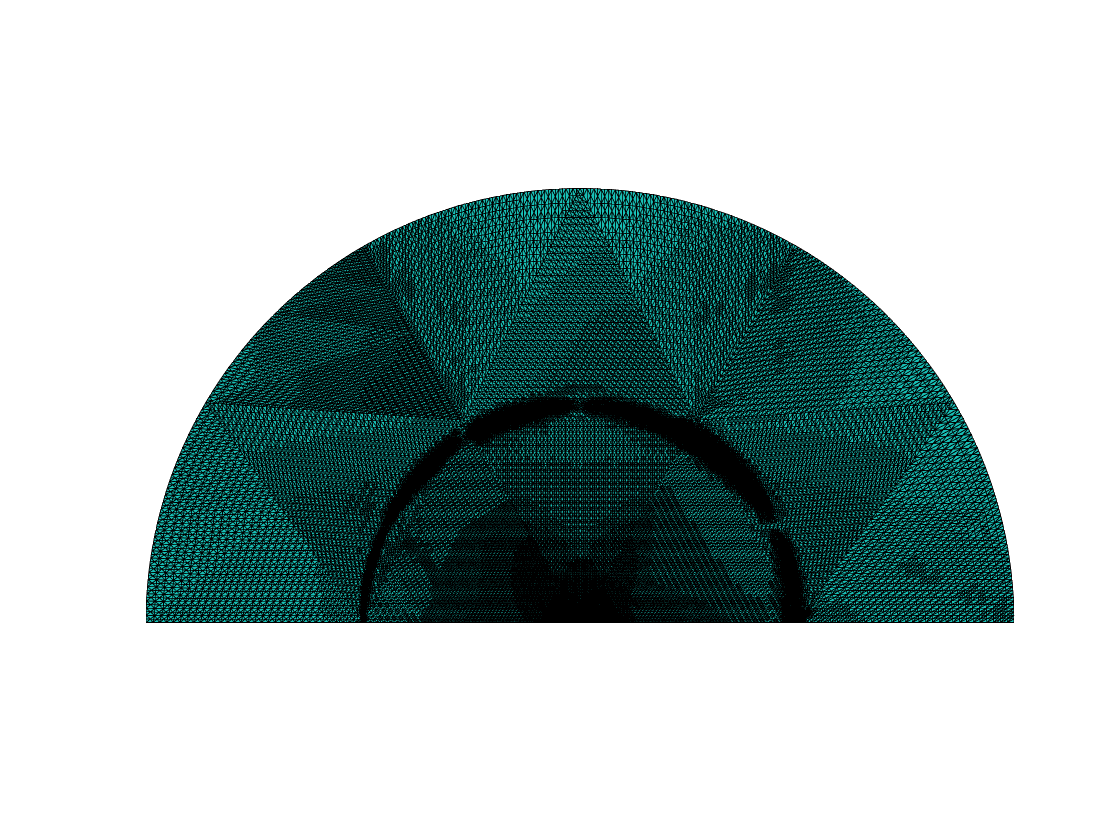}}
%\hspace{0.01\linewidth}
~
\subfigure[adaptive convergence history]{
\includegraphics[width=0.45\linewidth]{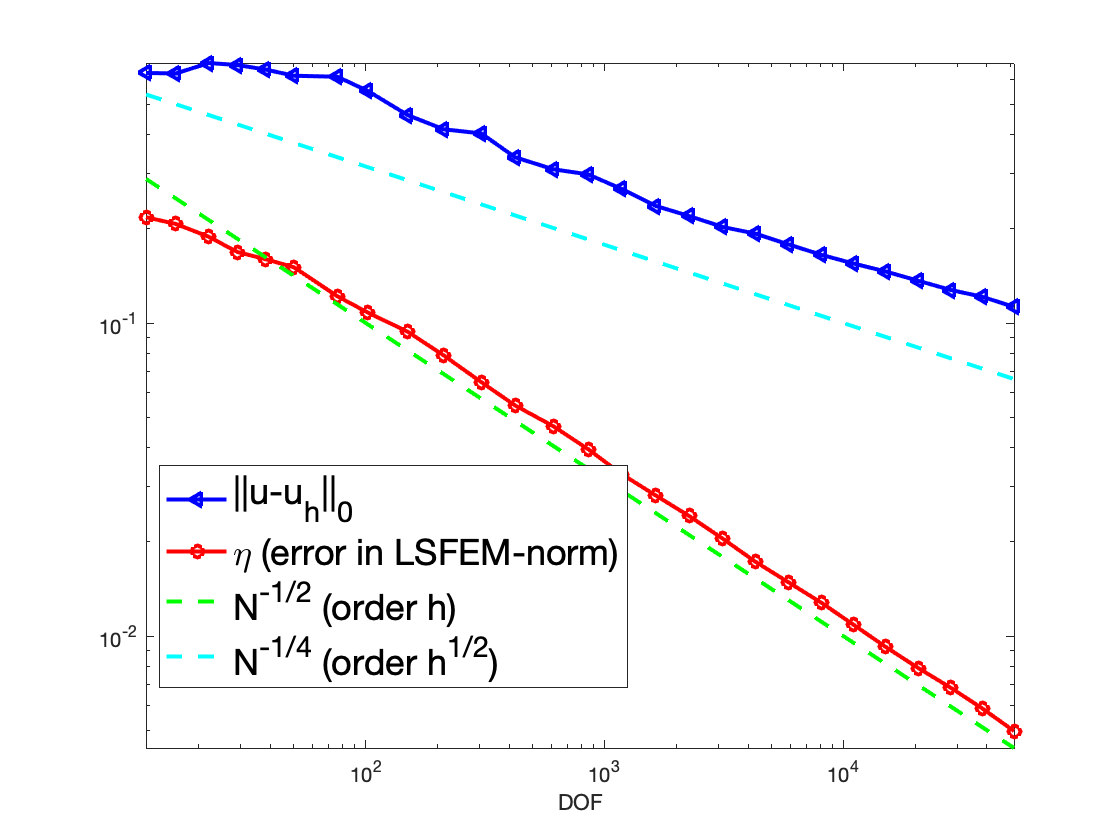}}
\caption{Curved transport problem 2}
 \label{curved_adaptivemesh_11_eta}
\end{figure}

%In Fig. \ref{curved_solution_11}, we present the numerical solutions. It is clear that 
%there is a singularity at the origin, while the solution near the discontinuity behaves 
%similarly as the previous $0$-$1$ example.
%
%\begin{figure}[!ht]
%    \label{curved_solution_11}
%    \begin{minipage}[!hbp]{0.48\linewidth}
%        \includegraphics[width=0.99\textwidth,angle=0]{curved_projectedsolution_11_eta}
%        \end{minipage}%
%        \quad
%    \begin{minipage}[!htbp]{0.48\linewidth}
%        \includegraphics[width=0.99\textwidth,angle=0]{curved_solution_11_more_eta}
%        \end{minipage}
%        \caption{Numerical solution for the $-1$\&$1$ curved transport problem on adaptively refined meshes (projected on the left and non-projected on the right).}%
%\end{figure}

\subsection{A smooth example with a sharp transient layer}
Consider the following problem: $\O = (0,1)^2$, $\gamma=0.1$, $f=0$, and $\bbeta = (y+1,-x)^T/\sqrt{x^2+(y+1)^2}$. The inflow boundary is $\{x=1, y\in (0,1)\} \cup \{x\in (0,1), y=0\}$, i.e., the west and north boundaries of the domain.  Choose $g$ such that the exact solution $u$ is
$$
u = \dfrac{1}{4}\exp\left(\gamma r\arcsin \left(\dfrac{y+1}{r}\right)\right) 
\arctan \left(\dfrac{r-1.5}{\epsilon}\right), \mbox{ with } r = \sqrt{x^2+(y+1)^2}.
$$

When $\epsilon = 0.01$, the layer can be fully resolved, see the left of Fig. \ref{burman_sol}. When $\epsilon = 10^{-10}$, the layer is never fully resolved in our experiments and can be viewed as discontinuous, see the right of Fig. \ref{burman_sol}. 

\begin{figure}[!htb]
\centering 
\subfigure[$\epsilon = 10^{-2}$]{ 
\includegraphics[width=0.45\linewidth]{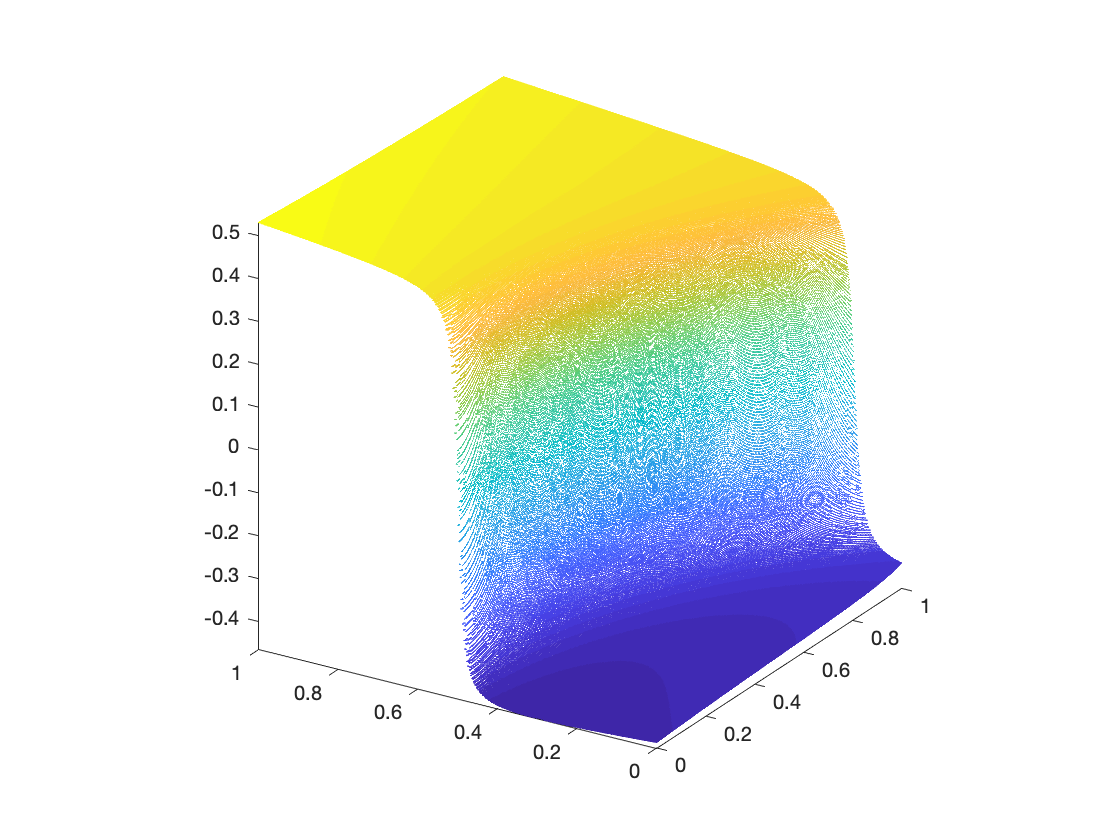}}
%\hspace{0.01\linewidth}
~
\subfigure[$\epsilon = 10^{-10}$]{
\includegraphics[width=0.45\linewidth]{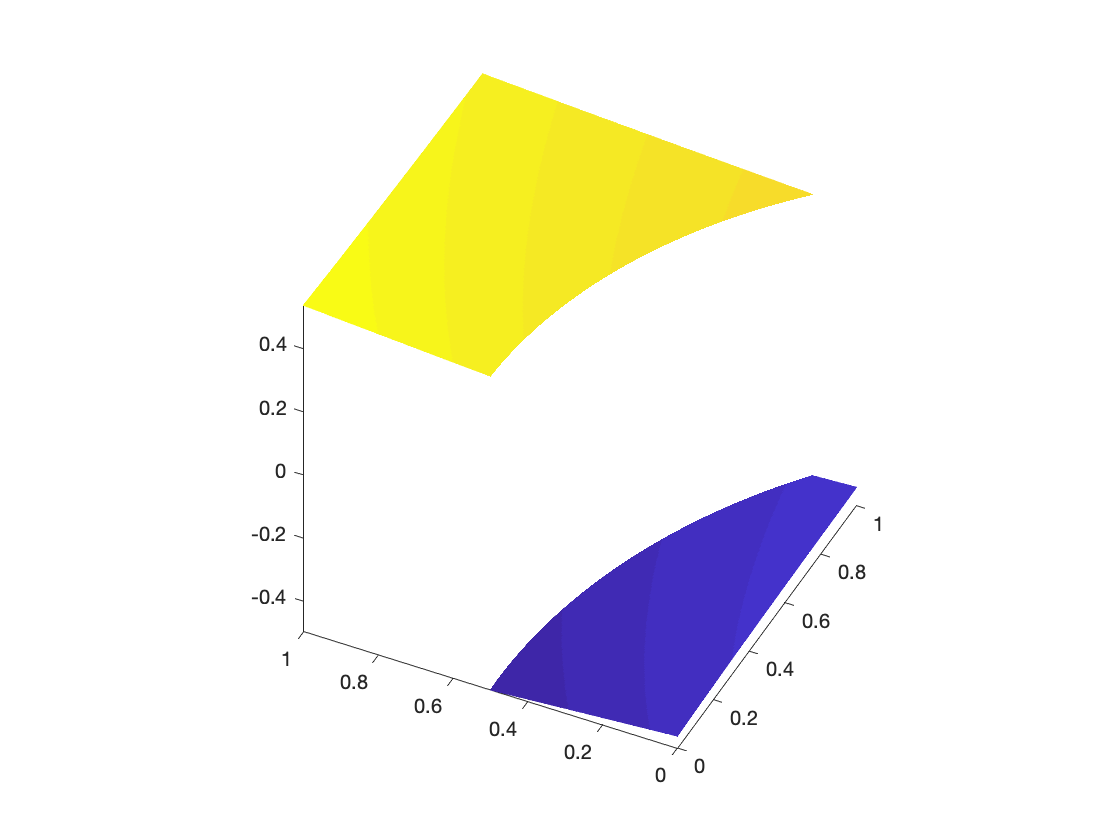}}
\caption{Transient layer problem: exact solutions}
 \label{burman_sol}
\end{figure}

\begin{figure}[!htb]
\centering 
\subfigure[a refined mesh]{ 
\includegraphics[width=0.3\linewidth]{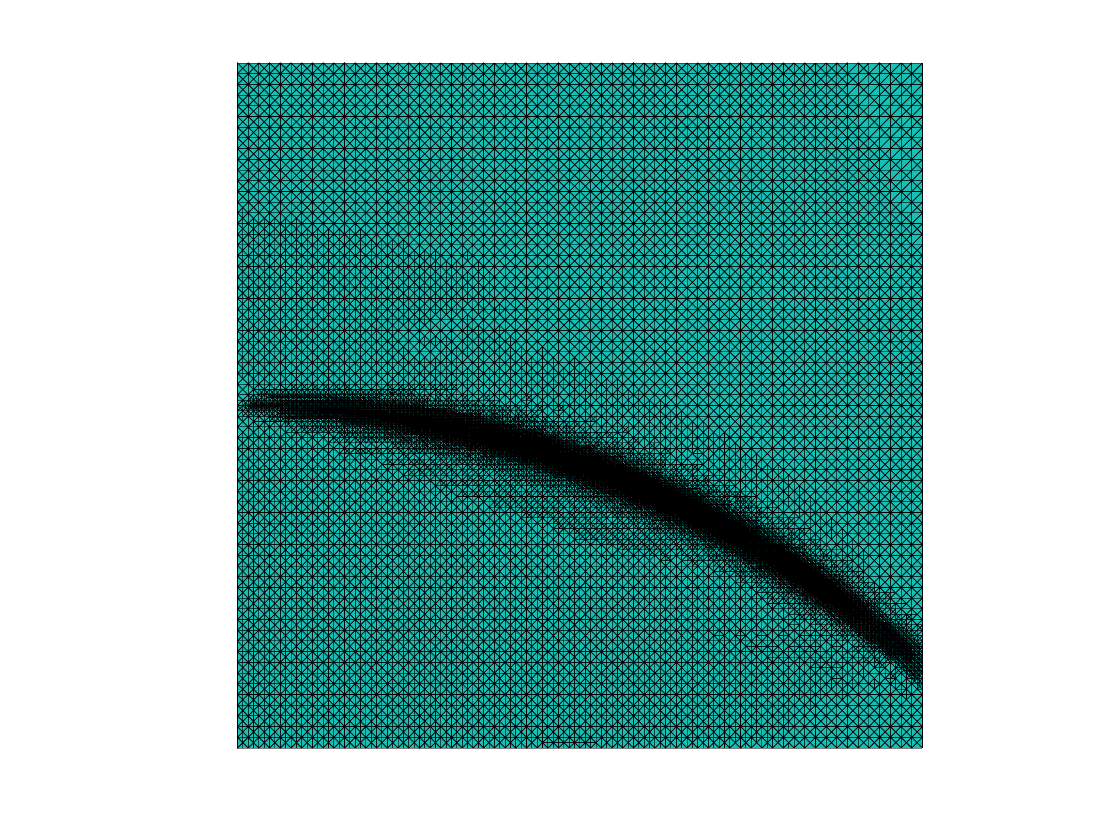}}
%\hspace{0.01\linewidth}
~
\subfigure[convergence history]{
\includegraphics[width=0.3\linewidth]{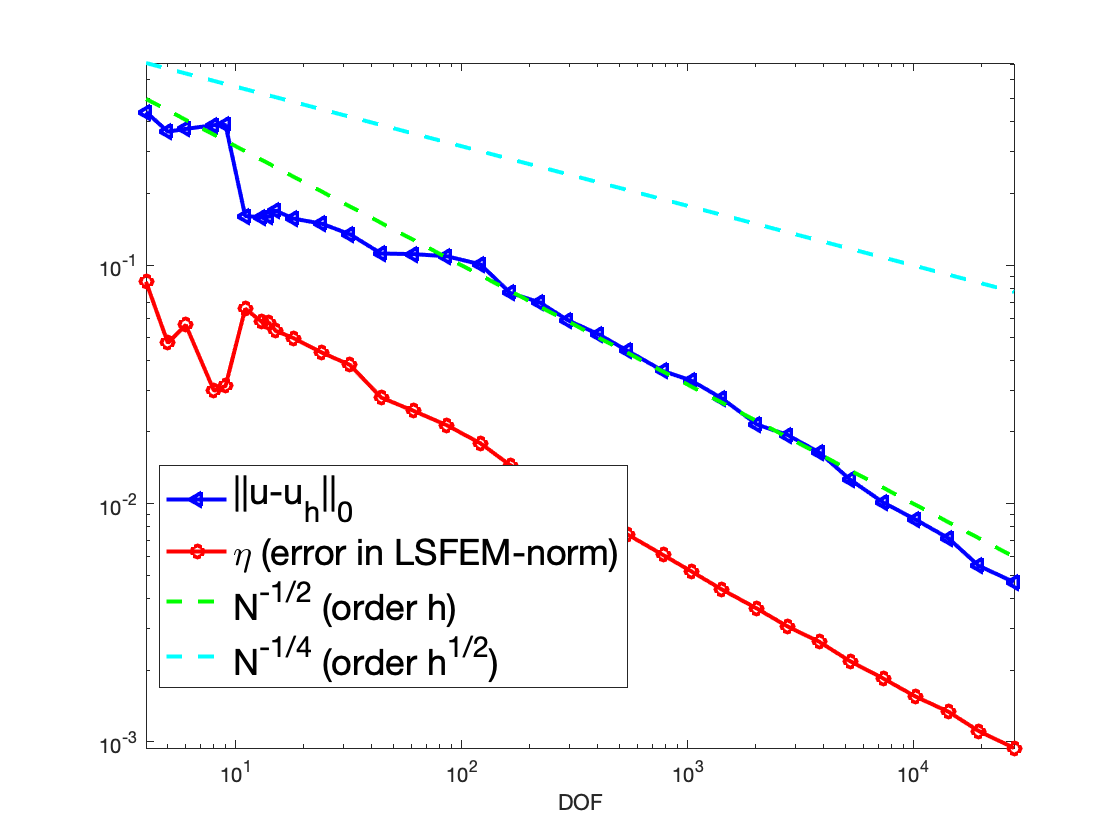}}
~
\subfigure[contours of solution]{
\includegraphics[width=0.3\linewidth]{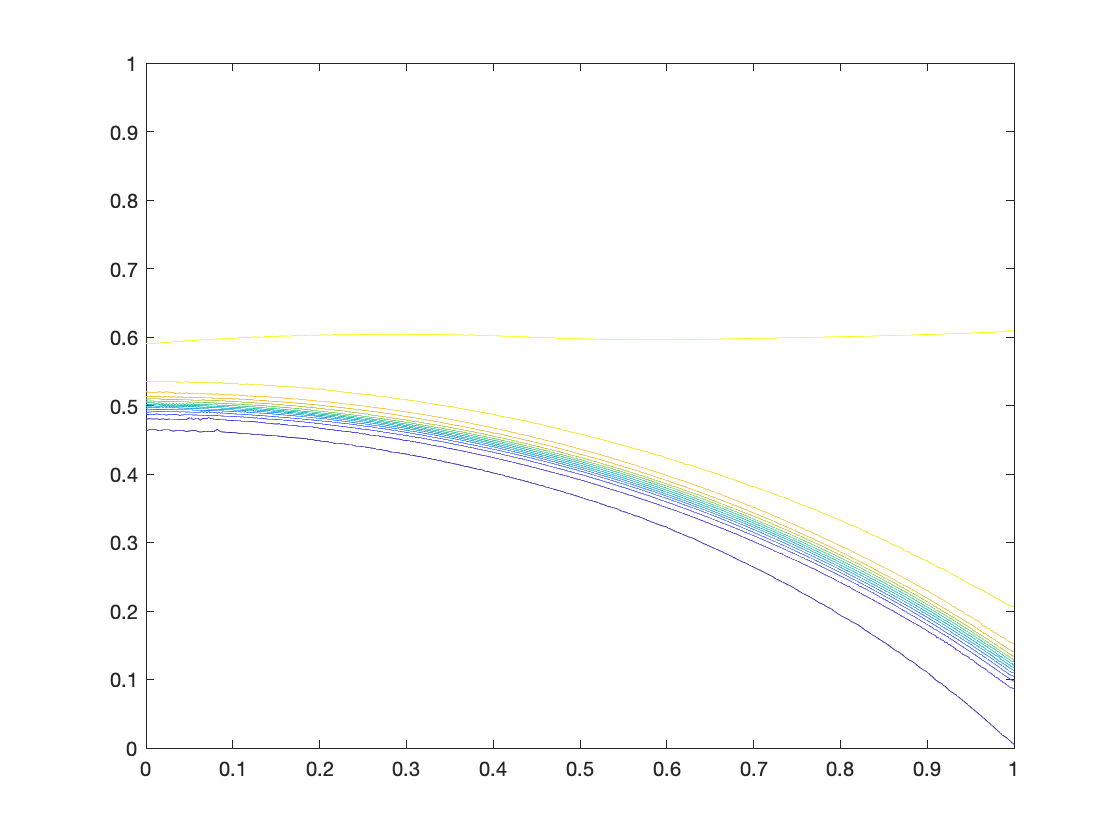}}
\caption{Transient layer problem: $\epsilon = 10^{-2}$}
 \label{burman_mesh_1e2}
\end{figure}

When $\epsilon = 0.01$, we show the numerical results in  Fig. \ref{burman_mesh_1e2}. The behaviors of the methods are very similar to the global continuous solution case. When $\epsilon = 10^{-10}$, we show the numerical results in  Fig. \ref{burman_mesh_1e10}. The behaviors of the methods are very similar to the  piecewise smooth solution with non-matching grid case, the example 7.8. The order of convergence of $\|u-u_h\|_0$ is about $0.12$. The contour of the solution on the right of Fig. \ref{burman_mesh_1e10} shows that the overshooting is neglectable when the mesh is fine enough. 

\begin{figure}[!htb]
\centering 
\subfigure[a refined mesh]{ 
\includegraphics[width=0.3\linewidth]{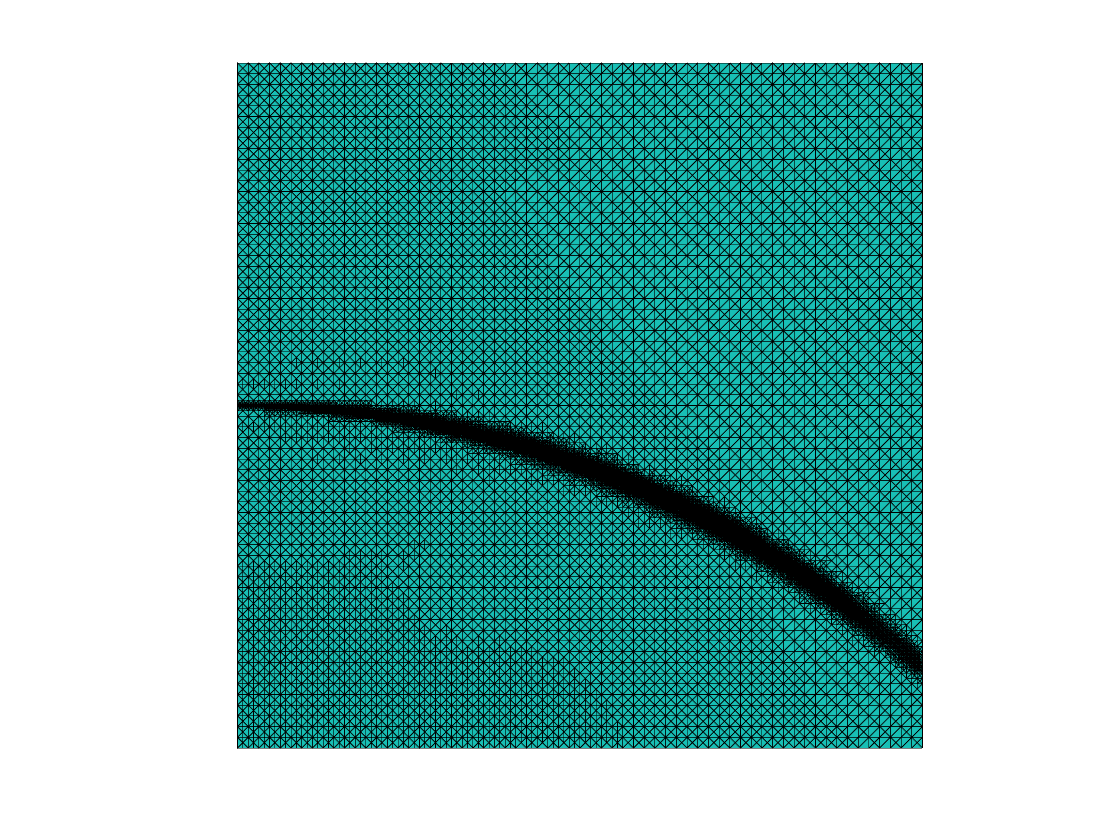}}
%\hspace{0.01\linewidth}
~
\subfigure[convergence history]{
\includegraphics[width=0.3\linewidth]{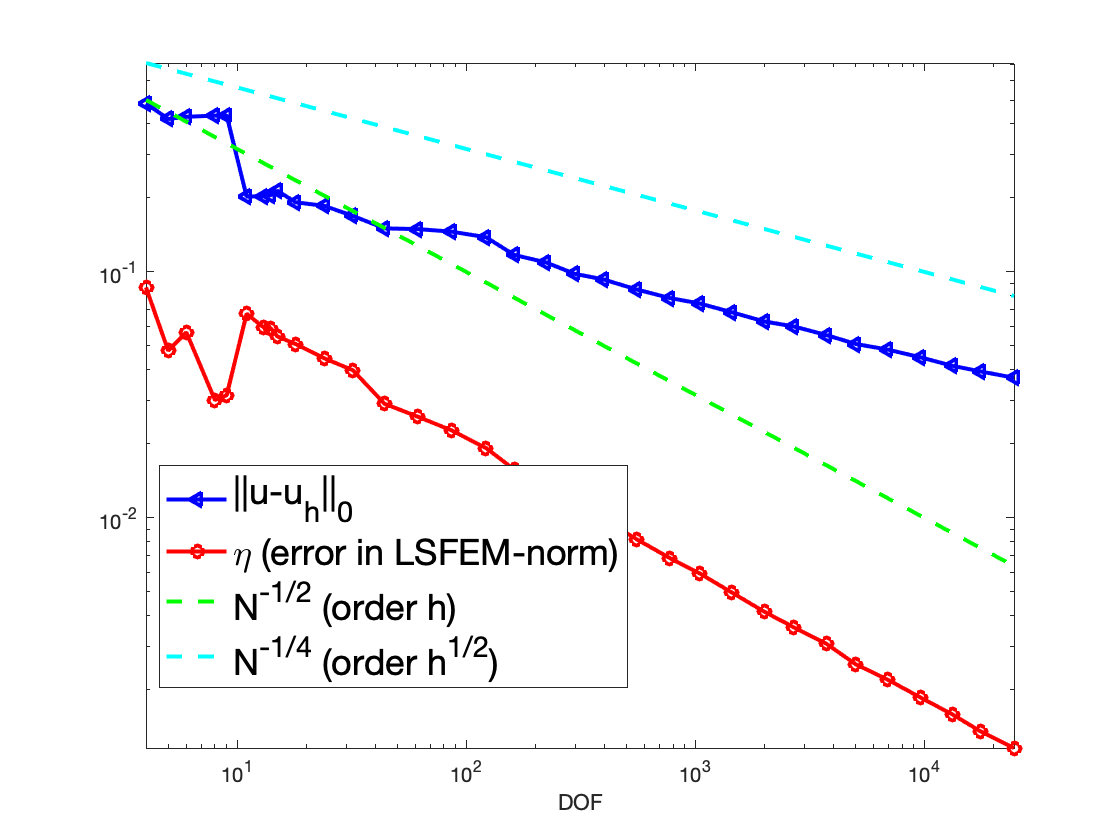}}
~
\subfigure[contours of solution]{
\includegraphics[width=0.3\linewidth]{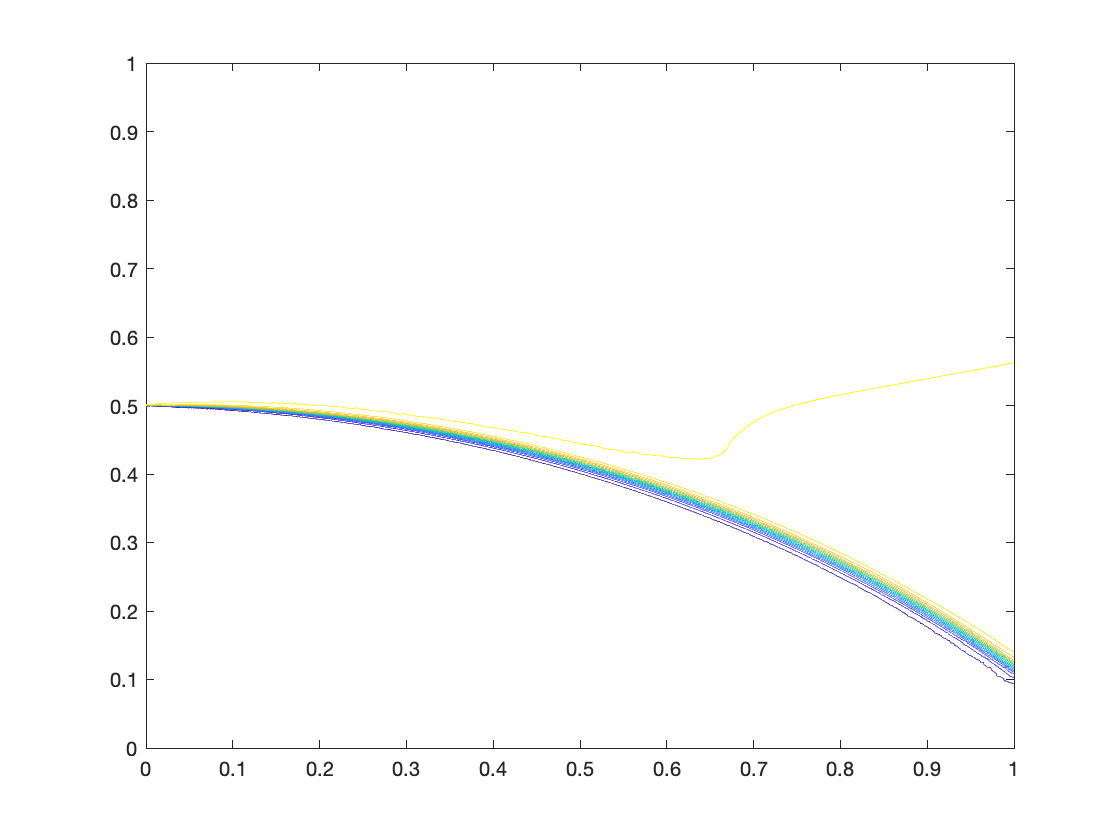}}
\caption{Transient layer problem: $\epsilon = 10^{-10}$}
 \label{burman_mesh_1e10}
\end{figure}

\subsection{General comments about the numerical experiments}
In our numerical tests, we found that all three flux-based formulations have almost identical results. 
%Only a small difference
%of the three formulations can be found in the Fig. \ref{curved_os}. 
The half-order loss on the inflow boundary elements for LSFEM-B1 is neglectable/non-observable, thus for the methods with weakly enforced inflow boundary conditions, we prefer LSFEM-B1 over LSFEM-B2, since for LSFEM-B2, the current choice of the weight $\a_F =10$ is based on our numerical experience only, we do face the possibility of a too small choice to make the boundary condition too weakly enforced and  the whole method unbalanced.

%When the solution is globally smooth, the convergence rates for both the LS norms and $\|u-u_h\|_0$ are of optimal order $1$ even with uniform refinements. When the solution is only piecewisely smooth, if the mesh is aligned with discontinuity, with uniform refinements,  the convergence rate for the LS norms is $1$, while the rate for $\|u-u_h\|_0$ is between $0.5$ and $1$. The difference of order between global smooth and piecewise smooth solutions suggests that a mesh-independent norm equivalence between LS norms and the standard norms of $u$ and $\bsigma$ does not hold, or at least, we should discuss this equivalence for these two different situations. When adaptive refinements are used, for non-aligned grids, we can have optimal convergence order $1$ in LS norms. For $\|u-u_h\|_0$, we can have order $0.5$, which is a half-order less than the LS norms.

Compared with the C-LSFEM where continuous approximation is used, the flux based LSFEMs is much better on handling the discontinuous boundary condition and discontinuous solution on matched and non-matched meshes.

For the case that the discontinuity is not aligned with the mesh, the numerical tests show that the adaptive $RT_0\times P_0$ LSFEMs have very small overshooting with adaptive refined meshes, similar to the discussion in \cite{Zhang:19}.

A very common folklore of the least-squares method is it tends to have a strong smearing effect. We should point out here this least-squares method often refer to Galerkin least-squares or stabilized methods where some least-squares terms are added to variational problems, see for example \cite{HFH:88}. For the bona fide least-squares methods developed in this paper, we do not observe smearing effect.

\section{Concluding Remarks}
\setcounter{equation}{0}
In this paper, several LSFEMs  for the linear hyperbolic transport problem are developed based on the flux reformulation of the problem. The new methods can separate two continuity requirements of the solution with the flux in $H(\divvr)$ and the solution in $L^2$. Thus, simple and natural $H(\divvr)\times L^2$ conforming finite element spaces can be used to approximate the flux and solution. Several variants of the methods are developed to handle the inflow boundary condition strongly or weakly. With the reformulation, the least-squares finite element methods can handle discontinuous solutions much better than the traditional continuous polynomial approximations. With least-squares functionals as a posteriori error estimators, the adaptive methods can naturally identify error sources including singularity and non-matching discontinuity. The flux-based LSFEMs with the lowest $RT_0\times P_0$ approximation have neglectable overshooting phenomenon with adaptive methods. Existence, uniqueness, a priori and a posteriori error estimates are established for the proposed methods. Extensive numerical tests are done to show the effectiveness of the methods developed in the paper.

There are several future research directions. The first is a flux-reformulated LSFEM based on $L^1$-minimization similar to that of \cite{Guermond:04}. With the $L^1$-minimization, the method have potential to handle the discontinuity better with smaller overshooting effects. Flux-reformulated LSFEMs based on adaptively weighted $L^2$ norms can also be developed to handle the discontinuity better \cite{Jiang:98,BG:16}. New algorithms are needed to combine the mesh and weight adaptivities. 

%Another interesting and important direction is an adaptive $hp$ flux-reformulated LSFEMs. The lowest order approximation is only used on those elements where the discontinuity is inside the element. Higher oder approximation can be used in those smooth regions.  

%It is also very natural to generalize the flux-reformulated LSFEMs to neutron transport equations \cite{MR:98,MRS:00}. With the flux reformulation, the methods have potentials to handle rougher solutions.

One of the advantages of the discontinuous Galerkin method is that the system can be solved by successive elimination starting from the inflow boundary, which makes the method semi-explicit, see \cite{RH:73,Johnson:87}. Modifying our methods to develop a similar implementation is an on-going work, and we will apply these methods to the time-dependent problems.

It is always more changeling when apply numerical methods to nonlinear problems. In \cite{DMMO:05}, flux-reformulated LSFEMs are already suggests for the Burgers equation. But there are many open questions left, for example, how to ensure the numerical solution is the physical meaningful solution, what is the right continuous and discrete space settings, and how to guarantee the existence and uniqueness of the numerical solution? Developing LSFEMs that can answer these questions is also one of our ongoing work.

\section*{Acknowledgements}
% Acknowledgements text here
S. Zhang is supported in part by Hong Kong Research Grants Council under the GRF Grant Project No. 11305319, CityU and a China Sichuan Provincial Science and Technology Research Grant 2018JY0187 via Chengdu Research Institute of City University of Hong Kong.
\section*{Acknowledgement}

\bibliographystyle{siam}
\bibliography{ls_transport}

\end{document}